\RequirePackage{rotating}
\documentclass[a4paper,twopage,reqno,12pt]{amsart}
\usepackage{eurosym}
\usepackage[top=30mm,right=30mm,bottom=30mm,left=30mm]{geometry}
\usepackage{tabularx}
\usepackage{graphicx}
\usepackage{amsmath}
\usepackage{amssymb}
\usepackage{amsfonts}
\usepackage{amsthm}
\usepackage{amstext}
\usepackage{amsbsy}
\usepackage{amsopn}
\usepackage{amscd}
\usepackage{enumerate}
\usepackage{color}
\usepackage[colorlinks]{hyperref}
\usepackage[hyperpageref]{backref}
\usepackage[notref,notcite]{showkeys}
\usepackage{multirow}
\usepackage{longtable}
\usepackage{float}
\usepackage{lscape}
\usepackage{pdflscape}
\usepackage{url}
\usepackage{adjustbox}
\usepackage{rotating}
\usepackage{etoolbox}
\newtheorem{theorem}{Theorem}[section]
\newtheorem{corollary}[theorem]{Corollary}
\newtheorem{lemma}[theorem]{Lemma}
\newtheorem{proposition}[theorem]{Proposition}
\newtheorem{claim}{Claim}
\AtEndEnvironment{proof}{\setcounter{claim}{0}}
\theoremstyle{definition}

\newtheorem{example}[theorem]{Example}
\newtheorem{remark}[theorem]{Remark}

\numberwithin{equation}{section}

\renewcommand{\leq}{\leqslant}
\renewcommand{\geq}{\geqslant}

\begin{document}

\title[A Classification of the flag-transitive $2$-designs with $\lambda =2$]{A Classification of the flag-transitive $2$-$(v,k,2)$ designs }
\author[]{Hongxue Liang and Alessandro Montinaro}
\address{Hongxue Liang, School o fMathematics and Big Data, Foshan University, Foshan 528000, P.R. China}
\email{hongxueliang@fosu.edu.cn}
\address{Alessandro Montinaro, Dipartimento di Matematica e Fisica “E. De Giorgi”, University of Salento, Lecce, Italy}
\email{alessandro.montinaro@unisalento.it}
\subjclass[MSC 2020:]{05B05; 05B25; 20B25}
\date{\today }

\begin{abstract}
In this paper, we provide a complete classification of $2$-$(v,k,2)$ design admitting a flag-transitive automorphism group of affine type with the only exception of the semilinear $1$-dimensional group. Alongside this analysis we provide a construction of seven new families of such flag-transitive $2$-designs, two of them infinite, and some of them involve remarkable objects such as $t$-spreads, translation planes, quadrics and Segre varieties.

Our result together with those Alavi et al. \cite{Alavi,ABDT}, Praeger et al. \cite{DLPX}, Zhou and the first author \cite{LT,LT1} provides a complete classification of $2$-$(v,k,2)$ design admitting a flag-transitive automorphism group with the only exception of the semilinear $1$-dimensional case.
\end{abstract}

\maketitle

\section{Introduction and main result}\label{IandR}
A $2$-$(v,k,\lambda )$ \emph{design} $\mathcal{D}$ is a pair $(\mathcal{P},%
\mathcal{B})$ with a set $\mathcal{P}$ of $v$ points and a set $\mathcal{B}$
of blocks such that each block is a $k$-subset of $\mathcal{P}$ and each two
distinct points are contained in $\lambda $ blocks. We say $\mathcal{D}$ is 
\emph{non-trivial} if $2<k<v$, and symmetric if $v=b$. All $2$-$(v,k,\lambda
)$ designs in this paper are assumed to be non-trivial. An automorphism of $%
\mathcal{D}$ is a permutation of the point set which preserves the block
set. The set of all automorphisms of $\mathcal{D}$ with the composition of
permutations forms a group, denoted by $\mathrm{Aut(\mathcal{D})}$. For a
subgroup $G$ of $\mathrm{Aut(\mathcal{D})}$, $G$ is said to be \emph{%
point-primitive} if $G$ acts primitively on $\mathcal{P}$, and said to be 
\emph{point-imprimitive} otherwise. In this setting, we also say that $%
\mathcal{D}$ is either \emph{point-primitive} or \emph{point-imprimitive}, respectively. A \emph{flag} of $\mathcal{D}$ is a pair $(x,B)$ where $x$ is
a point and $B$ is a block containing $x$. If $G\leq \mathrm{Aut(\mathcal{D})%
}$ acts transitively on the set of flags of $\mathcal{D}$, then we say that $%
G$ is \emph{flag-transitive} and that $\mathcal{D}$ is a \emph{%
flag-transitive design}.

The $2$-$(v,k,\lambda )$ designs $\mathcal{D}$ admitting a flag-transitive
automorphism group $G$ have been widely studied by several authors. If $\lambda=1$, that is when $\mathcal{D}$ is a linear space, then $G$ acts point-primitively on $\mathcal{D}$ by an important results due to Higman and McLaughlin \cite{HM} dating back to 1961. In 1990, Buekenhout, Delandtsheer, Doyen, Kleidman, Liebeck and Saxl \cite{BDDKLS} obtained a classification of $2$-designs with $\lambda =1$ except when $v$ is a power of a prime and $G \leq A \Gamma L_{1}(v)$. If $\lambda>1$, it is no longer true that $G$ acts point-primitively on $\mathcal{D}$ as shown by Davies in \cite{Da}. In this contest, a special attention is given to the case $\lambda=2$. In a series a paper, Regueiro \cite{ORR,ORR1,ORR2,ORR3} proved that, if $\mathcal{D}$ is symmetric, then either $(v,k)=(7,4),(11,5),(16,6)$, or $v$ is a power of an odd prime and $G \leq A \Gamma L_{1}(v)$. In 2016, Zhou and the first author \cite{LT} proved that, if $\mathcal{D}$ is non-symmetric and $G$ is point-primitive, then $G$ is an affine or an almost simple group. In each of these cases $G$ has a unique minimal normal subgroup, its socle $Soc(G)$, which is an elementary abelian group or a non-abelian simple group, respectively. Further, in the same paper and in \cite{LT1}, under the assumption of the point-primitivity of $G$ on $\mathcal{D}$, they classify $\mathcal{D}$ when $Soc(G)$ is either sporadic or an altenating group, respectively. In 2020, Devillers, Liang, Praeger and Xia \cite{DLPX} showed that, if $\mathcal{D}$ is non-symmetric, then $G$ is point-primitively on $\mathcal{D}$, and hence $G$ is affine or almost simple. Moreover, they classified $\mathcal{D}$ when $Soc(G) \cong PSL_{n}(q)$ for $n \geq 3$. The completion of the classification of $\mathcal{D}$ admitting a flag-transitive a almost simple automorphim group $G$ has been recently achieved by Alavi in \cite{Alavi}, and Alavi, Bayat, Daneshkhah and Tadbirinia in \cite{ABDT}. 

The objective of this paper is to investigate the non-trivial flag-transitive $2$-designs with $\lambda=2$ admitting a flag-transitive group $G$, which is the remaining case towards the classification of such $2$-designs. More precisely, we classify the non-trivial $2$-$(v,k,2)$ designs admitting a flag-transitive automorphism group $G$ of affine type except when $v$ is a power of a prime and $G \leq A \Gamma L_{1}(v)$. 

All the previously mentioned results for $\lambda=2$ together with the one obtained in this paper are summarized in the following theorem, which provides a complete classification of the non-trivial $2$-$(v,k,2)$ designs admitting a flag-transitive automorphism group $G$ except when $G$ is affine semilinear $1$-dimensional group.

\bigskip

\begin{theorem}\label{main}
Let $\mathcal{D}$ be a non-trivial $2$-$(v,k,2)$ design admitting a flag-transitive automorphism group $G$. Then one of the following holds:
\begin{enumerate}
 \item $v=p^{d}$, $p$ prime, $d \geq 1$, and $G \leq A \Gamma L_{1}(p^{d})$;
  \item $\mathcal{D}$ is one the two symmetric $2$-$(16,6,2)$ designs described in \cite[Section 1.2.1]{ORR} and $G$ is a subgroup of $(Z_{2})^4.S_{6}$ or $(Z_{2}\times Z_{8}).S_{4}.Z_{2}$, respectively.
  \item $\mathcal{D}$ is the $2$-$((3^{n}-1)/2,3,2)$ design as in \cite[Construction 1.1]{DLPX} and $Soc(G) \cong PSL_{n}(3)$, $n \geq 3$;
  \item $(\mathcal{D},Soc(G))$ is as in \cite[Table 1]{Alavi};
  \item $(\mathcal{D},G)$ is as in \cite[Example 7]{MBF}, \cite[Example 2.1]{Mo}, or as in Examples \ref{sem1dimjedan}--\ref{hyperbolic}. 
\end{enumerate}   
\end{theorem}

\bigskip

It is worth noting the following facts:
\begin{itemize}
\item[-] There several examples of flag-transitive point-primitive $2$-$(v,k,2)$-designs corresponding to case (1):
\begin{itemize}
    \item[(i)] The $2$-$(11,5,2)$ Paley biplane;
    \item[(ii)] Only one of the four known $2$-$(37,9,2)$ biplanes (see \cite[Section 1.2.1]{ORR});
    \item[(iii)]  The $2$-$(64,7,2)$ design constructed in \cite[Proposition 18]{BMR}; 
    \item[(iv)] Two infinite families of $2$-designs as in Example \ref{sem1dimjedan};
    \item[(v)] The $2$-$(81,6,2)$ design as in Example \ref{sem1dimdva}.    
\end{itemize}
Further, when $\mathcal{D}$ is asymmetric additional constraints on $(\mathcal{D},G)$ are provided in \cite[Theorems 1 and 2]{BM} when $k$ is odd.
\item[-] In all cases of Theorem \ref{main}, except for (2), $G$ acts point-primitively on $\mathcal{D}$ by \cite[Theorem 1.1]{DLPX}. In case (2), if $\mathcal{D}$ is the symmetric $2$-$(16,6,2)$ design arising from a difference set in $Z_{2} \times Z_{8}$, then $Aut(\mathcal{D}) \cong (Z_{2}\times Z_{8}).S_{4}.Z_{2}$ acts point-imprimitively on $\mathcal{D}$, whereas if  $\mathcal{D}$ arises from a difference set in $(Z_{2})^4$, then $G \leq (Z_{2})^4.S_{4}$ acts point-imprimitively on $\mathcal{D}$, and $(Z_{2})^4.A_{5} \leq G \leq (Z_{2})^4.S_{6}$ acts point-primitively on $\mathcal{D}$.  

\end{itemize}

\subsection{Structure of the paper and outline of the proof.} The paper consists of six sections. Section \ref{IandR} consists of the introduction and the statement of Theorem \ref{main}. In Section \ref{prel}, we provide some preliminary constraints on the pair $(\mathcal{D},G)$. In particular, Theorem \ref{redtheo}, which summarizes the results obtained in \cite{ORR,LT,LT1,DLPX,Alavi,ABDT}, shows that either (1) or (2a) of Theorem \ref{main} occurs, or $Soc(G)$ is an elementary abelian $p$-group for some prime $p$ acting point-regularly on $\mathcal{D}$. Moreover, some constraints on the size of the replication number $r$ of $\mathcal{D}$, on $k$ and the structure of the stabilizer in $G$ of a block of $\mathcal{D}$ are given. In section \ref{Exs}, six new families of examples, two of them infinite, are constructed. Finally, Sections \ref{Asch}--\ref{quasi} are devoted to the proof of case (2) of Theorem \ref{main} depending on which class of the Aschacher's theorem a maximal subgroup of $\Gamma L(V)$ containing the stabilizer in $G$ of the zero vector of $V$ belongs to.

The main idea behind the proof of the affine case of Theorem \ref{main} is the combination of the Aschbacher's theorem on the structure of the subgroups of the classical groups \cite{As} with an adaptation of the techniques for the analysis of the flag-transitive linear spaces developed by Liebeck in \cite{Lieb}. A key role in carrying out the study of the case affine case is certainly played by the geometry of the classical groups. As we will see in different parts of the paper, the fact that some candidate automorphism groups of $2$-designs with $\lambda=2$ preserve remarkable geometric structures such as an hermitian unital, a quadric, a $t$-spread, a translation plane, or a Segre variety is determinant either in constructing examples, or in showing that certain pairs $(v,k)$ cannot be parameters of a flag-transitive $2$-design.   
\section{Preliminaries}\label{prel}

Our starting point is the following result which provides a strong reduction to our investigation problem.

\begin{theorem}\label{redtheo}
Let $\mathcal{D}$ be a nontrivial $2$-$(v,k,2)$ design admitting a flag-transitive automorphism group $G$. Then one of the following holds:
\begin{enumerate}  
    \item $\mathcal{D}$ is one the two symmetric $2$-$(16,6,2)$ designs described in \cite[Section 1.2.1]{ORR} and $G$ is a subgroup of $(Z_{2})^4.S_{4}$ or $(Z_{2}\times Z_{8}).S_{4}.Z_{2}$, respectively.
\item $\mathcal{D}$ is non-symmetric, $G$ acts point-primitively on $\mathcal{D}$ and one of the following holds:
\begin{enumerate}
    \item $Soc(G)$ is a non-abelian simple group and one of the following holds:
  \begin{enumerate}
  \item $\mathcal{D}$ is the $2$-$((3^{n}-1)/2,3,2)$ design as in \cite[Construction 1.1]{DLPX} and $Soc(G) \cong PSL_{n}(3)$, $n \geq 3$;
  \item $(\mathcal{D},Soc(G))$ are as in \cite[Table 1]{Alavi};
  \end{enumerate}
  \item $Soc(G)$ is an elementary abelian $p$-group for some prime $p$ acting point-regularly on $\mathcal{D}$.
  \end{enumerate}
\end{enumerate}   
\end{theorem}

\begin{proof}
Let $\mathcal{D}$ be a nontrivial $2$-$(v,k,2)$ design admitting a flag-transitive automorphism group $G$. By \cite[Theorem 1.1]{DLPX} either (1) holds, or $G$ acts point-primitively on $\mathcal{D}$, and $Soc(G)$ is either a non-abelian simple group or an elementary abelian $p$-group for some prime $p$ (see also \cite[Theorem 3]{ORR} \cite[Theorem 1.1]{LT}). The latter case is (2.b). If $Soc(G)$ is a non-abelian simple group, then (2.a.i) and (2.a.ii) follow from which is the summary of the results \cite[Theorem 1.2]{DLPX} and \cite{ABDT}, \cite[Corollary 1.2]{Alavi}, \cite[Theorem 1.1]{LT1} and \cite[Theorem 1.2]{LT} according as $Soc(G)$ is classical, exceptional of Lie type, sporadic, or alternating, respectively.
\end{proof}

\bigskip

In order to prove Theorem \ref{main}, we only need to investigate case (2b) of Theorem \ref{redtheo}. Hence, let $\mathcal{D}=(\mathcal{P},\mathcal{B})$ be a non-trivial $2$-$(v,k,2)$ design admitting a flag-transitive point-primitive automorphism group such that $Soc(G)$ is an elementary abelian $p$-group for some prime $p$. Set $T=Soc(G)$, then $%
\mathcal{P}$ can be identified with the $d$-dimensional $GF(p)$-vector space $V$ in a way that $T$ is the translation group of $V$ and $G=T:G_{0}\leq AGL(V)$ since $G$ acts point-primitively on $\mathcal{D}$. Hence, $v=\left\vert T \right\vert=p^{d}$. Further, $G_{0}$ acts irreducibly on $V$.

\bigskip
Next lemma shows that, except for the $2$-$(11,5,2)$ Paley biplane, there is no overlapping between $2$-designs corresponding to cases (2a) and (2b) of Theorem \ref{redtheo}. Further, it will play a central role in determining the full automorphism group of the examples constructed in this paper.
\bigskip

\begin{lemma}\label{fullaut}
 If $\mathcal{D}$ is $2$-$(p^{d},k,2)$ admitting flag-transitive automorphism group $G$ of affine type, then $Aut(\mathcal{D})$ is of affine type except when $\mathcal{D}$ is the symmetric $2$-$(11,5,2)$ design, $Aut(\mathcal{D}) \cong PSL_{2}(11)$ and $G \cong F_{55}$.
\end{lemma}

\begin{proof}
Let $A=Aut(\mathcal{D})$. Then either $Soc(A)$ is non-abelian simple or $Soc(A)=T$ by \cite[Theorem 1.1]{LT} since $G\leq A$. In the former case $Soc(A)$ is one of the groups listed \cite[Theorem 1]{Gu} since $Soc(A)$ acts point-transitively on $\mathcal{D}$. The case where $Soc(A)$ isomorphic to one of the groups $A_{p^{n}}$, $M_{11}$ or $M_{23}$ cannot occur by \cite[Theorem 1.2]{LT} or \cite[Theorem 1.1]{LT1}, respectively. Further, $Soc(A)\cong PSp_{4}(3)$ is ruled out by \cite[Corollary 1.3]{Alavi}, and the same result implies that $\mathcal{D}$ is the symmetric $2$-$(11,5,2)$ design and $A \cong PSL_{2}(11)$ when $Soc(A)\cong PSL_{2}(s)$. In this case, one has $G \cong F_{55}$.

Finally, assume that $Soc(A)\cong PSL_{m}(s)$ with $m \geq 3$. Then either $\mathcal{D}$ is the complement of $PG_{2}(2)$, or $s=3$ and $\mathcal{D}$ is the $2$-design whose point set is that of $PG_{m-1}(3)$ and whose block set consists of the elements $\ell\setminus{x}$ with $(x,\ell)$ any incident point-line pair of $PG_{m-1}(3)$ by \cite[Theorem 1.2]{DLPX}. Thus $G$ acts point-transitively and line-transitively on $PG_{m-1}(3)$ since $G$ acts flag-transitively on $\mathcal{D}$, hence $m=3$ by \cite[Theorem 8.1]{BP} since $G$ is of affine type. So $G \leq PGL_{3}(3)$ with $\left\vert G\right\vert=144$ since $G$ acts flag-transitively on $\mathcal{D}$, which is not the case by \cite{At}. Therefore $\mathcal{D}$ is the complement of $PG_{2}(2)$, and hence $G \leq PGL_{3}(3)$ with $\left\vert G\right\vert=28$ since $G$ acts flag-transitively on $\mathcal{D}$, again a contradiction. Thus $Soc(A)=T$, and hence $A$ is of affine type.     
\end{proof}

\bigskip

\begin{lemma}\label{rodd}
If $r$ is odd either $G_{0}\leq \Gamma L_{1}(p^{d})$, or $\mathcal{D}$ is the $2$-$(64,7,2)$ design constructed in \cite[Example 7]{MBF} and $G_{0} \cong S_{3} \times PSL_{2}(7)$.    
\end{lemma} 
\begin{proof}
If $r$ is odd, then $(r,\lambda)=1$ since $\lambda=2$, and hence the assertion follows from \cite{BM,BMR,MBF} according to whether $\mathcal{D}$ is symmetric, $\mathcal{D}$ is non-symmetric and $G_{0}$ is soluble, or $\mathcal{D}$ is non-symmetric and $G_{0}$ is non-soluble, respectively.      
\end{proof}

\bigskip
\subsection{Hypothesis.}\label{Hippo} On the basis of Lemma \ref{rodd}, in the sequel we assume that $r$ is even.
\bigskip

\begin{lemma}
\label{sudbina}The following hold:

\begin{enumerate}
\item $r>\sqrt{2}p^{d/2}$;

\item $\frac{r}{2}\mid (c_{1},...,c_{m},p^{d}-1)$, where $c_{i}$ are the $G_{0}$%
-orbits on $V^{\ast }$.
\end{enumerate}
\end{lemma}

\begin{proof}
Since $r$ is even, $r \geq k>2$ and $\lambda =2$, it follows that $r/2=\frac{p^{d}-1%
}{k-1} \geq \frac{p^{d}-1}{r-1}$. Therefore $r(r-1)/2>p^{d}$, and hence $%
r>p^{d/2}\sqrt{2}$, which is (1). Finally, (2) follows from the fact that
for any $x\in V^{\ast }$ and $B$ block containing $0$, the incidence
structure $(x^{G_{0}},B^{G_{0}})$ is a $1$-design by \cite[1.2.6]%
{Demb}.
\end{proof}

\begin{lemma}
\label{cici} Set $p^{t}=\left\vert T_{B}\right\vert $, then $G_{B}/T_{B}$ is
isomorphic to a subgroup of $G_{0}$ of index $\frac{2 p^{t}\left(
p^{d}-1\right)}{k(k-1)}$ containing an isomorphic copy of $%
G_{0,B}$.
\end{lemma}

\begin{proof}
Let $B$ be any block of $\mathcal{D}$ incident with $0$. Then $B$ is split
into $T_{B}$-orbits of equal length $p^{t}$, $t\geq 0$, permuted
transitively by $G_{B}$ since $T_{B}\trianglelefteq G_{B}$. Thus $\left\vert
G_{B}:G_{0,B}T_{B}\right\vert =\frac{k}{p^{t}}$, and hence $G_{B}/T_{B}$ is
isomorphic to a subgroup $J$ of $G_{0}$ such that $\left\vert J\right\vert =%
\frac{k}{p^{t}}\left\vert G_{0,B}\right\vert $ containing an isomorphic
copy of $G_{0,B}$ since $G_{0,B} \cap T_{B}=1$. Thus, $r=\left \vert G_{0}:J\right\vert \frac{k}{p^{t}}$ implies $%
\left \vert G_{0}:J\right \vert =\frac{2 p^{t}\left( p^{d}-1\right)}{k(k-1)}$.
\end{proof}

\begin{corollary}\label{p2}
One of the following holds:
\begin{enumerate}
    \item $T$ acts block-semiregularly on $\mathcal{D}$, and $k \mid r$;
    \item $T$ does not act block-semiregularly on $\mathcal{D}$, and one of the following holds:
    \begin{enumerate}
        \item $k=p^{t}$ with $t>1$, and the blocks of $\mathcal{D}$ are subspaces of $AG_{d}(p)$;
        \item $k=2p^{t}$ with $p$ odd and $t \geq 1$.
    \end{enumerate}
\end{enumerate}
\end{corollary}

\begin{proof}
If $T$ acts block-semiregularly on $\mathcal{D}$, then $k \mid r$ by Lemma \ref{cici}, and we obtain (1). Hence, let $B
$ be any block of $\mathcal{D}$ containing $0$ and assume that $T_{B} \neq 1$. Set $p^{t}=\left\vert T_{B}\right\vert$. Then $B$ is union of $%
T_{B}$-orbits of length $p^{t}$ with $t\geq 1$, since $T_{B}$ acts point-semiregularly on $\mathcal{D}$. Thus $p^{t} \mid k$.

If $k=p^{t}$, then the blocks of $\mathcal{D}$ are subspaces of $AG_{d}(p)$. Further $t>1$, otherwise $\lambda=1$. Thus, we obtain (2a).

If $p^{t}<k$, then $%
B=\bigcup_{i=0}^{k/p^{t}-1}\left(W +y_{i}\right) $,
where $W$ is the $GF(p)$-subspace $0^{T_{B}}$ of $V$, and $y_{i}$ with $i=0,...,k/p^{t}-1$
are representatives of distinct additive cosets of $W$ in $V$ such that $y_{0}=0$. Then $W \subseteq B^{\tau
_{-y_{i}}}$ for each $i=0,...,k/p^{t}-1$, and hence $W \subseteq \bigcap_{i=0}^{k/p^{t}-1}B^{\tau
_{-y_{i}}}$. Note that, $B^{\tau _{-y_{j}}}\neq B^{\tau
_{-y_{h}}}$ for $j\neq h$ since $y_{i}$ for $i=0,...,k/p^{t}-1$
are representatives of distinct additive cosets of $W$ in $V$, and hence $k/p^{t}\leq \lambda =2$, and hence $k=2p^{t}$ since $p^{t}<k$.

Assume that $p=2$ and $k=2^{t+1}$, then $B=W\cup (W+y_{1})=W\oplus \left\langle y_{1}\right\rangle $, which is a contradiction. Thus $p$ is odd, and we obtain (2b).
\end{proof}

\bigskip

\section{Examples}\label{Exs}
This section is devoted to the constructions of the flag-transitive $2$-$(v,k,2)$ designs with replication number $r$ even (the case with $r$ odd is estabilshed in Lemma \ref{rodd}). Moreover, as stated in Theorem \ref{main}, if $G_{0} \nleq \Gamma L_{1}(p^{d})$ and $\mathcal{D}$ is not the symmetric $2$-$(16,6,2)$ design, any flag-transitive design with $r$ even and $\lambda=2$ is either as in  \cite[Example 2.1]{Mo}, or it is isomorphic to one of the examples provided in this section.

\bigskip

\begin{example}\label{sem1dimjedan}
Let $V=V_{n}(q)$, where $q=p^{d/n}$, and let $G=T:G_{0}$, where $T$ is the translation group of $V$ and $G_{0}$ is any subgroup of $\Gamma L_{n}(q)$ containing a normal subgroup $X$ isomorphic to one of the groups $SL_{n}(q)$, $Sp_{n}(q)$ for $n$ even, $G_{2}(q)$ for $n=6$ and $q$ even, or $GL_{1}(q^{n})$.
Let $\omega $ be any primitive element of $GF(q)$ and $x$ be any nonzero vector of $V$, then the following hold: 
\begin{enumerate}
\item If $p^{d} \equiv 1 \pmod{3}$, then $\mathcal{D}=(V,B^{G})$ with $B=\bigl\{ \omega ^{\frac{(p^{d}-1)j}{3}}x:j=0,1,2\bigr\}$ is a $2$-$(p^{d},3,2)$ design admitting $G$ as a flag-transitive automorphism group. Further, $Aut(\mathcal{D}) \cong A\Gamma L_{n}(q)$;
 
\item If $t$ is a proper even divisor of $d/n$, $G_{0}=X:\left\langle \sigma^{t/2} \right\rangle$, where $\sigma :(y_{1},...,y_{n})\rightarrow (y^{p}_{1},...,y^{p}_{n})$ is the semilinear map induced by the Frobenius automorphism of $GF(q)$, then $\mathcal{D}=(V,B^{G})$ with $B=\left\langle x \right\rangle_{GF(p^t)}$ is a $2$-$(p^{d},p^{t},2)$ design admitting $G$ as a flag-transitive automorphism group. Further, $Aut(\mathcal{D}) \cong T:(GL_{n}(q):\left\langle \sigma^{t/2} \right\rangle)$.
\end{enumerate}
\end{example}
\begin{proof}
The incidence struture $\mathcal{D}=(V,B^{G})$, where $B$ and $G$ are as in (1) or (2). is a $2$-design since $G$ acts point-$2$-transitively on $\mathcal{D}$. Then $$b=\left\vert G:G_{B}\right \vert  =q^{n-1}\frac{q^{n}-1}{q-1}\left\vert H:H_{B}\right \vert \textit{,}$$ where $H$ is the group induced by $G_{\left\langle x \right\rangle}$ on $\left\langle x \right\rangle$, since $B \subset \left\langle x \right\rangle$. Again by the point-$2$-transitivity of $G$ on $\mathcal{D}$, we may assume that $x=(1,0,...,0)$.

Assume that $B$ and $G$ are as in (1). Let $\alpha:\omega^{i} x\rightarrow \omega^{i+1} x$ and $\beta:\omega^{i} x \rightarrow \omega^{ip} x$ for each $i=0,...,d/n-1$, then $\left\langle \alpha^{\frac{q^{n}-1}{3}} \right\rangle \leq H_{B}$. Actually, $\left\langle \alpha^{\frac{q^{n}-1}{3}} \right\rangle \unlhd H_{B}$ since $T:\left\langle \alpha \right\rangle$ is a Frobenius group and $T_{B}=1$. Thus, $H_{B}$ fixes the zero vector, and hence $H_{B}=\left\langle \alpha^{\frac{q^{n}-1}{3}}, \beta \right\rangle$ since $\left\langle \alpha^{\frac{q^{n}-1}{3}} \right\rangle \unlhd \left\langle \alpha, \beta \right\rangle \cong \Gamma L_{1}(q^{n})$, $B$ is a $\left\langle \alpha^{\frac{q^{n}-1}{3}} \right\rangle$-orbit and $\beta$ fixes $x$. Therefore, $b=q^{n-1}\frac{q^{n}-1}{q-1}q\frac{q-1}{3}=q^{n}\frac{q^{n}-1}{3}$, and hence the parameters of $\mathcal{D}$ are as in (1) since $\left\vert B \right \vert=3$. The previous argument also shows that $A\Gamma L_{n}(q) \leq Aut(\mathcal{D})$, and hence $Aut(\mathcal{D}) \cong A\Gamma L_{n}(q)$ by Lemma \ref{fullaut}.

Assume that $B$ and $G$ are as in (2). Then $H_{0}=\left\langle \alpha,\beta ^{t/2} \right\rangle$ and $H_{B}=T_{B}:H_{0,B}$ with $H_{0,B}=\left\langle \alpha^{\frac{p^{d}-1}{p^{t}-1}},\beta ^{t} \right\rangle$, where $\alpha$ and $\beta$ are defined above, since $\sigma^{t/2}$ induces $\beta^{t/2}$ on $\left\langle x \right\rangle$. Thus $\left\vert H:H_{B}\right \vert=2\frac{q}{p^t}\frac{q-1}{p^{t}-1}$. Therefore $b=2\frac{q^{n}}{p^{t}}\frac{q^{n}-1}{p^{t}-1}$, and hence the parameters of $\mathcal{D}$ are as in (2) since $\left\vert B \right \vert=p^{t}$. Further, $T:(GL_{n}(q):\left\langle \sigma^{t/2} \right\rangle) \leq Aut(\mathcal{D})$. On the other hand, $Aut(\mathcal{D}) \leq A\Gamma L_{n}(q)$ by Lemma \ref{fullaut}. Hence, $Aut(\mathcal{D}) \cong T:(GL_{n}(q):\left\langle \sigma^{t/2} \right\rangle)$ by the definition of $B$.
\end{proof}

\bigskip
\begin{remark}\label{osserv}
Then $2$-$(2^{4},2^{2},2)$ designs as in Example \ref{sem1dimjedan}(2) for $d=4$, $t=n=d/n=2$ and as in \cite[Example 2.1]{Mo} are not isomorphic. Indeed, in the first case any block containing $0$ is $1$-dimensional subspace of $V=V_{2}(4)$, whereas in the latter it is a $GF(2)$-span of a basis of $V$.   
\end{remark}

\bigskip

\begin{example}\label{various}
Let $V=V_{2}(p)$, $\omega$ be a primitive element of $GF(p)$, and let $G=T:G_{0}$ and $\mathcal{D}=(V,B^{G})$ with $B$ and $G_{0}$ as in Table \ref{smol}. Then $\mathcal{D}$ is a $2$-designs with parameters are as in Table \ref{smol} admitting $G$ as a flag-transitive automorphism group. Further, $Aut(\mathcal{D})=T:G_{0}$ with $G_{0}$ as in Table \ref{smol} for each such $2$-designs.

 \begin{sidewaystable} 
\caption{Examples}
\label{smol}
\tiny
\begin{tabular}{llllcc}
Line & $p$ & $(v,k,r,b,\lambda )$ & Base Block & $G_{0}$ & Structure description of 
$G_{0}$ \\
\hline \hline
1& $5$ & $(5^{2},4,16,100,2)$ & $\left\{ (0,0),(0,1),(\omega ,\omega
^{3}),(\omega ^{3},\omega ^{3})\right\} $ & $\left\langle \left( 
\begin{array}{cc}
\omega  & 0 \\ 
0 & 1%
\end{array}%
\right) ,\left( 
\begin{array}{cc}
0 & 1 \\ 
1 & 0%
\end{array}%
\right) \right\rangle $ & $\left( Z_{4}\times Z_{4}\right) :Z_{2}$ \\ 
\hline
2 & $7$ & $(7^{2},3,48,\allowbreak 784,2)$ & $\left\{ (0,0),(0,1),(1,\omega
)\right\} $ & $\left\langle \left( 
\begin{array}{cc}
1 & \omega ^{5} \\ 
\omega ^{5} & \omega ^{3}%
\end{array}%
\right) ,\left( 
\begin{array}{cc}
\omega ^{3} & \omega ^{5} \\ 
\omega ^{4} & 0%
\end{array}%
\right) \right\rangle $ & $Z_{3}\times Z_{2}.S_{4}^{-}$ \\
3 & &  &  & $\left\langle \left( 
\begin{array}{cc}
0 & \omega  \\ 
\omega ^{2} & \omega ^{3}%
\end{array}%
\right) ,\left( 
\begin{array}{cc}
\omega ^{3} & \omega ^{3} \\ 
\omega ^{2} & 1%
\end{array}%
\right) \right\rangle $ & $Z_{2}.S_{4}^{-}$ \\
4& &  &  & $\left\langle \left( 
\begin{array}{cc}
\omega  & \omega ^{5} \\ 
\omega ^{5} & \omega ^{4}%
\end{array}%
\right) ,\left( 
\begin{array}{cc}
1 & \omega  \\ 
\omega  & \omega ^{3}%
\end{array}%
\right) \right\rangle $ & $Z_{3}\times Q_{16}$ \\
\hline
5 & $11$ & $(11^{2},3,120,4840,2)$ & $\left\{ (0,0),(0,1),(\omega ^{3},\omega
^{4})\right\} $ & $\left\langle \left( 
\begin{array}{cc}
\omega ^{8} & \omega ^{7} \\ 
\omega ^{9} & \omega ^{7}%
\end{array}%
\right) ,\left( 
\begin{array}{cc}
0 & \omega ^{2} \\ 
\omega ^{2} & \omega ^{5}%
\end{array}%
\right) \right\rangle $ & $Z_{5}\times GL_{2}(3)$ \\ 
6 & &  & $\left\{ (0,0),(0,1),(\omega ^{4},\omega ^{2})\right\} $ &  &  \\ 
7 & &  & $\left\{ (0,0),(0,1),(\omega ^{2},\omega ^{2})\right\} $ & $%
\left\langle \left( 
\begin{array}{cc}
\omega ^{4} & \omega ^{4} \\ 
1 & \omega 
\end{array}%
\right) ,\left( 
\begin{array}{cc}
\omega ^{8} & 1 \\ 
1 & \omega ^{3}%
\end{array}%
\right) \right\rangle $ & $Z_{5}\times SL_{2}(3)$ \\
\hline
8 & $11$ & $(11^{2},4,80,2420,2)$ & $\left\{ (0,0),(0,1),(\omega ^{4},\omega
),(\omega ^{9},\omega ^{5})\right\} $ & $\left\langle \left( 
\begin{array}{cc}
1 & 0 \\ 
\omega ^{3} & \omega ^{5}%
\end{array}%
\right) ,\left( 
\begin{array}{cc}
0 & \omega ^{2} \\ 
\omega ^{7} & 0%
\end{array}%
\right) \right\rangle $ & $Z_{5}\times SD_{16}$ \\ 
\hline
9& $19$ & $(19^{2},6,144,8664,2)$ & $\left\{ (0,0),(0,1),(\omega ^{4},\omega
^{5}),(\omega ^{4},\omega ^{14}),(\omega ^{7},\omega ),(\omega ^{7},\omega
^{17})\right\} $ & $\left\langle \left( 
\begin{array}{cc}
\omega ^{11} & \omega ^{13} \\ 
\omega  & \omega ^{13}%
\end{array}%
\right) ,\left( 
\begin{array}{cc}
0 & \omega ^{2} \\ 
\omega ^{2} & \omega ^{7}%
\end{array}%
\right) \right\rangle $ & $Z_{9}\times GL_{2}(3)$ \\ 
10& &  &  & $\left\langle \left( 
\begin{array}{cc}
1 & 0 \\ 
\omega ^{5} & \omega ^{9}%
\end{array}%
\right) ,\left( 
\begin{array}{cc}
0 & \omega ^{2} \\ 
\omega ^{11} & 0%
\end{array}%
\right) \right\rangle $ & $Z_{9}\times SD_{16}$ \\ 
11& &  & $\left\{ (0,0),(0,1),(\omega ^{5},\omega ^{11}),(\omega ^{5},\omega
^{13}),(\omega ^{8},\omega ^{10}),(\omega ^{8},\omega ^{14})\right\} $ & $%
\left\langle \left( 
\begin{array}{cc}
\omega ^{11} & \omega ^{13} \\ 
\omega  & \omega ^{13}%
\end{array}%
\right) ,\left( 
\begin{array}{cc}
0 & \omega ^{2} \\ 
\omega ^{2} & \omega ^{7}%
\end{array}%
\right) \right\rangle $ & $Z_{9}\times GL_{2}(3)$ \\ 
12& &  &  & $\left\langle \left( 
\begin{array}{cc}
1 & 0 \\ 
\omega ^{5} & \omega ^{9}%
\end{array}%
\right) ,\left( 
\begin{array}{cc}
0 & \omega ^{2} \\ 
\omega ^{11} & 0%
\end{array}%
\right) \right\rangle $ & $Z_{9}\times SD_{16}$ \\ 
13& &  & $\left\{ (0,0),(0,1),(1,\omega ^{12}),(1,\omega ^{13}),(\omega
^{15},\omega ^{9}),(\omega ^{15},\omega ^{14})\right\} $ & $\left\langle
\left( 
\begin{array}{cc}
1 & 0 \\ 
\omega ^{5} & \omega ^{9}%
\end{array}%
\right) ,\left( 
\begin{array}{cc}
0 & \omega ^{2} \\ 
\omega ^{11} & 0%
\end{array}%
\right) \right\rangle $ & $Z_{9}\times SD_{16}$ \\ 
14& &  & $\left\{ (0,0),(0,1),(\omega ,\omega ^{10}),(\omega ,\omega
^{17}),(\omega ^{4},\omega ^{2}),(\omega ^{4},\omega ^{11})\right\} $ &  & 
\\ 
15& &  & $\left\{ (0,0),(0,1),(\omega ^{2},\omega ^{2}),(\omega ^{2},\omega
^{12}),(\omega ^{5},\omega ^{5}),(\omega ^{5},\omega ^{16})\right\} $ &  & 
\\ 
16 & &  & $\left\{ (0,0),(0,1),(\omega ^{3},\omega ^{9}),(\omega ^{3},\omega
^{12}),(\omega ^{6},\omega ^{2}),(\omega ^{6},\omega ^{15})\right\} $ &  & 
\\
\hline
17& $23$ & $(23^{2},3,528,93104,2)$ & $\left\{ (0,0),(0,1),(\omega ^{6},\omega
^{7})\right\} $ & $\left\langle \left( 
\begin{array}{cc}
\omega ^{6} & \omega ^{7} \\ 
\omega ^{8} & \omega ^{21}%
\end{array}%
\right) ,\left( 
\begin{array}{cc}
\omega  & \omega ^{12} \\ 
\omega  & \omega 
\end{array}%
\right) \right\rangle $ & $Z_{11}\times Z_{2}.S_{4}^{-}$ \\ 
18 & &  & $\left\{ (0,0),(0,1),(\omega ^{7},\omega ^{17})\right\} $ &  &  \\ 
19 & &  & $\left\{ (0,0),(0,1),(\omega ^{8},\omega ^{17})\right\} $ &  &  \\ 
20 & &  & $\left\{ (0,0),(0,1),(\omega ^{10},\omega ^{17})\right\} $ &  &  \\
\hline
\end{tabular}%
\end{sidewaystable}
\end{example}

\begin{proof}
The fact that $G_{0}$ and $\left\vert B^{G_{0}} \right \vert =r$ are as in Table \ref{smol} can be either checked by a direct computation, or just by using \texttt{GAP} \cite{GAP}. Note that $Z(GL_{2}(p)) \leq G_{0}$ in each case except that in Line 3. Thus the lengths of $G_{0}$-orbits on $V^{\ast}$ can precisely be determined by \cite[Lemmas 8 and 10]{COT}. Since for any orbit $x^{G_{0}}$ in $V^{\ast}$, the structure $(x^{G_{0}},B^{G_{0}})$ is $1$-design by \cite[1.2.6]{Demb}, and in each case $\frac{\left\vert x^{G_{0}}\right\vert}{r\left\vert B\cap x^{G_{0}}\right\vert}=2$ by using \texttt{GAP} \cite{GAP}, it follows that $\mathcal{D}$ is a $2$-design with the parameters as in Table \ref{smol}.

Let $A=Aut(\mathcal{D})$, where $\mathcal{D}$ is as in Table \ref{smol}. Then $A \leq AGL_{2}(p)$ by Lemma \ref{fullaut}, and hence $A_{0}=G_{0}$ with $G_{0}$ as in Table \ref{smol} again by \texttt{GAP} \cite{GAP}.
\end{proof}

\subsection{Segre varieties $\mathcal{S}_{n_{1}-1,n_{2}-1}$} Let $V_{n}(q)=V_{1} \otimes V_{2}$, where $dimV_{i}=n_{i}$, $i=1,2$.  Let $PG_{n_{i}-1}(q)$ be the projective space over $GF(q)$ corresponding to $V_{i}$. The set of all vectors in $V$ of the form $v_{1}\otimes v_{2}$ with $v_{i} \in V_{i}$ and $v_{i} \neq 0$ correspond to a set of points of $PG_{n-1}(q)$ known as a Segre variety $\mathcal{S}_{n_{1}-1,n_{2}-1}$ of $PG_{n_{1}-1}(q)$, $PG_{n_{2}-1}(q)$. More information on Segre varieties can be found in \cite[Section 25.5]{HT}.

\bigskip

\begin{example}\label{tens}
Let $V=V_{3}(2)\otimes V_{2}(2)$ and let $G=TG_{0}$ with $T$ the translation group of $V$ and $G_{0}=\left\langle
\left( \alpha ,1\right) \right\rangle \times \left\langle (1,\beta
),(1,\gamma )\right\rangle $, where $%
\left\langle \alpha \right\rangle $ is a Singer cycle of $V_{3}(2)$ and $%
\left\langle (1,\beta ),(1,\gamma )\right\rangle $ with $o(\beta )=3$ and $%
o(\gamma )=2$ is isomorphic to $S_{3}$. Now, let $\left\{ x_{0},x_{1},x_{2}\right\}$ and $\left\{ u_{1},u_{2}\right\} $ be any two basis of $V_{3}(2)$ and $V_{2}(2)$, respectively, such that $u_{1}$ is fixed by $\gamma$. Then $\mathcal{D%
}_{i}=(V,\pi _{i}^{G})$, $i=1,2$, where 
\begin{eqnarray*}
\pi_{1} &=&\left\langle   x_{0}\otimes u_{1},x_{1}\otimes u_{1}+x_{2}\otimes
u_{2}\right\rangle \\
\pi_{2} &=&\left\langle   x_{0}\otimes u_{1},x_{1}\otimes u_{1}+(x_{2}+x_{3})\otimes
u_{2}\right\rangle,
\end{eqnarray*}
are two non-isomorphic $2$-$(2^{6},2^{2},2)$ designs
admitting $G$ as automorphism group. Further, $G=Aut(\mathcal{D}_{i})$ and $G$ acts flag-transitively on $\mathcal{D%
}_{i}$, $i=1,2$. 
\end{example}

\begin{proof}
It easy to see that 
\begin{eqnarray*}
\mathcal{O}_{1} &=&\left\{ x_{1}\otimes u_{1},x_{1}\otimes
u_{2},x_{1}\otimes  (u_{1}+u_{2}) :x_{1}\in V_{3}(2)^{\ast
}\right\}  \\
\mathcal{O}_{2} &=&\left\{ x_{1}\otimes u_{1}+x_{2}\otimes
u_{2}:x_{1},x_{2}\in V_{3}(2)^{\ast },x_{1}\neq x_{2}\right\} 
\end{eqnarray*}%
are the two $G_{0}$-orbits of length $21$ and $42$ and partitioning $V^{\ast }$. They correspond to a Segre variety $\mathcal{S}_{2,1}$ of $PG_{5}(2)$ and its complementary set, respectively. Further, 
$G_{0,\pi _{i}}=1$, $\left\vert \pi_{i} \cap \mathcal{O}_{1}\right\vert =1$
and $\left\vert \pi_{i} \cap \mathcal{O}_{2}\right\vert =2$  for each $i=1,2$ since $\left\{ x_{0},x_{1},x_{2}\right\}$ is a basis of $V_{3}(2)$. Thus $\left\vert
\pi _{i}^{G_{0}}\right\vert =42$ for each $i=1,2$. For any nonzero vector of $V$ there
are exactly $2$ elements of $\pi _{i}^{G_{0}}$ containing it since $(\mathcal{O}%
_{j},\pi _{i}^{G_{0}})$ is a $1$-design for each $i,j=1,2$ by \cite[1.2.6]%
{Demb}. Therefore, $\mathcal{D%
}_{i}=(V,\pi _{i}^{G})$ with $i=1,2$ and $G=T:G_{0}$ are two $2$-$(2^{6},2^{2},2)$ designs
admitting $G$ as a flag-transitive automorphism group.

Let $A=Aut(\mathcal{D})$. Then $Soc(A)=T$ by Lemma \ref{fullaut}, and hence $A_{0} \leq GL_{6}(2)$.

If $A_{0}$ is transitive on $V^{\ast }$, then $A_{0}$ contains one of the group $%
SL_{6}(2)$, $Sp_{6}(2)$ or $G_{2}(2)^{\prime }$ as a normal subgroup by \cite%
[Appendix 1]{Lieb0}. However, $A_{0}$ does not have transitive permutation
representations of degree $45$ by \cite[Theorem 5.2.2]{KL} in the former case,
and by \cite{At} in the remaining ones. Thus $\mathcal{O}%
_{1}$ and $\mathcal{O}_{2}$ are both $A_{0}$-orbits and hence $A_{0}\leq
GL_{3}(2)\times GL_{2}(2)$ by \cite[Theorem]{Lieb0}. Easy computations show
that $A_{0,\pi _{i}}\leq
Z_{2}\times 1$. If $A_{0,\pi _{i}}=Z_{2}\times 1$ then $A_{0}=GL_{3}(2)%
\times GL_{2}(2)$ since $G_{0}\leq A_{0}$ and hence $\left\vert A_{0,\pi
_{i}}\right\vert =24$, a contradiction. Thus $A_{0,\pi _{i}}=1$ and hence $%
A_{0}=G_{0}$ and $A=G$.

Let $\varphi$ be an isomorphism from $\mathcal{D}_{1}$ onto $\mathcal{D}_{2}$. Then $\varphi \tau_{-0^{\varphi}}$ is isomorphism from $\mathcal{D}_{1}$ onto $\mathcal{D}_{2}$ fixing $0$. Therefore we may assume that $\varphi$ preserves $0$. Then $\varphi$ normalizes $A$ since $A=Aut(\mathcal{D}_{i})$, $i=1,2$. Thus $\varphi$ normalizes the elementary abelian $2$-group $Soc(A)$ and hence $\varphi \in N_{GL_{6}(2)}(A_{0})$.

If $N_{GL_{6}(2)}(A_{0})$ acts transitively on $V^{\ast }$, then $A_{0}$ contains one of the group $%
SL_{6}(2)$, $Sp_{6}(2)$ or $G_{2}(2)^{\prime }$ but none of these has $A_{0}$ as a normal subgroup. Thus $\mathcal{O}%
_{1}$ and $\mathcal{O}_{2}$ are both $N_{GL_{6}(2)}(A_{0})$-orbits and hence $N_{GL_{6}(2)}(A_{0})\leq
GL_{3}(2)\times GL_{2}(2)$ by \cite[Theorem]{Lieb0}. Actually, $N_{GL_{6}(2)}(A_{0}) =
F_{21}\times GL_{2}(2)$. Hence $\varphi$ is either $(\delta ,1 )$ or $(\delta ,\gamma )$ with $\delta \in F_{21}$. 
We may assume that $\varphi$ fixes $ x_{0}\otimes u_{1}$ since 
$A_{0}$ acts transitively $\mathcal{O}_{1}$. Hence $o(\delta )=3$. Now, either 
\[
\pi _{1}^{\varphi }=\left\langle  x_{0}\otimes u_{1},x_{1}^{\delta }\otimes
u_{1}+x_{2}^{\delta }\otimes u_{2}\right\rangle \textit{ or } \pi _{1}^{\varphi }=\left\langle  x_{0}\otimes u_{1},(x_{1}+x_{2})^{\delta }\otimes
u_{1}+x_{2}^{\delta }\otimes u_{2}\right\rangle \textit{.}
\]
according to whether $\varphi$ is $(\delta ,1 )$ or $(\delta ,\gamma )$ respectively. In each case either $\pi _{1}^{\varphi }=\pi _{2}$ or $\pi _{1}^{\varphi }=\pi
_{2}^{(1,\gamma )}$, where
$$
\pi _{2}^{(1,\gamma )} =\left\langle  x_{0}\otimes u_{1},\left(
 x_{0}+x_{1}+x_{2}+e\right) \otimes u_{1}+( x_{0}+x_{2})\otimes u_{2}\right\rangle
$$
since the $\pi _{2},\pi
_{2}^{(1,\gamma )}$ are the blocks of $\mathcal{
D}_{2}$ containing both $0$ and $x_{0}\otimes u_{1}$.

If $\varphi=(\delta ,1 )$ then $\left( x_{1}^{\delta
},x_{2}^{\delta }\right) =(x_{1}, x_{0}+x_{2})$ or $( x_{0}+x_{1}, x_{0}+x_{2})$ for $\pi _{1}^{\varphi }=\pi _{2}$, or
 $\left( x_{1}^{\delta },x_{2}^{\delta }\right)
=( x_{0}+x_{1}+x_{2}, x_{0}+x_{2})$ or $(x_{1}+x_{2}, x_{0}+x_{2})$ for $\pi _{1}^{\varphi }=\pi
_{2}^{(1,\gamma )}$. However, both cases are ruled out since $\delta $ induces a collineation of $PG_{2}(2)$ fixing $ x_{0}$
an permuting transitively the lines $\left\{  x_{0},x_{1}, x_{0}+x_{1}\right\} $, $%
\left\{  x_{0},x_{2}, x_{0}+x_{2}\right\} $ and $\left\{
 x_{0},x_{1}+x_{2}, x_{0}+x_{1}+x_{2}\right\} $.
 
If $\varphi=(\delta ,\gamma )$ then $\left( (x_{1}+x_{2})^{\delta },x_{2}^{\delta }\right) =(x_{1}, x_{0}+x_{2})$ or $( x_{0}+x_{1}, x_{0}+x_{2})$ for $\pi _{1}^{\varphi }=\pi _{2}$, or
 $\left( (x_{1}+x_{2})^{\delta },x_{2}^{\delta }\right)
=( x_{0}+x_{1}+x_{2}, x_{0}+x_{2})$ or $(x_{1}+x_{2}, x_{0}+x_{2})$ for $\pi _{1}^{\varphi }=\pi
_{2}^{(1,\gamma )}$, and again all cases are ruled out by the argument used previously. Thus $\mathcal{D}_{1}$ and $\mathcal{D%
}_{2}$ are not isomorphic.
\end{proof}

\begin{example}\label{hyperbolic}
Let $G=T:G_{0}$, where $T$ is the translation group of $V=V_{4}(3)$ and $G_{0}$ the subgroup of $GO_{4}^{+}(3)$ generated by the  $\alpha $ and $\beta $
respectively represented by the matrices%
\[
\left( 
\begin{array}{cccc}
0 & 0 & -1 & 0 \\ 
0 & 0 & 0 & -1 \\ 
-1 & 0 & 0 & 0 \\ 
0 & -1 & 0 & 0%
\end{array}%
\right) \text{ and }\left( 
\begin{array}{cccc}
1 & 0 & 0 & 1 \\ 
1 & 1 & 1 & -1 \\ 
1 & -1 & 1 & 1 \\ 
-1 & 0 & 0 & 1%
\end{array}%
\right) \textit{.} 
\]%
Then $\mathcal{D}=(V,B^{G})$, where $B=\left\langle (1,0,0,0))\right\rangle \cup \left( \left\langle
(1,0,0,0)\right\rangle +(0,1,1,0)\right) $, is a $2$-$(3^4,6,2)$ design admitting $G$ as the full automorphism group. Also, up to conjugacy, the unique subgroups of $G$ acting flag-transitively on $\mathcal{D}$ are those of the form $T:H_{0}$ with $H_{0}$ listed in Table \ref{Ortog}.

\begin{table}[tbp]
\footnotesize
\caption{Translation complements of flag-transitive automorphism groups of the $2$-$(3^4,6,2)$ design.}
\label{Ortog}
\footnotesize
\begin{tabular}{l|l|l|l}
\hline
Line & Generators of $H_{0}$ & $H_{0}$ & $H_{0,B}$ \\
\hline
1 & $\alpha ,\beta $ & $\left( \left( Z_{2}.S_{4}^{-}\right) :Z_{2}\right)
:Z_{2}$ & $S_{3}$ \\ 
2 & $\beta (\beta \alpha )^{2}\beta ^{9},\alpha ^{-1}\beta ^{-1}(\beta
^{-1}\alpha ^{-1})^{2}\beta ^{-2}$ & $\left( Z_{2}.S_{4}^{-}\right) :Z_{2}$
& $Z_{3}$ \\ 
3 & $\alpha \beta ^{-1}\alpha ^{-1}\beta ^{-1},(\alpha ^{-1}\beta
^{-1})^{3}\alpha ^{-1}$ & $\left( \left( Z_{8}\times Z_{2}\right)
:Z_{2}\right) :Z_{3}$ & $Z_{3}$ \\ 
4 & $(\beta \alpha )^{3}\beta ,\beta ^{-2}$ & $\left( \left( \left(
Z_{4}\times Z_{2}\right) :Z_{2}\right) :Z_{3}\right) :Z_{2}$ & $Z_{3}$ \\ 
5 & $\beta ^{3}\alpha ,\alpha \beta (\beta \alpha )^{2}\beta ^{6}(\beta
\alpha )^{3}\beta ^{2},\beta (\beta \alpha )^{2}\beta ^{5}(\beta \alpha
)^{3}\beta $ & $\left( Z_{2}\times SD_{16}\right) :Z_{2}$ & $Z_{2}$ \\ 
6 & $\beta ^{3}\alpha ,(\beta \alpha )^{3}\beta ,\beta (\beta \alpha
)^{2}\beta ^{5}(\beta \alpha )^{3}\beta $ & $Z_{2}\times SD_{16}$ & $1$ \\ 
7 & $\alpha \beta \left( \beta \alpha \right) ^{2}\beta ^{2}\alpha $, $%
\alpha \beta \left( \beta \alpha \right) ^{2}\beta ^{6}\left( \beta \alpha
\right) ^{3}\beta ^{2}$ & $\left( Z_{8}\times Z_{2}\right) :Z_{2}$ & $1$ \\ 
8 & $\beta \alpha \beta ^{4},\beta (\beta \alpha )^{2}\beta ^{5}(\beta
\alpha )^{3}\beta $ & $\left( Z_{8}:Z_{2}\right) :Z_{2}$ & $1$ \\ 
9 & $\beta (\beta \alpha )^{2}\beta ^{4}\alpha \beta ,\beta ^{3}\alpha $ & $%
\left( Z_{2}\times Z_{2}\right) .\left( Z_{4}\times Z_{2}\right) $ & $1$ \\ 
10 & $\beta (\beta \alpha )^{2}\beta ^{4}\alpha \beta ,\alpha \beta (\beta
\alpha )^{2}\beta ^{2}\alpha $ & $Z_{4}.D_{8}$ & $1$ \\ 
11 & $\alpha \beta (\beta \alpha )^{2}\beta ^{2}\alpha ,\alpha \beta (\beta
\alpha )^{2}\beta ^{6}(\beta \alpha )^{3}\beta ^{2},\beta (\beta \alpha
)^{2}\beta ^{5}$ & $\left( Z_{8}\times Z_{2}\right) :Z_{2}$ & $1$ \\ 
12 & $\alpha \beta (\beta \alpha )^{2}\beta ^{6}(\beta \alpha )^{3}\beta
^{2},\alpha \beta (\beta \alpha )^{2}\beta ,\beta (\beta \alpha )^{2}\beta
^{5}(\beta \alpha )^{3}\beta $ & $Z_{8}:\left( Z_{2}\times Z_{2}\right) $ & $%
1$ \\ 
13 & $\beta ^{3}\alpha ,\alpha \beta (\beta \alpha )^{2}\beta ^{6}(\beta
\alpha )^{3}\beta ^{2},\beta (\beta \alpha )^{2}\beta ^{5}$ & $\left(
Z_{2}\times Q_{8}\right) :Z_{2}$ & $1$ \\
\hline
\end{tabular}
\end{table}
\end{example}
\normalsize
\begin{proof}
$G$ is a rank $3$ group isomorphic to $\left( \left\langle
-1\right\rangle .S_{4}^{-}:Z_{2}\right) :Z_{2}$ with orbits on $V^{\ast }$of
length $32$ and $48$, which induce on $PG_{3}(3)$ the hyperbolic quadric $%
\mathcal{H}:X_{1}X_{4}+X_{2}X_{3}=0$, which is a Segre variety $\mathcal{S}_{1,1}$, and its complementary set respectively.

Now, let $e_{1}=(1,0,0,0)$ and $e_{2}+e_{3}=(0,1,1,0)$, and consider $%
B=\left\langle e_{1}\right\rangle \cup \left( \left\langle
e_{1}\right\rangle +e_{2}+e_{3}\right) $ consisting of two parallel affine $1
$-dimensional subspaces of the $2$-dimensional subspace $\pi =\left\langle
e_{1},e_{2}+e_{3}\right\rangle $. Then $\left\langle \tau _{e_{1}},-1\tau
_{e_{2}+e_{3}}\right\rangle $ is a transitive subgroup of $G_{B}$ isomorphic
to $S_{3}$, and by \texttt{GAP} \cite{GAP} it results that      
\[
G_{0,B}=\left\langle \tau _{e_{1}},-1\tau _{e_{2}+e_{3}},\alpha ^{-1}\beta
^{-1}(\beta ^{-1}\alpha ^{-1})^{2}\beta ^{-11},\beta \alpha \beta
^{-3}\alpha \right\rangle \cong S_{3}\times S_{3}\textit{.}
\]%
Thus $\mathcal{D}=(V,B^{G})$ is a flag-transitive $1$-design by \cite[1.2.6]%
{Demb}. Further $r=\left\vert B^{G_{0}}\right\vert =32$ and $\left\vert B \cap x^{G_{0}}\right\vert$ is $2$ or $3$ according to whether $x=e_{1}$ or $e_{2}+e_{3}$ respectively. Therefore, $\frac{\left\vert x^{G_{0}}\right\vert}{\left\vert B \cap x^{G_{0}}\right\vert r}=2$ and hence $\mathcal{D}$ is a $2$-$(3^4,6,2)$ design admitting $G$ as a flag-transitive automorphism group.

Let $A=Aut(\mathcal{D})$, then $A\leq AGL_{4}(3)$ by Lemma \ref{fullaut}, and hence $\left( \left\langle
-1\right\rangle .S_{4}^{-}:Z_{2}\right) :Z_{2} \cong G_{0}\leq G_{0}\leq
GL_{4}(3)$. If $A_{0}$ acts transitively on $V$, then there is a Sylow $5$-subgroup $S$ of $A_{0}$ preserving $B$ since $r=32$. So $S$ fixes $\left\langle
e_{1}\right\rangle$ pointwise, whereas $U$ acts irreducibly on $V$. Thus $A$ is rank $3$ group ad the $G_{0}$-orbits are $A_{0}$-orbits, hence $A_{0} \leq GO^{+}_{3}(4)$. Now, by \texttt{GAP} \cite{GAP} it results that $A_{0}=G_{0}$, hence $A=G$. Moreover, up to conjugacy, the unique subgroups of $G$ acting flag-transitively on $\mathcal{D}$ are those of the form $T:H_{0}$ with $H_{0}$ listed in Table \ref{Ortog}.
\end{proof}

\bigskip

In order to construct the next example we need the recall the following useful facts.

\bigskip

\subsection{$t$-spreads} A vectorial $t$\emph{-spread} $\mathcal{S}$ of $W=W_{h}(q)$ is a set of $%
t$-dimensional subspaces of $W$ partitioning $W\setminus\left\{ 0\right\} $. Clearly,
a $t\,$-spread of $W$ exists only if $t\mid h$. The incidence structure $%
\mathcal{A}(W,\mathcal{S})$, where the points are the vectors of $W$, the
lines are the additive cosets of the elements of $\mathcal{S}$, and the
incidence is the set-theoretic inclusion, is an \emph{Andr\'{e} translation
structure} with associated translation group $T=\left\{ x\rightarrow
x+w:w\in W\right\} $ and lines of size $q^{t}$. In particular $\mathcal{A}(W,%
\mathcal{S})$ is a $2$-$(q^{h},q^{t},1)$ design. If $\mathcal{S}$ and $%
\mathcal{S}^{\prime }$ are two $t$-spreads of $W$ such that $\psi $ is an
isomorphism from $\mathcal{A}(W,\mathcal{S})$ onto $\mathcal{A}(W,\mathcal{S}%
^{\prime })$ fixing the zero vector, then $\psi \in \Gamma L(W)$ and $%
\mathcal{S}^{\psi }=\mathcal{S}^{\prime }$. The converse is also true.

Let $%
H_{0}\leq \Gamma L(W)$ preserving $\mathcal{S}$, then $T:H_{0}$ is the
automorphism group of $\mathcal{A}(W,\mathcal{S})$. The subset of $End(W,+)$
preserving each component of $\mathcal{S}$ is a field called the \emph{kernel%
} of $\mathcal{S}$ and is denoted by $\mathcal{K}(W,\mathcal{S})$ or, simply,
by $\mathcal{K}$. Clearly, $\mathcal{K}\setminus\left\{ 0\right\} \leq H_{0}$. Each
component of $\mathcal{S}$ is a vector subspace of $W$ over $\mathcal{K}$
and $H_{0}\leq \Gamma L_{\mathcal{K}} (W)$. In particular, $\mathcal{A}(W,%
\mathcal{S})$ is a desarguesian affine space if, and only if, $\dim _{%
\mathcal{K}}Y=1$ for each $Y\in \mathcal{S}$. More information on $t$%
-spreads and Andr\'{e} translation structures can be found in \cite{Demb,Joh,Lu}.

\bigskip

\begin{example}\label{hall}
Let $V=V_{4}(3)$, $T$ br the translation group of $V$ and $X\cong SL_{2}(5)$, then the followings hold:
\begin{enumerate}
    \item $X$ preserves precisely two Hall $2$-spreads $\mathcal{S}$ and $\mathcal{S}^{\prime }$ of $V$;
    \item $Z_{8}\circ X$, $\left(Z_{8}\circ X\right) :Z_{2}$ or any of the two subgroups of $\left(Z_{8}\circ X\right) :Z_{2}$ isomorphic to $\left(Z_{4}\circ X\right) :Z_{2}$ acts transitively on $\mathcal{S}\cup \mathcal{S}^{\prime }$;
    \item $\mathcal{D}=(V,B^{G})$ with $B \in \mathcal{S}\cup \mathcal{S}^{\prime }$ and $G=T:G_{0}$, where $G_{0}$ in any of the groups in (2), is a $2$-$(81,9,2)$ design admitting $G$ as a flag-transitive automorphism group. Further $Aut(\mathcal{D})=T:\left( \left( Z_{8}\circ X\right) :Z_{2}\right)$.
\end{enumerate}
\end{example}
\begin{proof}
$X$ preserves precisely two Hall $2$-spreads $\mathcal{S}$ and $\mathcal{S}^{\prime }$ of $V$, and partitions $\mathcal{S}\cup \mathcal{S}^{\prime }$ into four orbits of length $5$ permuted transitively by preserved $N_{GL_{4}(3)}=(Z_{8} \circ X):Z_{2}$ by Lemma \ref{A5bid} and its proof (see below). Thus any of the groups $Z_{8}\circ X$, $\left(
Z_{4}\circ X\right) :Z_{2}$ or $\left( Z_{4}\circ X\right) :Z_{2}$, $\left(
Z_{8}\circ X\right) :Z_{2}$ acts transitively on $\mathcal{S}\cup \mathcal{S}%
^{\prime }$.

Let $B\in \mathcal{S}\cup \mathcal{S}^{\prime }$ and $%
G=T:G_{0}$ where $G_{0}$ is any of the groups $Z_{8}\circ X$, $\left(
Z_{4}\circ X\right) :Z_{2}$ or $\left( Z_{8}\circ X\right) :Z_{2}$, $\left(
Z_{4}\circ X\right) :Z_{2}$. Then $\mathcal{D}=(V,B^{G})$ is a $2$-$(81,9,2)$
design since both $\mathcal{S}$ and $\mathcal{S}^{\prime }$ are spreads of $V
$. Further, $G$ acts flag-transitively on $\mathcal{D}$.

Finally, let $A=Aut(\mathcal{D})$, then $A\leq AGL_{4}(3)$ by Lemma \ref{fullaut}.
Hence $\left( Z_{8}\circ X\right) :Z_{2}\leq A_{0}\leq \Gamma L_{2}(9)$ or $%
SL_{4}(3)\trianglelefteq A_{0}$ by \cite[Tables 8.8]{BHRD}. The latter is
actually ruled out since $SL_{4}(3)$ does not have a transitive permutation
representation of degree $20$ by \cite{At}. Then either $A_{0}=\left(
Z_{8}\circ X\right) :Z_{2}$ or $GL_{2}(9)$ $\leq A_{0}\leq \Gamma L_{2}(9)$
since $S_{5}$ is maximal in $P\Gamma L_{2}(9)$. In the latter case, since $%
Z_{8}=Z(GL_{4}(3))\trianglelefteq A_{0}$ and $Z_{8}$ partitions $\mathcal{S}%
\cup \mathcal{S}^{\prime }$ in $5$ orbits of length $4$ by \cite[Theorem
I.1.12 and subsequent remark]{Lu}, then $GL_{2}(9)$ must permute such
orbits. Then $SL_{2}(9)\leq A_{B}$ since the minimal primitive permutation
representation of $PSL_{2}(9)$ is $6$, and we reach a contradiction since $%
B=V_{2}(3)$. Thus $A_{0}=\left( Z_{8}\circ X\right) :Z_{2}$ and hence $%
A=T:\left( \left( Z_{8}\circ X\right) :Z_{2}\right)$. This completes the proof.
\end{proof}

\begin{example}\label{D8Q8}
Let $V=V_{3}(4)$ and $G=T:G_{0}$, where $T$ is the translation group of $V$ and $G_{0}\cong (D_{8}\circ Q_{8}).F_{10}$ as in \cite[pp. 144--145]%
{Hup0}. If $B$ the $2$-dimensional subspace of $V$ of equations $X_{1}=X_{2}=0$, then $\mathcal{D}=(V,B^{G})$ is a  $2$-$(81,9,2)
$-design admitting $G$ as a flag-transitive, point-$2$-transitive automorphism group. Further, $Aut(\mathcal{D})\cong G$.
\end{example}

\begin{proof}
The incidence structure $\mathcal{D}=(V,B^{G})$ is a $2$-$(81,9,\lambda )$ design
since $G$ acts $2$-transitively on $V$ by \cite{Hup0}. It
is not difficult to check that $G_{B}=T_{B}:G_{0,B}$ with $G_{0,B}
\cong D_{8}$. Thus $b=180$, and hence $r=20$ and $\lambda =2$. Clearly, $G$
acts flag-transitively on $\mathcal{D}$. Moreover, $N:=N_{GL_{4}(3)}(G_{0}) \cong (D_{8}\circ Q_{8}).F_{20}$, and $B^{N}=B^{G_{0}} \cup C^{G_{0}}$, where  $C$ has equations $X_{1}-X_{3}=X_{2}=0$ and $C^{G_{0}}\neq B^{G_{0}}$. In particular, $N \cap Aut(\mathcal{D})_{0}=G_{0}$ and $\left\vert N :G_{0} \right\vert=2$. Thus, $(V,C^{G})$ is isomorphic to $\mathcal{D}$.

Now, $Aut(\mathcal{D})\leq AGL_{4}(3)$ by Lemma \ref{fullaut}. Then $Aut(%
\mathcal{D})$ acts $2$-transitively on $V$ since $G\leq Aut(\mathcal{D})$. Moreover, $Aut(\mathcal{D})$ does not contain $SL_{4}(3)$ or $%
Sp_{4}(3)$ since any of these does not have a transitive permutation representation of
degree $20$ by \cite{At}. Thus, $G_{0}\leq Aut(\mathcal{D})_{0}\leq
(D_{8}\circ Q_{8}).A_{5}$ by \cite[Appendix 1]{Lieb0} since  $N \cap Aut(\mathcal{D})_{0}=G_{0}$.

If $Aut(\mathcal{D})_{0}\cong (D_{8}\circ Q_{8}).A_{5}$, then there is a copy of 
$SL_{2}(5)$ splitting $B^{G_{0}}$, as well as $C^{G_{0}}$, into two
orbits of length $10$. However, this impossible  by Lemma %
\ref{A5bid}(1)--(2) (see below). Thus $Aut(\mathcal{D})_{0}=G_{0}$, and
hence $Aut(\mathcal{D})=G$. This also shows that $\mathcal{D}$
is not isomorphic to the $2$-design constructed in Example \ref{hall}.
\end{proof}

\begin{example}\label{sem1dimdva}
Let $GF(3^{4}) \cong GF(3)\left( \omega \right)$, where $\omega$ is a root of $X^{4}+X-1$, and let $G=T:G_{0}$, where $T$ is the translation group of $%
GF(3^{4})$ and $G_{0}=\left\langle \alpha ,\beta \right\rangle < \Gamma L_{1}(3^{4})$ with $\alpha :x\rightarrow \omega ^{5}x$ and $\beta :x\rightarrow
x^{9}$. Then $\mathcal{D}=(GF(3^{4}),B^{G})$ with
$$B=\left\{ 0,1,\omega ^{16},\omega^{56},\omega ^{58},\omega ^{73}\right\}$$
is a $2$-$(3^{4},6,2)$ design admitting $G$ as the full automorphism group. Further, $G$ acts flag-transitively on $\mathcal{D}$.
\end{example}

\begin{proof}
The orbits
on $GF(3^{4})^{\ast }$ under $G_{0}\cong Z_{16}:Z_{2}$ are%
\begin{eqnarray*}
\mathcal{O}_{1} &=&GF(3^{2})^{\ast }\cup GF(3^{2})^{\ast }\omega ^{5} \\
\mathcal{O}_{2} &=&GF(3^{2})^{\ast }\omega \cup GF(3^{2})^{\ast }\omega
^{6}\cup GF(3^{2})^{\ast }\omega ^{9}\cup GF(3^{2})^{\ast }\omega ^{13} \\
\mathcal{O}_{3} &=&GF(3^{2})^{\ast }\omega ^{2}\cup GF(3^{2})^{\ast }\omega ^{7}\cup
GF(3^{2})^{\ast }\omega ^{18}\cup GF(3^{2})^{\ast }\omega ^{23}\text{,}
\end{eqnarray*}%
which have lengths $16,32$ and $32$, respectively. Then $B\cap \mathcal{O}%
_{1}=\left\{ 1\right\} $, $B\cap \mathcal{O}_{2}=\left\{ \omega ^{16},\omega
^{56}\right\} $ and $B\cap \mathcal{O}_{2}=\left\{ \omega ^{58},\omega
^{73}\right\} $, and hence $G_{0,B}=1$, which implies $\left\vert
B^{G_{0}}\right\vert =32$. Then $\mathcal{D}=(GF(3^{4}),B^{G})$ is a $1$%
-design by \cite[1.2.6]%
{Demb}, where $(v,r,b,k)=(3^{4},2^{5},6,2^{4}3^{3})$ since $%
G_{B}=T_{B}:\left\langle -1\tau _{1}\right\rangle \cong S_{3}$ with $\tau
_{1}:x\rightarrow x+1$. Further, $\frac{\left\vert B\cap \mathcal{O}%
_{i}\right\vert \left\vert B^{G_{0}}\right\vert }{\left\vert \mathcal{O}%
_{i}\right\vert }=2$ for each $i=1,2,3$. Thus $\mathcal{D}=(GF(3^{4}),B^{G})$
is a $2$-$(3^{4},6,2)$ design admitting $G$ as a flag-transitive
automorphism group.

Let $A=Aut(\mathcal{D})$, then $A\leq AGL_{4}(3)$ by Lemma \ref{fullaut}, and hence $Z_{16}:Z_{2} \cong G_{0}\leq A_{0}\leq
GL_{4}(3)$. Now, it is not difficult to see that $A_{0}\leq \Gamma L_{1}(3^{4})$
by using \cite{At}.

Note that $B\cap \mathcal{O}_{2}\cup \left\{ 0\right\} =\left\{
0,\omega ^{16},\omega ^{56}\right\} $ is a $1$-dimensional $GF(3)$ subspace,
hence each element $B^{G_{0}}$ intersects $\mathcal{O}_{2}\cup \left\{
0\right\} $ in a $1$-dimensional $GF(3)$ subspace of $V$. 

If $A_{0}$ contains an element $\delta$ of order $5$, then $\delta$ preserves an element $C$ of $B^{G_{0}}$ since $r=32$, hence $\delta$ fixes the $1$-dimensional $GF(3)$ subspace $V$ contained in $C$ pointwise, whereas $\delta$ acts semiregularly on $V^{\ast }$. Thus $A_{0}\leq $ $%
Z_{16}:Z_{4}$.

Since $%
\psi :x\rightarrow x^{3}$ maps $B$ onto $B^{\prime }=\left\{ 0,1,\omega
^{8},\omega ^{14},\omega ^{48},\omega ^{59}\right\} $ with $B^{\prime }\cap 
\mathcal{O}_{2}\cup \left\{ 0\right\} =\left\{ 0,\omega ^{14},\omega
^{59}\right\} $, which is not a $1$-dimensional $GF(3)$-subspace of $V$, it follows that
$B^{\prime }\notin B^{G_{0}}$ and so $\psi \notin A_{0}$. Thus $A_{0}=G_{0}$ and hence $A=G$.    
\end{proof}

\section{Further reductions and the Aschbacher's theorem}\label{Asch}

In this section, we provide additional constraints on the replication number of $\mathcal{D}$ and settle the case $SL_{n}(q) \unlhd G_{0}$. Finally, we start the study of the possibilities for $G_{0}$ based on the Aschbacher's theorem when $SL_{n}(q) \nleq G_{0}$. 

\bigskip
The case where the replication number $r$ of a flag-transitive $2$-$(v,k,2)$ design divides $ 4(\sqrt{v}-1)$ deserves a special attention, as it occurs many times in the paper. It is tackled in the next lemma.
\bigskip

\begin{lemma}\label{fato}
If $r \mid 4(p^{d/2}-1)$, then one of the following holds:
\begin{enumerate}
    \item[(i)] $\mathcal{D}$ is the $2$-$(5^{2},4,2)$ design as in Line 1 of Table \ref{smol};
    \item[(ii)]  $\mathcal{D}$ is the symmetric $2$-$(2^{4},6,2)$ design as in \cite[Section 1.2.1]{ORR} and $G_{0}$ is one of the groups $Z_{3} \times S_{3}$, $(Z_{3} \times Z_{3}) : Z_{4}$, $S_{3} \times S_{3}$,  $(S_{3} \times S_{3}) : Z_{2}$, $A_{5}$, $S_{5}$, $A_{6}$, $S_{6}$;
    \item[(iii)] $\mathcal{D}$ is the $2$-$(3^{4},6,2)$ designs as in Example \ref{hyperbolic} or \ref{sem1dimdva}.
\end{enumerate}
\end{lemma}
\begin{proof}
Assume that $r \mid 4(p^{d/2}-1)$. Then $r=\frac{4(p^{d/2}-1)}{\theta}$ for some $\theta \geq 1$, and hence $k=\frac{\theta(p^{d/2}+1)}{2}+1$. Further, $k \leq r$ implies $\theta =1,2$.  

Assume that $\theta=2$.
Then $k=p^{d/2}+2$, $r=2(p^{d/2}-1)=2k-6$. If $p\neq 2$, then $(k,v)=1$, and so 
$k\mid r$, and hence $k\mid 6$.
As $k\geq 4$, we obtain that $k=6$,
forcing $p=2$, $d=4$, 
a contradiction. Thus $p=2$ and hence $(k,v)=2$. So
$k\mid 2r$, forcing $k\mid 12$.
Hence $d=4$ and $(r,k)=(6,6)$. Then $\mathcal{D}$ is the symmetric $2$-$(16,6,2)$ design and $G_{0} \leq S_{6}$ described in \cite[Section 1.2.1]{ORR}. It is not difficult to see that $G_{0}$ is isomorphic to one of the groups $S_{6},A_{6},S_{5},(S_{3}\times S_{3}):Z_{2}$ or $Z_{2} \times S_{4}$ since $G$ acts point-primitively on $\mathcal{D}$, and we obtain (ii).

Therefore, $\theta=1$,
$k=\frac{%
p^{d/2}+1}{2}+1$
and $r=4(p^{d/2}-1)=8k-16$.
If $p=3$, then $(k,v)=3$, and so $k\mid 3r$, forcing $k\mid 48$.
Therefore, $(v,k,r)=(3^2,3,8),(3^4,6,32)$ with $d=2$ or $4$, respectively. 

Let $B$ be any block containing $0$. 
Since $k\nmid r$, we have that $T_B\neq 1$ by Lemma \ref{cici}, and so $T_{B}\cong Z_{3}$. If $k=p=3$, then $B$ is a $1$-dimensional $GF(3)$-subspace, and hence $\lambda
=1$, whereas $\lambda=2$. So, this case is ruled out. Therefore $(v,k,r)=(3^4,6,32)$, and hence $\mathcal{D}$ is one of the two $2$-$(3^{4},6,2)$ designs as in Examples \ref{hyperbolic} and \ref{sem1dimdva} by \texttt{GAP} \cite{GAP}, which is (iii).

If $p\neq 3$, then $(k,v)=1$, and so 
$k\mid r$, and hence $k\mid 16$, implying $d=2$ with 
\[
(v,k,r)=(5^{2},4,16),(13^{2},8,48),(29^{2},16,112)\text{.}
\]
Thus $G_{0}\leq GL_{2}(p)$ in each case. Furthermore, for any block $B$, 
we have that  $T_B=1$,
and so $G_{B}$ is isomorphic to a subgroup of $G_{0}$ of index $r/k$ by Lemma \ref{cici}. 

Assume that $\left( \left\vert G_{0}\right\vert ,p\right)
>1$, and let $C$ be any block containing $0$.
Then there exists an element
$\alpha\in G_{0}$ of order $p$ fixing $C$
pointwise since $(r,p)=1$ and $k<p$. Therefore $C\subseteq Fix(\alpha )$, which is a $1$-dimensional subspace of $V$. 
By the flag-transitivity of $G$, each block through $0$ is contained in exactly a $1$-dimensional subspace of $V$ and any $1$-dimensional subspace of 
$V$ contains a block through $0$. Moreover, for each block through $0$, there is a unique $1$-dimensional subspace containing it since $\mathcal{D}$ is a $2$-design. So $p+1\mid r$,
which is a contradiction. Thus $\left( \left\vert G_{0}\right\vert ,p\right)
=1$.

Let $B$ be any block of $\mathcal{D}$ containing $0$. Assume that $G_{B}$ is non-soluble. Then $p=29$ and $M \unlhd G_{B}$ with $M\cong Z_{2}.A_{5}$ by \cite[Tables 8.1 and 8.2]{BHRD} since $\left( \left\vert G_{0}\right\vert ,p\right)=1$. Then $B$ is union of $M$-orbits of equal length since $G_{B}$ acts transitively on $B$. Then $M$ fixes $B$ pointwise since $k=16$ and $M$ does not have non-trivial transitive permutation representations of degree a divisor of $16$ by \cite{At}. Thus $16\left\vert M \right\vert  \mid \left\vert G_{B} \right\vert$. On the other hand,  $G_{B}$ is isomorphic to a subgroup of $G_{0}$ of index $r/k=7$ by Lemma \ref{cici}. Hence, $112 \left\vert M \right\vert$ divides the order of $G_{0}$, and hence that of $GL_{2}(9)$, a contradiction. Thus, $G_{B}$ is solvable.

Since $G_{B}$ is isomorphic to a subgroup of $G_{0}$, it follows that $\left( \left\vert G_{B}\right\vert ,p\right)=1$, hence $G_{B}=G_{x,B}$ wih $x \notin B$ by \cite[Theorem 6.2.1]
{Go} since $G_{B}$ acts transitively on $B$. If $Fix(G_{B})$ contains a further point $y$, then $\frac{\left\vert G_{B} \right\vert}{2} \mid \left\vert G_{x,y, B^{\prime}} \right\vert$, where $B^{\prime}$ is any of the two blocks containing $x,y$. Thus $\left\vert G_{B}\right\vert=\left\vert G_{B^{\prime}}\right\vert \geq k \frac{\left\vert G_{B} \right\vert}{2}$, and hence $k\leq 2$, a contradiction. Thus $Fix(G_{B})=\{x\}$.

Let $z \in B$. Then $B \subset z^{G_{x}}$ since $G_{B}=G_{x,B}$. Then $\left\vert B^{G_{x}}\right\vert=\left\vert G_{x}:G_{x,B}\right\vert=\frac{r}{k}$ by Lemma \ref{cici}. Now, $\left(z^{G_{x}},B^{G_{x}}\right)$ is a $1$-design by by \cite[1.2.6]{Demb}, hence $\frac{r}{k}k=\left\vert z^{G_{x}}\right\vert t$, where $t$ is the number of blocks in $B^{G_{x}}$ containing $z$. Then $t=1,2$ since $r/2 \mid \left\vert z^{G_{x}}\right\vert$ by Lemma \ref{sudbina}(2), and hence $\left\vert z^{G_{x}}\right\vert=r$ or $r/2$, respectively. Consequently, there is a non-zero vector $w$ such that $w^{G_{0}}$ is of length $r/t$. If $G_{0}$ contains a $u$-element $\beta$, where $u$ is an odd prime dividing $p+1$, then $\beta$ must fix a point in $w^{G_{0}}$, whereas $\beta$ acts irreducibly on $V$. Thus, $G_{0}$ is a subgroup of the rank $3$ group $(Z_{p-1}\times Z_{p-1}):Z_{2}$ of order divisible by $r=4(p-1)=2^{4}s$, where $s=1,3,7$ for $p=5,13,29$, respectively, since $\left( \left\vert G_{0}\right\vert ,p\right)
=1$. Also, $G_{0} \nleq Z_{p-1}\times Z_{p-1}$ since each $G_{0}$-orbit on $V^{\ast}$ is divisible by $r/2$ by Lemma \ref{sudbina}(2). 

Suppose that $s>1$. Then $p=13$ or $29$. Let $S$ be a Sylow $s$-subgroup of $G_{0}$. Then $\left\vert
S\right\vert \leq s^{2}$ since $G_{0}\leq GL_{2}(p)$. If $\left\vert
S\right\vert =s^{2}$, then $s\mid \left\vert G_{B}\right\vert $ and $%
s^{2}\nmid \left\vert G_{B}\right\vert $ since $s\mid r/k$ and $s^{2}\nmid
r/k$. Then $Z_{s}\trianglelefteq G_{B}$ since $G_{B}$ is isomorphic to a
subgroup of $G_{0}$. Therefore $Z_{s}\leq G_{0,B}$ being $k$ a power of $2$,
and hence $B\subseteq Fix(Z_{s})$ since $G_{B}$ acts transitively on $B$.
Now, $Fix(Z_{s})=\left\langle y\right\rangle $ and so $p+1\mid r$, which is
not the case. Therefore $\left\vert S\right\vert =s$ and so $\left\vert
G_{0}\right\vert =2^{4+j}s$ with $j=0,1$ since $r=2^{4}s$ divides the order
of $G_{0}$ and $G_{0}$ is a subgroup of $(Z_{p-1}\times Z_{p-1}):Z_{2}$ not contained in $Z_{p-1}\times Z_{p-1}$. The same
conclusion easily follows for $p=5$ since $s=1$ in this case. Then $\mathcal{D}$ is the $2$-$(5^{2},4,2)$ design as in Line 1 of Table \ref{smol} by \texttt{GAP} \cite{GAP}, which is (i).  
\end{proof}

\begin{lemma}
\label{evo}If $\mathcal{D}$ admits $G\cong AGL_{1}(p^{d})$ as a
flag-transitive automorphism group, then $p^{d}\equiv 1\pmod{3}$ and $\mathcal{D}$ is a $2$-design as in Example \ref{sem1dimjedan}(1).
\end{lemma}

\begin{proof}
Let $B \in \mathcal{B}$, then $G_{B}$ is either $G_{B} \leq T$ or $G_{B} \cap T=1$, or $G_{B} \cap T \neq 1$ and $G_{B}$ is a Frobenius group by \cite[Proposition III.17.3]{Pass}.

If $G_{B} \leq T$, then $B$ is a $t$-dimensional subspace of 
$AG_{d}(p)$ since $G_{B}$ acts transitively on $B$, and hence $r=p^{d}-1$ since $G_{a,B}=1$, where $a \in B$. Then $k=\lambda +1= p^{t}$, where $p^{t}=\left\vert G_{B}\right \vert$, and hence $p^{t}=3$ since $\lambda=2$. Then $B$ is a $t$-dimensional subspace of 
$AG_{d}(p)$, and hence $\lambda=1$, a contradiction.

If $G_{B} \cap T=1$, then $G_{B}=G_{z,B}$ for a some $z \in GF(q)$ by \cite[Theorem 6.2.1]{Go} since $G_{B} \leq T:G_{0}$ and $G_{B}$ has order coprime to $p$. Actually, $z$ is the unique point of $GF(q)$ since $G \cong AGL_{1}(p^{d})$. Further, $z \notin B$ since $G_{B}$ acts transitively on $B$. Thus, if $y$ is any point in $B$, $G_{y,B}=1$ and hence $r=p^{d}-1$. Therefore, $\mathcal{D}$ is a $2$-$( p^{d},3,2)$ design with $p^{d} \equiv 1 \pmod{3}$. We may assume that $z=0$ since $G$ acts point-transitively on $\mathcal{D}$, and hence $B$ is any of the $\left\langle \alpha ^{\frac{p^{d}-1}{3} }\right\rangle$-orbits on $GF(q)^{\ast}$, where $\alpha: x \rightarrow \omega x$ with $\omega $ is a
primitive element of $GF(p^{d})$, since  $G \cong AGL_{1}(p^{d})$. Thus, $\mathcal{D}$ is a $2$-design as in Example \ref{sem1dimjedan}(1). 

If $G_{B}$ is a Frobenius group, then either $B$ is a $t$-dimensional subspace of 
$AG_{d}(p)$, or $G_{B}$ acts regularly on $B$ by \cite{Frob}. The former implies $2=\lambda=\frac{p^{t}-1}{p^{u}-1} \geq p^{u}+1$ with $u\mid (t,d)$ by \cite[Lemma 2.2]{Dru}, and we reach a contradiction since $p^{u}\geq 2$. Therefore $G_{B}$ acts regularly on $B$, and hence $\left\vert G_{0}\right\vert =r=\frac{2(p^{d}-1)}{k-1}$. Now, since $p^{d}-1$ divides the order of $G_{0}$ and $k>2$, it follows that $k=3$. Thus, $G_{B}=T_{B}\cong Z_{3}$, which is not the case by the above argument. This completes the proof.     
\end{proof}

\begin{lemma}
\label{zatvori}If $\mathcal{D}$ admits $AGL_{1}(p^{d})\trianglelefteq G\leq
A\Gamma L_{1}(p^{d})$ as a flag-transitive automorphism group, then $\mathcal{D}$ is one of the two $2$-designs as in Example \ref{sem1dimjedan}.
\end{lemma}

\begin{proof}
Let $X \cong AGL_{1}(p^{d})$ and let $B$ be any block of $\mathcal{D}$. Then $B$ is union of $X_{B}$-orbits of equal length, say $k_{0}$, since $X_{B}\trianglelefteq G_{B}$ and $G_{B}$ acts transitively on $B$. Let $C$ be any such $X_{B}$-orbits, then $\mathcal{D}_{X}=(GF(p^{d}),C^{X})$ is a $2$-$(v,k,\lambda
_{X})$ design since $X$ acts point-$2$-transitively on $GF(p^{d})$. Moreover, $%
\lambda _{X}\leq \lambda =2$ since $X\trianglelefteq G$ and $C \subseteq B$. If $\mathcal{D}_{X}=\mathcal{D}$, then $\mathcal{D}$ is a $2$-design as in Example \ref{sem1dimjedan}(1) by Lemma \ref{evo}. If $\mathcal{D}_{X}\neq\mathcal{D}$, then $G_{\mathcal{D}_{X}}$ is a proper subgroup of $G$.

Assume that $\lambda_{X}=2$. Let $x,x^{\prime}$ be any two distinct elements of $GF(p^{d})$ and let $C_{1},C_{2}$ be the blocks of $\mathcal{D}_{X}$ containing both $x$ and $x^{\prime}$. Now, let $g\in G\setminus G_{\mathcal{D}_{X}}$ and let $C_{3},C_{4}$ be the blocks of $\mathcal{D}_{X}^{g}$ containing both $x$ and $x^{\prime}$. Finally, let $B_{i}$ the blocks of $\mathcal{D}$ containing $C_{i}$ with $i=1,2,3,4$. Then $\{B_{1},B_{2}\}=\{B_{3},B_{4}\}$ since $\lambda=2$. We may assume that $B_{1}=B_{3}$ (and hence $B_{2}=B_{4}$). Thus both $C_{1}$ and $C_{3}$ are $X_{B_{1}}$-orbits, and hence $C_{1}=C_{3}$ since $x\in C_{1}\cap C_{3}$, Therefore $\mathcal{D}_{X}^{g}=\mathcal{D}_{X}$ and hence $g\in G_{\mathcal{D}_{X}}$, a contradiction.     

Assume that $\lambda _{X}=1$. Then $\mathcal{D}_{X}$ is an affine Desarguesian space
over a suitable subfield $GF(p^{t})$ of $GF(p^{d})$ by \cite[Theorem 1]{Ka},
and hence $t\mid d$ and $r_{X}=\frac{p^{d}-1}{p^{t}-1}$. In particular, $G_{\mathcal{D}_{X}}$ is a subgroup of $X_{0}:\left\langle \sigma ^{t}\right\rangle$, where $\sigma :x\rightarrow x^{p}$, acting flag-transitively on $\mathcal{D}_{X}$. Further, $%
\left\vert G:G_{\mathcal{D}_{X}}\right\vert =2$ since $\lambda =2$, and hence each block of $\mathcal{D}_{X}$ is contained in a unique block of $\mathcal{D}$. Therefore $r=2r_{X}$, and hence $k=\frac{2(p^{d}-1)}{r}+1=\frac{(p^{d}-1)}{r_{X}}+1=(p^{t}-1)+1=p^{t}$. Thus, the blocks of $\mathcal{D}_{X}$ are blocks of $\mathcal{D}$, and hence $\mathcal{D}$ is a $2$-$(p^{d},p^{t},2)$ design.  

If $%
\xi =\omega ^{\frac{p^{d}-1}{p^{t}-1}}$, where $\omega $ is a primitive
element of $GF(p^{d})$, then the blocks of $\mathcal{D}_{X}$ containing $0$
are $B_{i}=\omega ^{i}GF(p^{t})$ with $i=0,...,p^{d}-1$, where $B_{i}=B_{j}$
if and only if $\omega ^{i-j}\in GF(p^{t})$ since $\mathcal{D}_{X} \cong AG_{d/t}(p^{t})$. Hence, if $GF(p^{d})$ is regarded as $%
d $-dimensional $GF(p)$-subspace, $\mathcal{S}=\left\{ B_{i}\right\}
_{i=0}^{p^{d}-1}$ is a Desarguesian $t$-spread of $GF(p^{d})$. Moreover, $\mathcal{S}$ is a $X_{0}$-orbit and the set of blocks of $\mathcal{D}$ containing $0$ is the (disjoint) union of two $X_{0}$-orbits, namely $\mathcal{S}$ and $\mathcal{S}^{\sigma^{t/2}}$ ($t$ must be even) since $\left\vert G:G_{\mathcal{D}_{X}}\right\vert =2$, $X\trianglelefteq G\leq
A\Gamma L_{1}(p^{d})$ and $A\Gamma L_{1}(p^{d}) \cap Aut(\mathcal{D}_{X})=T:(X_{0}:\left\langle \sigma ^{t}\right\rangle)$. Thus, $\mathcal{D}$ is a $2$-design as in Example \ref{sem1dimjedan}(2).
\end{proof}

\bigskip

\subsection{Irreducibility of $G_{0}$ on V}\label{irr}
For each divisor $n$ of $d$ the group $\Gamma L_{n}(p^{d/n})$ has a natural
irreducible action on $V$. As the point-primitivity of $G$ on $\mathcal{D}$
implies that $G_{0}$ acts irreducibly on $V_{d}(p)$, we may chose $n$ to be
minimal such that $G_{0}\leq \Gamma L_{n}(p^{d/n})$ in this action and write 
$q=p^{d/n}$. In the sequel, we assume that $n>1$. Then $r<k$ for $(v,k) \neq (16,6)$ by \cite[Theorem 3]{ORR} again since $G$ acts point-primitively on $\mathcal{D}$.

\bigskip 
Let $X$ be any of the classical groups $SL_{n}(q)$, $Sp_{n}(q)$, $SU_{n}(q^{1/2})$ or $\Omega^{\varepsilon}(q)$ on $V=V_{n}(q)$, chose to be minimal such that $G_{0} \leq \Gamma$, where $\Gamma =N_{\Gamma L_{n}(q)}(X)$. Now, we analyze the case where $G_{0}$ contains $X$.
\bigskip

\begin{lemma}
\label{vrata} If $G_{0}$ contains as a normal subgroup one of the groups $SL_{n}(q)$ or $%
Sp_{n}(q)$, then one of the following holds:
\begin{enumerate}
    \item $\mathcal{D}$ is one of the two $2$-designs as in Example \ref{sem1dimjedan};
    \item $\mathcal{D}$ is the $2$-$(2^4,2^2,2)$ design as in \cite[Example 2.1]{Mo} and $SL_{2}(4) \unlhd G_{0} \leq \Sigma L_{2}(4)$;
    \item $\mathcal{D}$ is the symmetric $2$-$(2^{4},6,2)$ design described in \cite[Section 1.2.1]{ORR} and $SL_{2}(4) \unlhd G_{0} \leq \Sigma L_{2}(4)$.
\end{enumerate}
\end{lemma}

\begin{proof}
Assume that $G_{0}$ contains as a normal subgroup $X_{0}$ isomorphic to one of the groups $SL_{n}(q)$ or $%
Sp_{n}(q)$, and let $B$ be any
block of $\mathcal{D}$ containing $0$. Then $\mathcal{D}_{X}=(V,B^{X})$, where $X=T:X_{0}$, is a $2$-$%
(q^{n},k,\lambda ^{\prime })$ design with $\lambda ^{\prime }\leq \lambda =2$ since $X$ acts
point-$2$-transitively on $V$.

Assume that $\lambda ^{\prime }=1$. Then $B\subseteq \left\langle
x\right\rangle $ for some $1$-dimensional subspace  $\left\langle
x\right\rangle$ of $V$ and $k=p^{t}$ by \cite[Theorem 1 and Remark after Lemma 4.6%
]{Ka}. Moreover, $AGL_{1}(q)\trianglelefteq G_{\left\langle
x\right\rangle }^{\left\langle x\right\rangle }\leq A\Gamma L_{1}(q)$, where $G_{\left\langle
x\right\rangle }^{\left\langle x\right\rangle }$ is the group induced on $\left\langle
x\right\rangle$ by $G_{\left\langle x\right\rangle }$. Hence, $\left( \left\langle x\right\rangle ,B^{G_{\left\langle
x\right\rangle }^{\left\langle x\right\rangle }}\right)$ is a $2$-$%
(q,p^{t},2)$ design as in Example \ref{sem1dimjedan}(2) by Lemma \ref{zatvori}. Further, $t$ is an even divisor of $d/n$ since $q=p^{d/n}$, and $AGL_{1}(q)\trianglelefteq G_{\left\langle
x\right\rangle }^{\left\langle x\right\rangle }\leq T_{\left\langle x\right\rangle }:(GL_{1}(q):\left\langle \sigma ^{t/2}\right\rangle)$, where $\sigma$ is a semilinear map of $V_{n}(q)$ induced by the Frobenius automorphism of $GF(q)$. Thus, $X_{0} \unlhd G_{0} \leq X_{0}:\left\langle \sigma ^{t/2}\right\rangle$, and hence $\mathcal{D}=(V,B^{G})$ is as in Example \ref{sem1dimjedan}(2).

Assume that $\lambda ^{\prime }=2$. Then $\mathcal{D}_{X}=\mathcal{D}$, and hence $X$ acts flag-transitively
on $\mathcal{D}$. Therefore, we may assume that $G=X$.

Let $x \in V$, $x \neq 0$, and $B$ any block of $\mathcal{D}$ containing $x,0$. Then $G_{0,x}$ contains $[q^{n-1}]:SL_{n-1}(q)$ or $%
[q^{n-1}]:Sp_{n-2}(q)$ according to whether $G_{0}$ contains $SL_{n}(q)$ or $%
Sp_{n}(q)$, respectively, by \cite[Propositions 4.1.17(II) and 4.1.19(II)]{KL}%
. Thus, either $G_{0,x}\leq G_{0,B}$, or $n=2$, $q$ is even, $%
\left\vert G_{0,x}\right\vert =q$, $\left\vert G_{0,x,B}\right\vert =q/2$ and $G_{0,x}\nleq G_{0,B}$ since $\lambda=2$.

Assume that $G_{0,x}\leq G_{0,B}$. Suppose that $B$ intersects some $1$%
-dimensional $GF(q)$-subspace $\left\langle y\right\rangle $ distinct from $%
\left\langle x\right\rangle $ in a point other than $0$. Since $G_{0,x}$ acts transitively on $%
PG_{1}(q)\setminus \left\{ \left\langle x\right\rangle \right\} $, it
follows that $k-1\geq \frac{q^{n}-1}{q-1}$. Then $r\leq 2(q-1)$ since $r=%
\frac{2(q^{n}-1)}{k-1}$. On the other hand, $r>q^{n/2}\sqrt{2}$ and so $%
q^{n/2}<\sqrt{2}(q-1)$. Thus $n=2$, $q>3$ and $k\geq q+2$, and hence $%
r<2(q-1) $. Since $G_{0,x}$ contains a Sylow $p$%
-subgroup of $G$, it follows that $\left\vert G_{0,B}\right\vert \mid %
q(q-1)$ and hence $q+1 \mid r$. Actually, $r=q+1$ since $q\sqrt{2}<r<2(q-1)$, but this contradicts $r \geq k \geq q+2$. Thus $ B \subset \left\langle
x\right\rangle $ since $\lambda =2$, hence $\left( \left\langle x\right\rangle, B^{G_{\left\langle x\right\rangle }^{\left\langle x\right\rangle }}\right )$  is a $2$-design as in Example \ref{sem1dimjedan}(1) by Lemma \ref{evo} since $G_{\left\langle x\right\rangle }^{\left\langle x\right\rangle } \cong AGL_{1}(q)$. Therefore, $\mathcal{D}=(V,B^{G})$ is a $2$-design as in Example \ref{sem1dimjedan}(1). 

It remains to analyze the case $n=2$, $q$ is even, $%
\left\vert G_{0,x}\right\vert =q$, $\left\vert G_{0,x,B}\right\vert =q/2$ and $G_{0,x}\nleq G_{0,B}$. In this case, if $q>4$, $G_{0,B}$ is a subgroup of a Frobenius subgroup of order $q(q-1)$ by \cite[Tables 8.1 and 8.2]{BHRD}. Hence $\left\vert G_{0,B}\right\vert=\left\vert G_{0,x,B}\right\vert =q/2$ since $\left\vert G_{0,x}\right\vert =q$ and $G_{0,x}\nleq G_{0,B}$. Therefore $r=2(q^{2}-1)$, and hence $k=2$ since $r= \frac{2(q^{2}-1)}{k-1}$, a contradiction since $\mathcal{D}$ is non-trivial by our assumption. Then $q=4$, $G_{0} \cong SL_{2}(4)$. Further, $r=6,10$ again since $\mathcal{D}$ is non-trivial and $r$ is even. Then $\mathcal{D}$ is the $2$-$(2^4,2^2,2)$ design as in \cite[Example 2.1]{Mo} by \cite[Corollary 1.2]{Mo} for $r=10$, and $\mathcal{D}$ is the symmetric $2$-$(2^{4},6,2)$ design described in \cite[Section 1.2.1]{ORR} for $r=6$. This completes the proof.   
\end{proof}

\begin{lemma}\label{cupeta}
If $SU_{n}(q^{1/2})\trianglelefteq G_{0}$, then one of the following holds:
\begin{enumerate}
    \item $n=2$, $q^{1/2}=2$ and $\mathcal{D}$ is the $2$-$(2^{4},6,2)$ design as in \cite[Section 1.2.1]{ORR};
    \item $n=2$, $q^{1/2}=3$ and $\mathcal{D}$ is the $2$-$(3^{4},6,2)$ designs as in Example \ref{hyperbolic}.
\end{enumerate}

\end{lemma}
\begin{proof}
Assume that $SU_{n}(q^{1/2})\trianglelefteq G_{0}$. Then a $G_{0}$-orbit is the set of isotropic non-zero vectors of $V_{n}(q)$, and its length is $(q^{n/2}-1)(q^{(n-1)/2}+1)$ for $n$ even. In this case, $%
r/2\mid ((q^{n/2}-1)(q^{(n-1)/2}+1),q^{n}-1)$ by Lemma \ref{sudbina}(2), and hence $r\mid 4(q^{n/2}-1)$. Thus, $n=2$ and $q^{1/2}=2,3$ and $\mathcal{D}$ is either the $2$-$(2^{4},6,2)$ design described in \cite[Section 1.2.1]{ORR},  or the $2$-$(3^{4},6,2)$ design as in Example \ref{hyperbolic}, respectively, by Lemma \ref{fato}. Thus $n$ is odd.

Now, we are going to complete the proof in a series of steps.
\begin{claim}
$n=3$, $r\mid 2(q^{1/2}-1)(q^{3/2}+1)$ and $r\leq
(q^{1/2}-1)(q^{3/2}+1)$.    
\end{claim}
Then the number of of isotropic non-zero vectors of $V_{n}(q)$ is $(q^{n/2}+1)(q^{(n-1)/2}-1)$. Arguing as in \cite[Lemma 3.2]{Lieb} with $r/2$ and $q^{3/2} / \sqrt{2}$ in the role of $r$ and $q^{3/2}$, respectively, one has $%
r/2\mid (q^{1/2}-1)(q^{n/2}+1)$. Thus $r\mid 2(q^{1/2}-1)(q^{n/2}+1)$, and
hence either $r=2(q^{1/2}-1)(q^{n/2}+1)$ or $r\leq (q^{1/2}-1)(q^{n/2}+1)$.
Suppose that $r=2(q^{1/2}-1)(q^{n/2}+1)$. Hence, $k=\frac{q^{n/2}+q^{1/2}-2}{q^{1/2}-1}$. If $q^{1/2}>2$, then $(k,v)=1$ for 
$v$ odd and $(k,v)=2$ for $v$ even. Also $%
(q^{n/2}+q^{1/2}-2,q^{n/2}+1)=(q^{1/2}-3,q^{n/2}+1)$ and hence $k\leq
2(k,v)(q^{1/2}-3)(q^{1/2}-1)$. Thus%
\[
\frac{q^{n/2}+q^{1/2}-2}{q^{1/2}-1}\leq 2(k,v)(q^{1/2}-3)(q^{1/2}-1)\textit{,}
\]%
and hence $n=3$ since $n$ is odd. Then $k=\allowbreak q+q^{1/2}+2$ and $%
r=2(q^{1/2}-1)(q^{3/2}+1)$. Now, $k\mid vr$ implies $k\mid 14$ since $%
q^{1/2}>2$, which is not the case. Thus $q^{1/2}=2$, and hence $r=2(2^{n}+1)$
and $k=2^{n}$. The set of blocks of $\mathcal{D}$ containing $0$ is partitioned in $SU_{n}(2)$-orbits of equal length since $G_{0}$ acts transitively on it and $SU_{n}(2) \trianglelefteq G_{0}$. Further, the minimal degree of the primitive permutation representations of $%
SU_{n}(2)$ is a lower bound for any such orbits, and so $r$ is greater than, or equal to, the the minimal degree of the primitive permutation representations of $%
SU_{n}(2)$, which is $9$ or $(2^{n}+1)(2^{n-1}-1)/3$ according to whether $n=3$ or $%
n\geq 5$, respectively, by \cite[Proposition 5.2.1 and Theorem 5.2.2]{KL}. However, only the
former is less than $r$, hence $v=8^{2}$, $k=8$ and $SU_{3}(2)\trianglelefteq
G\leq \Gamma U_{3}(2)$. In particular, $G\nleq A\Gamma L_{1}(8)$ and hence
no \ cases occur by \cite[Corollary 1.3]{Mo}.

Suppose that $r\leq (q^{1/2}-1)(q^{n/2}+1)$, the argument of \cite[Lemma 3.2]%
{Lieb} with $r/2$ and $q^{3/2} / \sqrt{2}$ in the role of $r$ and $q^{3/2}$, respectively, implies $n=3$. Therefore, $r\mid 2(q^{1/2}-1)(q^{3/2}+1)$ with $r\leq
(q^{1/2}-1)(q^{3/2}+1)$. 

\begin{claim}
$r=x(q^{3/2}+1)$ with $x\mid 2(q^{1/2}-1)$ and $1<x<2(q^{1/2}-1)$, and $k=\frac{2(q^{3/2}-1)}{x}+1$.    
\end{claim}

Clearly, the set of $r$ blocks of $\mathcal{D}$ containing $0$ are permuted transitively by $G_{0}$, and hence it is a union of $SU_{3}(q^{1/2})$-orbits of equal lengths, say $r_{0}$, since $SU_{3}(q^{1/2})\trianglelefteq G_{0}$. Each such $SU_{3}(q^{1/2})$-orbit is in turn a union $r_{1}$ orbits of equal length under $Z$, where $Z$ is the center of $SU_{3}(q^{1/2})$, and these are permuted transitively by $SU_{3}(q^{1/2})$. Thus either $r_{1}=r_{0}$ or $r_{1}=\frac{r_{0}}{(3,q+1)}$. Further, $SU_{3}(q^{1/2})$ induces $PSU_{3}(q^{1/2})$ on each set of $r_{1}$ orbits under $Z$. Now, if $q^{1/2} \geq 7$, the smallest nontrivial permutation degrees of $PSU_{3}(q^{1/2})$ are $q^{3/2}+1$ and $q(q-q^{1/2}+1)$ which correspond to the actions of $PSU_{3}(q^{1/2})$ on the set of isotropic and non-isotropic points of $PG_{2}(q)$, respectively. If $r_{1} \geq q(q-q^{1/2}+1)$, then $q(q-q^{1/2}+1)\leq r\leq (q^{1/2}-1)(q^{3/2}+1)$, a contradiction. Thus $r_{1}=q^{3/2}+1$, and hence $r=x(q^{3/2}+1)$ with $x \geq 1$ since $r_{1}$ divides $r_{0}$ and hence $r$. Therefore, $k=\frac{2(q^{3/2}-1)}{x}+1$. The same conclusion holds for $q^{1/2} \leq 5$ by \cite{At}. Further, for any $q^{1/2} \geq2$ it results that $x\mid 2(q^{1/2}-1)$ and $1<x<2(q^{1/2}-1)$ since $q^{3/2}\sqrt{2} <r \leq
(q^{1/2}-1)(q^{3/2}+1)$.

\begin{claim}
$G_{0,B}$ contains a subgroup $P$ of a Sylow $p$-subgroup $S$ of $SU_{3}(q^{1/2})$.    
\end{claim}
Let $B$ be any block of $\mathcal{D}$ incident with $0$. Then $G_{0,B}$ contains a subgroup $P$ of a Sylow $p$%
-subgroup $S$ of $SU_{3}(q^{1/2})$ of index at most $(2,p)$ since $4$ does not divide $r$ when $p$ is $2$. Further, either $SU_{3}(q^{1/2}) \unlhd G_{0,B}$, or $G_{0,B}\cap SU_{3}(q^{1/2}) \leq S:H$, where $H$ is cyclic of order $q-1$, by \cite[Tables 8.5 and 8.6]{BHRD}. The former implies that $r$ divides $2\frac{d}{n}(q^{1/2}-1)^{2}$, which is the index of $SU_{3}(q^{1/2})$ in $\Gamma U_{3}(q^{1/2})$. However, this is impossible since $r=x(q^{3/2}+1)$ with $x > 1$. Thus, $G_{0,B}\cap SU_{3}(q^{1/2}) \leq S:H$.

Assume that that $p=2$ and $\left\vert P\right \vert=q^{3/2}/2$. If $\left \vert P \cap Z(S) \right\vert =q^{1/2}/2$, then $P \cap Q =\varnothing$, where $Q$ denotes the set of the elements of order $4$ of $S$ whose squares are involutions lying in $Z(S) \setminus P$. Now, $S$ contains $q^{3/2}-q^{1/2}$ elements of order $4$ and $q^{1/2}-1$ involutions. Further, the involutions of $S$ lie in a unique conjugacy class under the subgroup of order $q^{1/2}-1$ of $H$ (see \cite[Satz II.10.12]{Hup}). Then $\left\vert Q \right\vert =(q+q^{1/2})q^{1/2}/2$ since $\left\vert Z(S) \setminus P \right\vert =q^{1/2}/2$, and hence $q^{3/2}/2=\left\vert P\right \vert \leq \left\vert S \setminus Q \right\vert=q^{3/2}-(q+q^{1/2})q^{1/2}/2$, a contradiction. Thus $Z(S)<P$, and hence $$P/Z(S) \leq (G_{0,B}\cap SU_{3}(q^{1/2}))/Z(S) < SH/Z(S) \textit{.}$$
Hence, $P/Z(S)$ is a Sylow $2$-subgroup of $(G_{0,B}\cap SU_{3}(q^{1/2}))/Z(S)$ and is a subgroup of order $q/2$ of the Frobenius group $SH/Z(S)$ of order $q(q-1)$. Then $(G_{0,B}\cap SU_{3}(q^{1/2}))/Z(S)=P/Z(S)$ by \cite[Proposition III.17.3]{Pass} since $(q/2-1,q-1)=1$, and hence $G_{0,B}=P$ since $P$ is a Sylow $2$-subgroup of $G_{0,B}$ and $Z(S)<P$. So $(q^{3/2}+1)(q-1) \mid r$, whereas $r=x(q^{3/2}+1)$ with $q$ even and $x\mid 2(q^{1/2}-1)$, a contradiction.
\begin{claim}
$(q^{1/2},r,k)=(3,56,27)$ or $(7,688,343)$.    
\end{claim}
Assume that $P=S$. Hence, $S$ is a Sylow $p$-subgroup of $G_{0,B} \cap SU_{3}(q^{1/2})$. Note that, $B$ contains both non-isotropic and non-zero isotropic vectors of $V$ since $\mathcal{D}$ is a $2$-design, and $q+q^{1/2}+2<k \leq q^{3/2}$ since $1<x<2(q^{1/2}-1)$. It can be deduced from \cite[Satz II.10.12]{Hup} that $P$ preserves a subspace chain $\left\langle y \right \rangle <\left\langle y,z \right \rangle <V$, where $\left\langle y \right \rangle$ is the unique isotropic subspace of $\left\langle y,z \right \rangle$. Further, $P$ fixes $\left\langle y \right \rangle$ pointwise, each $P$-orbit in $\left\langle y,z \right \rangle \setminus \left\langle y \right \rangle$ is of length $q$, and $P$ acts semiregularly on $V\setminus \left\langle y,z \right \rangle$. Thus, $B \subset \left\langle y,z \right \rangle$ and $B$ intersects each $1$-dimensional subspace of $\left\langle y,z \right \rangle$ in at least a non-zero vector since $B$ contains both non-isotropic and non-zero isotropic vectors of $V$ and $q+q^{1/2}+2<k \leq q^{3/2}$. Then $G_{0,B}\cap SU_{3}(q^{1/2})$ contains $S$ and a cyclic subgroup of order $\frac{q^{1/2}-1}{(q^{1/2}-1,2)}$ by \cite[Satz II.10.12]{Hup} since $\lambda=2$, and hence $$\left\vert SU_{3}(q^{1/2}) : G_{0,B}\cap SU_{3}(q^{1/2})\right \vert=(q^{3/2}+1)(2,q^{1/2}-1)\frac{q^{1/2}+1}{\theta}$$ for some $\theta \geq 1$ since $G_{0,B}\cap SU_{3}(q^{1/2}) \leq S:H$. Moreover, $\left\vert SU_{3}(q^{1/2}) : G_{0,B}\cap SU_{3}(q^{1/2})\right \vert$ divides $r$ since $SU_{3}(q^{1/2}) \unlhd G_{0}$, and hence $\frac{q^{1/2}+1}{(2,q^{1/2}+1)} \mid \theta$ since $r=x(q^{3/2}+1)$ with $x\mid 2(q^{1/2}-1)$. Therefore, $G_{0,B}\cap SU_{3}(q^{1/2})=S:C$ with $C$  a subgroup of $H$ of order $(q-1)/c$, $c \mid (4,q-1)$. We may assume that $C$ preserves $\left\langle z \right \rangle$, other than $\left\langle y \right \rangle$, since $S$ acts transitively on the set of $1$-dimensional subspaces of $\left\langle y,z \right \rangle$ distinct from $\left\langle y \right \rangle$. Further, $C$ acts semiregularly on $\left\langle y \right \rangle^{\ast}$, on $\left\langle y,z \right \rangle \setminus (\left\langle y \right \rangle \cup \left\langle z \right \rangle )$, and each $C$-orbit on $\left\langle y \right \rangle^{\ast}$ is divisible by $\frac{q^{1/2}+1}{(c,q^{1/2}+1)}$ by \cite[Satz II.10.12]{Hup}. Thus, each $C$-orbit on $\left\langle x,y \right \rangle^{\ast}$ is divisible by $\frac{q^{1/2}+1}{(c,q^{1/2}+1)}$, and hence $\frac{q^{1/2}+1}{(c,q^{1/2}+1)} \mid k-1$ since $C <G_{0,B}$ and $B \subset \left\langle y,z \right \rangle$. Then $\frac{q^{1/2}+1}{(c,q^{1/2}+1)} \mid \frac{2(q^{3/2}-1)}{x}$ since $k=\frac{2(q^{3/2}-1)}{x}+1$, and hence $q^{1/2}+1 \mid 16$. Thus $q^{1/2}=p=3$ or $7$, and hence $(r,k)=(56,27)$ or $(688,343)$, respectively.
\begin{claim}
The final contradiction.    
\end{claim}
Note that $k \nmid r$ in both cases, then $T_{B} \neq 1$ by Lemma \ref{cici}, hence $B$ is a $3$-dimensional $GF(p)$-subspace of $\left\langle y,z \right \rangle$ (regarded as a $4$-dimensional space over $GF(p)$) by Corollary \ref{p2}. Then $p^{3}=k=\left\vert B \cap \left\langle y \right \rangle \right\vert +p^{2}(\left\vert B \cap \left\langle z \right \rangle^{\ast} \right\vert)$ since $B$ intersects each $1$-dimensional subspace of $\left\langle y,z \right \rangle$ in at least a non-zero vector and $S$ acts transitively on the set of $1$-dimensional subspaces of $\left\langle y,z \right \rangle$ distinct from $\left\langle y \right \rangle$. Thus $\left\langle y \right \rangle \subset B$ since $\left\vert B \cap \left\langle z \right \rangle^{\ast} \right\vert \geq p-1$, and hence $B$ contains a $GF(p)$-basis of $\left\langle y,z \right \rangle$, whereas $B$ is a $3$-dimensional $GF(p)$-subspace of $\left\langle y,z \right \rangle$. This completes the proof.
\end{proof}

\begin{proposition}\label{C8}
If $G_{0}$ contains any of the classical groups $Sp_{n}(q)$, $SU_{n}(q^{1/2})$, or $\Omega _{n}^{\varepsilon }(q)$, then one of the following holds:
\begin{enumerate}
\item $Sp_{n}(q) \unlhd G_{0}$ and $\mathcal{D}$ is one of the two $2$-designs as in Example \ref{sem1dimjedan}.
    \item $SU_{2}(2)\trianglelefteq G_{0}$ and $\mathcal{D}$ is the $2$-$(2^{4},6,2)$ design as in \cite[Section 1.2.1]{ORR};
    \item $SU_{2}(3)\trianglelefteq G_{0}$ and is the $2$-$(3^{4},6,2)$ designs as in Example \ref{hyperbolic}.
\end{enumerate}
\end{proposition}
\begin{proof}  
In order to prove the assertion we only need to analyze the case $\Omega _{n}^{\varepsilon }(q)\trianglelefteq G_{0}$, where $%
\varepsilon =\circ ,\pm $, since (1) follows from Lemma \ref{vrata} for $Sp_{n}(q)\trianglelefteq G_{0}$, and (2) and (3) from Lemma \ref{cupeta} for $SU_{n}(q^{1/2}) \trianglelefteq G_{0}$. Hence, assume that $\Omega _{n}^{\varepsilon }(q)\trianglelefteq G_{0}$, where $%
\varepsilon =\circ ,\pm $.

Assume that $\varepsilon =\circ $. Thus, $n$ is
odd. Arguing as in \cite[Lemma 3.2]{Lieb} with $r/2$ in the role of his $r$,
one has $r/2\mid q-1$, whereas $r>q^{n/2}\sqrt{2}$. Thus $n$ is even.

Assume that $\varepsilon =+$. Then $r/2\mid 2(q^{n/2}-1)$
by \cite[Lemma 3.2]{Lieb}, and hence $r\mid 4(q^{n/2}-1)$. Then $\mathcal{D}$ is as $2$-$(3^{4},6,2)$ design as in Example \ref{hyperbolic}, and hence $G_{0}$ is one of the groups listed in Table \ref{Ortog}, by Lemma \ref{fato} since $G \nleq A \Gamma L_{1}(81)$. However, this is impossible since none of the groups listed in Table \ref{Ortog} contains $\Omega_{4}^{+}(3)$, whereas $G_{0}$ does. 

Assume that $\varepsilon =-$. If $n=2$, then $G_{0}\leq
N_{\Gamma L_{2}(q)}(\Omega _{2}^{-}(q))\leq \Gamma L_{1}(q^{2})$, contrary
to our assumptions. Thus $n\geq 4$. Arguing as in \cite[Lemma 3.2]{Lieb},
with $r/2$ in the role of his $r$, one obtains $r/2\mid (q-1)(q^{n/2}+1)$.
Since $G_{0}$ acts transitively on the blocks of $\mathcal{D}$ containing $0$ and $\Omega_{n}^{-}(q) \trianglelefteq G_{0}$, the set of blocks of $\mathcal{D}$ containing $0$ is partitioned in $\Omega_{n}^{-}(q)$-orbits of equal length. Further, the  minimal primitive degree of the permutation representations of $%
\Omega_{n}^{-}(q)$ is a lower bound for any such orbits. Hence, $r$ is greater than, or equal to, the minimal degree of the primitive permutation representations of $%
\Omega_{n}^{-}(q)$.

Now, we are going to rule out this remaining case in a series of steps.

\begin{claim}
 $n=4$.   
\end{claim}
Assume that $n\geq 6$. Then $r \geq (q^{n/2}+1)(q^{n/2-1}-1)/(q-1)$ by \cite[Proposition 5.2.1
and Theorem 5.2.2]{KL}, and hence $n=6$, $%
r=2(q-1)(q^{3}+1)$ and $k=q^{2}+q+2$. Then $k\mid 16$ for $q>2$
since $k\mid vr$ and $(q^{2}+q+2,q^{2}-1)\mid 8$, which is not the case. Thus $q=2$, and hence $r=18$ and $%
k=8$, but this cannot occur by \cite[Corollary 1.3]{Mo} since $G\nleq
A\Gamma L_{1}(8)$. 

\begin{claim}
$r=(q^{2}+1)x$ with $x\mid 2(q-1)$, $x>1$ and $x$ even when $q$ is even, and $k=\frac{2(q^{2}-1)}{x}+1$. Further, if $B$ is any block of $\mathcal{D}$ containing $0$, then $U \leq G_{0,B} \leq N_{G_{0}}(U)$, where $U$ is a Sylow $p$-subgroup of $\Omega _{4}^{-}(q)$.  
\end{claim}
It follows from Claim 1 that $n=4$. Then $\Omega _{4}^{-}(q)\cong PSL_{2}(q^{2})$, $r/2 \mid (q-1)(q^{2}+1)$, and hence either $G_{0,B}$ contains a Sylow $%
p $-subgroup of $\Omega _{4}^{-}(q)$, or $\Omega
_{4}^{-}(q)_{B}$ contains a subgroup of index $2$ of a Sylow $%
2$-subgroup of $\Omega _{4}^{-}(q)$ for $q$ is even. Suppose that the latter occurs. Then $q=2$ by \cite[Table 8.17]{BHRD}, and hence $r=10$ and $k=4$ since $r>q^{2}\sqrt{2}$. However, this is impossible by \cite[Lemma 6.8]{Mo}. Thus $G_{0,B}$ contains a Sylow $p$-subgroup of $\Omega _{4}^{-}(q)$, say $U$. Hence, either $U \leq G_{0,B} \leq N_{G_{0}}(U)$, or $\Omega _{4}^{-}(q) \leq G_{0,B}$ again by \cite[Table 8.17]{BHRD}. The latter implies
\begin{equation}\label{pomognime}
q^{2}\sqrt{2}<r \leq \left\vert \Gamma O_{4}^{-}(q):\Omega _{4}^{-}(q) \right\vert =2(2,q-1)(q-1)\log_{p}q
\end{equation}
since $G_{0}\leq \Gamma O_{4}^{-}(q)$, and we reach a contradiction. Thus $U \leq G_{0,B} \leq N_{G_{0}}(U)$, and hence $q^{2}+1$ divides $\left\vert \Omega _{4}^{-}(q):\Omega _{4}^{-}(q)_{B}\right\vert$ and hence $r$. Therefore,  $r=(q^{2}+1)x$ with $x\mid 2(q-1)$, $x>1$ and $x$ even when $q$ is even since $r/2 \mid (q-1)(q^{2}+1)$ and $r>q^{2}\sqrt{2}$, and hence $k=\frac{2(q^{2}-1)}{x}+1$.

\begin{claim}
$r=2(q^{2}+1)$ and $k=q^{2}$. Further, $B$ is a $GF(p)$-subspace of $V$, and $\Omega_{4}^{-}(q)_{B} \cong U:H$, where $H \cong Z_{\frac{q^{2}-1}{(q-1,2)\theta}}$ with $\theta \mid (2,q+1)$.  
\end{claim}
Assume that $T_{B}=1$. Then $G_{B}$ is isomorphic to a subgroup of $G_{0}$ (of index $r/k$), and hence of $\Gamma O_{4}^{-}(q)$, by Lemma \ref{cici}. Further, $U \leq G_{0,B}<G_{B}$ and $U$ is non-normal in $G_{B}$ since $G_{B}$ acts transitively on $B$. It follows that $G_{B}$ contains a copy of $\Omega _{4}^{-}(q)$ as a normal subgroup. Clearly, $\Omega _{4}^{-}(q)$ does not fix any point of $B$, otherwise we reach a contradiction by (\ref{pomognime}). Hence, $k$ is greater than, or equal to, the minimal degree of the non-trivial primitive permutation representations of $\Omega _{4}^{-}(q)$, which is $q^{2}+1$ or $6$ according as $q \neq 3$ or $q=3$, respectively. In the former case, one has $q^{2}+1\leq k=\frac{2(q^{2}-1)}{x}+1\leq q^{2}$, as $x>1$, and which is a contradiction. Thus $q=3$, and hence $(r,k)=(40,5)$ since $r/k$ must be an integer, whereas $k \geq 6$.

Assume that $T_{B}\neq 1$. Then either $k=p^{t}$, or $k=2p^{t}$ with $p$ odd by Corollary \ref{p2}. The latter implies $k$ even, whereas $k=\frac{2(q-1)}{x}(q+1)+1$ is odd for $p$ odd. Thus $k=p^{t}$, and hence $p^{t}-1=\frac{2(q-1)}{x}(q+1)$. The latter implies $\frac{2d}{n} \mid t$ since $q=p^{d/n}$, and hence $t=\frac{2d}{n}$ since $t<\frac{4d}{n}$ being $\mathcal{D}$ non-trivial. Then $x=2$, and hence $r=2(q^{2}+1)$ and $k=q^{2}$. Moreover, $
\Omega_{4}^{-}(q)_{B} \cong U:H$, where $H \cong Z_{\frac{q^{2}-1}{(q-1,2)\theta}}$ with $\theta \mid (2,q+1)$.

\begin{claim}
$q=3$, $r=20$ and $k=9$, and $B$ corresponds to a line of $PG_{3}(3)$ tangent to a $G_{0}$-invariant elliptic quadric $\mathcal{O}$.   
\end{claim}
The group $G_{0}$ acts 
$2$-transitively on the set $\mathcal{O}$ of $q^{2}+1$ singular $1$%
-dimensional subspaces (which is an elliptic quadric of $PG_{3}(q)$). Then $B$ must contain a
non-zero singular vector since $\mathcal{D}$ is a $2$-design and $G_{0}$ acts flag-transitively on $\mathcal{D}$. Assume that $B$ intersects at least two points of 
$\mathcal{O}$ each of them in a non-zero vector, then $k-1\geq q^{2}$ since $%
U$ fixes a point of $\mathcal{O}$ and acts regularly one the $%
q^{2} $ remaining ones. Thus $q^{2}+1\leq k=\frac{2(q^{2}-1)}{x}+1\leq q^{2}$, as $x>1$%
, and we reach a contradiction. Then $B$ intersects exactly one point of $%
\mathcal{O}$ in a non-zero vector, namely the point fixed by $U$, say $%
\left\langle u\right\rangle $. 

The actions of $\Omega _{4}^{-}(q) \cong PSL_{2}(q^{2})$ on $PG_{3}(q)\setminus \mathcal{O}$ and on
the set of Baer sublines $PG_{1}(q)$ of $PG_{1}(q^{2})$ are equivalent by \cite[Theorems
15.3.10(i) and 15.3.11]{Hir}. Thus each $U$-orbit on $PG_{3}(q)\setminus 
\mathcal{O}$ is of length $q$ or $q^{2}$, and hence each $U$ orbit on the set of
non-singular vectors is of length divisible by $q$. Thus $q\mid k-\left\vert B\cap
\left\langle u\right\rangle \right\vert $ and so $\left\langle u\right\rangle \subset B$ since $k=q^{2}$. Hence, each $U$-orbit on  $B \setminus \left\langle u\right\rangle$ is of length $q$ since $\left\vert B \setminus \left\langle u\right\rangle\right\vert =q^{2}-q$. 
Then there is $z \in B \setminus \left\langle u\right\rangle$ such that $\Omega_{4}^{-}(q)_{z,B}$ is a subgroup of order divisible by $q\frac{q+1}{(2,q+1)}$ of a Frobenius subgroup $F_{q^{2}\frac{q^{2}-1}{(2,q-1)}}$ of $\Omega _{4}^{-}(q)$ since $
\Omega_{4}^{-}(q)_{B} \cong U:H$, where $H \cong Z_{\frac{q^{2}-1}{(q-1,2)\theta}}$ with $\theta \mid (2,q+1)$. If $\frac{q+1}{(2,q+1)}>2$, then $U \leq \Omega_{4}^{-}(q)_{z,B} \leq \Omega_{4}^{-}(q)_{\left\langle z\right\rangle}$ by \cite[Proposition III.17.3]{Pass}, whereas each $U$-orbit on $PG_{3}(q)\setminus \mathcal{O}$ is of length $q$ or $q^{2}$. Thus $\frac{q+1}{(2,q+1)}=2$, and hence $q=3$, $r=20$ and $k=9$. In this case, $B$ corresponds to a line $\ell$ of $PG_{3}(3)$ tangent to the elliptic quadric $\mathcal{O}$. 

\begin{claim}
The final contradiction.    
\end{claim}
Choose a point $P$ on $\ell$ not on $\mathcal{O}$. There are exactly four tangents to $\mathcal{O}$ through $P$ including $\ell$ by \cite[Theorems
15.3.9 and Table 15.7(d)]{Hir}. The intersection points of these lines with $\mathcal{O}$ form a conic lying in the polar plane of $P$ with respect the quadratic form defining $\mathcal{O}$ (see the proof of \cite[Theorem
15.3.10(i)]{Hir}) and $\Omega_{4}^{-}(3)_{P} \cong PGL_{2}(3)$ acts transitively on the conic by \cite[Theorem 15.3.11]{Hir}. Thus, there are exactly four elements of $B^{\Omega_{4}^{-}(3)}$ through any non-singular vectors of $V$ corresponding to $P$, but this contradicts $\lambda=2$. Thus this case is ruled out, and hence the proof is completed.
\end{proof}

\bigskip

\subsection{Aschbacher's theorem.} \label{Aschbacher}Recall that $G_{0} \leq \Gamma$, where $\Gamma =N_{\Gamma L_{n}(q)}(X)$ and $X$ is any of the classical groups $SL_{n}(q)$, $Sp_{n}(q)$, $SU_{n}(q^{1/2})$ or $\Omega^{\varepsilon}(q)$. The case where where $X \leq G_{0}$ has been settled in Proposition \ref{C8}, hence, in the sequel, we assume that $G_{0}$ does not contain $X$. Now, according to \cite{As}, one of the following holds: 
\begin{enumerate}
    \item[(I)] $G_{0}$ lies in a maximal member of one the geometric classes $\mathcal{C}_{i}$ of $\Gamma$, $i=1,...,8$;
    \item[(II)] $G_{0}^{(\infty)}$ is a quasisimple group, and its action on $V_{n}(q)$ is absolutely irreducible and not realisable over any proper subfield of $GF(q)$.
\end{enumerate}
Description of each class  $\mathcal{C}_{i}$, $i=1,...,8$, can be found in \cite[Chapter 4]{KL}.   
\bigskip

We are going to analyze cases (I) and (II) in separate sections.

\section{The geometric group case}\label{geom}
In this section, we assume that $G_{0}$ lies in a maximal member of one the geometric
classes $\mathcal{C}_{i}$ of $\Gamma$, $i=1,...,8$. The case where $i=8$ has been established in Proposition \ref{C8}, hence we may assume $i<8$ throghout the remainder of the paper. Our aim is to prove the following theorem.

\bigskip

\begin{theorem}\label{MaxGeom}
Let $\mathcal{D}$ be a non-trivial $2$-$(v,k,2)$ design admitting a flag-transitive point-primitive automorphism group $G=TG_{0}$. If $G_{0}$ lies in a maximal geometric subgroup of $\Gamma$, then $(\mathcal{D},G)$ is  as in Examples \ref{sem1dimjedan}--\ref{hyperbolic}.
\end{theorem}

\bigskip

We are going to analyze the cases where $G_{0}$ lies in a maximal member of class $\mathcal{C}_{8}$, $\mathcal{C}_{i}$ with $1 \leq i <8$ and $i \neq 6$, and $\mathcal{C}_{6}$ in separate sections.

\bigskip

\subsection{The case where $G_{0}$ lies in a maximal $\mathcal{C}_{i}$-subgroup of $\Gamma$ with $i\neq 6,8$}

\begin{lemma}\label{maxC3}
If $G_{0}$ lies in a maximal $\mathcal{C}_{i}$-subgroup of $\Gamma$ with $i=1,2,3$, then $i=2$ and one of the following holds:
\begin{enumerate}
    \item[(i)] $\mathcal{D}$ is the $2$-$(25,4,2)$ design as in Line 1 of Table \ref{smol};
    \item[(ii)] $\mathcal{D}$ is the symmetric $2$-$(16,6,2)$ design as in \cite[Section 1.2.1]{ORR};
    \item[(iii)] $\mathcal{D}$ is the $2$-$(81,6,2)$ design as in Example \ref{hyperbolic}.
\end{enumerate} 
\end{lemma}
\begin{proof}
The group $G_{0}$ does not lie maximal member of type $\mathcal{C}_{1}$
since $G_{0}$ acts irreducibly on $V_{n}(q)$. Moreover, by the definition of $q$, $G_{0}$ does not lie in a
member of $\mathcal{C}_{3}$. Hence, assume that $G_{0}$ lies in a maximal $\mathcal{C}_{2}$-subgroup of $\Gamma$. Then $G_{0}$ preserves a sum decomposition of $V=V_{a}(q)\oplus
\cdots \oplus V_{a}(q)$ with $n/a\geq 2$, and hence $\bigcup_{i=1}^{n/q}V^{\ast}_{a}(q)$ is a union of $G_{0}$-orbits. Then $\frac{r}{2}\mid \frac{n}{a}(q^{a}-1)$ by Lemma \ref{sudbina}(2) since the size of $\bigcup_{i=1}^{n/q}V^{\ast}_{a}(q)$ is $\frac{n}{a}(q^{a}-1)$. If $n/a\geq 3$, then $n/a=3$
and $q^{a}\leq 13$, or $n/a=4$ and $q^{a}\leq 4$, or $n/a=5,6$ and $q^{a}=2$ since $\sqrt{2}q^{n/2}<r\leq 2\frac{n}{a}(q^{a}-1)$. Now, exploiting $r \mid 2\frac{n}{a}(q^{a}-1)$, $r>\sqrt{2}q^{n/2}$, $k\leq r$ and $k \mid vr$, one can see that no cases arise. Thus $n/a=2$, and the assertion follows from Lemma \ref{fato}.
\end{proof}

\bigskip

\begin{lemma}\label{tensred}
If $G_{0}$ lies in a maximal $\mathcal{C}_{i}$-subgroup of $\Gamma$ with $i=4,5,7$, then one of the following cases is
admissible:
\begin{enumerate}
\item $G_{0}\leq (GL_{2}(2)\times GL_{2}(2)).Z_{2}$, $k=r=6$;

\item $G_{0}\leq (GL_{2}(3)\circ GL_{2}(3)).Z_{2}$, $k=6$ and $r=32$;

\item $G_{0}\leq GL_{3}(2)\times GL_{2}(2)$, $k=4$ and $r=42$;

\item $G_{0}\leq GL_{3}(3)\circ GL_{2}(3)$, $k=8$ and $r=208$.
\end{enumerate}
\end{lemma}

\begin{proof}
Assume that $G_{0}$ lie in a maximal member of type $\mathcal{C}_{4}$ or $\mathcal{C}_{7}$ of $%
\Gamma L_{n}(q)$. Then either $V=V_{d}(p)=V_{a}(p)\otimes V_{c}(p)$, $d=ac$ with $a\leq c \leq 2$, and $G_{0} \leq GL_{a}(p)\circ GL_{c}(p)$ in its natural action on $V$, or $V=V_{d}(p)=V_{a}(p)\otimes \cdots \otimes V_{c}(p)$, $d=a^m$, and $G_{0}\leq N_{GL_{d}(p)}(GL_{a}(p)\circ \cdots GL_{a}(p))$, respectively.
Assume that the latter occurs with $m\geq 3$. Arguing as in \cite[Lemma 3.4]{Lieb}, with $r/2$ in the role of $r$,
we obtain $r\mid \frac{2\left( p^{a}-1\right) ^{m}}{(p-1)^{m-1}}$ and $r>\sqrt{2}p^{a^{m}/2}$,
implying $2p^{am}>\frac{2\left( p^{a}-1\right) ^{m}}{(p-1)^{m-1}}%
>\sqrt{2}p^{a^{m}/2}$. Hence $2am>a^{m}$, and so $a=2$ and $m=3$.
Therefore, $v=p^{8}$ and $r\mid 2\left( p-1\right) \left( p+1\right) ^{3}$.
Combining this with $r\mid 2(p^{8}-1)$, we have that $r\mid 2(p^2-1)((p+1)^2,(p^4+1)(p^2+1))$.
Note that 
$((p+1)^2,(p^4+1)(p^2+1))$ divides $4$.
Thus, $r$ divides $8(p^2-1)$, 
following from $r>\sqrt{2}p^{4}$, we obtain $8(p^2-1)>\sqrt{2}p^{4}$, 
and so $p=2$ and $r=24$, but this contradicts $r \mid 2(2^{8}-1)$. Thus $m=2$, and hence all the cases where $G_{0}$ lie in a maximal member of type $\mathcal{C}_{4}$ or $\mathcal{C}_{7}$ of $%
\Gamma L_{n}(q)$ reduce to $V=V_{d}(p)=V_{a}(p)\otimes V_{c}(p)$, $a\leq c \leq 2$%
, where $d=ac$, $a\geq c\geq 2$ 
and $G_{0}\leq N_{GL_{d}(p)}(GL_{a}(p)\circ GL_{c}(p))$.

An argument similar to that of \cite[Lemma 3.5]{Lieb} leads to 
$r\mid \frac{2(p^{a}-1)(p^{c}-1)}{(p-1)}$ and $r>\sqrt{2}p^{ac/2}$,
forcing $2p^{a+c}>\frac{2(p^{a}-1)(p^{c}-1)}{(p-1)}>\sqrt{2}p^{ac/2}$,
and hence $ 0\geq  a(c-2)-2c\geq c(c-2)-2c$
and so $2\leq c\leq 4$.

If $c=4$, then $a=4$ and $p=2$.
Combining  $r\mid 2\cdot (2^4-1)^2$  and $r\mid 2\cdot (2^{16}-1)$, 
we have $r\mid 30$, a contradiction.

If $c=3$ then either $a\leq 4$, or $(a,p)=(6,2)$, $(5,2)$ or $(5,3)$. 
In case $(a,p)=(6,2)$, we have $r\mid 126$ since $r \mid 2(2^{18}-1)$, contradicting $r^2>2v$;
in case $(a,p)=(5,2)$ and $(5,3)$, 
we have $(r,k)=(434,152)$ and $(6292,4562)$, respectively,
but $k$ does not divide $vr$.
Thus, $(c,a)=(3,3)$ or $(3,4)$.

If $(c,a)=(3,4)$, then $v=p^{12}$,
$r=\frac{2(p^{4}-1)(p^{2}+p+1)}{\theta}$
and $k=\allowbreak \theta \left( p^{6}-p^{5}+p^{3}-p+1\right) +1$, 
where either $\theta =1$, or $\theta =2$, $p=3$ and $(r,k)=(1040,1023)$ since $k<r$, $r$ is even. However, the latter is ruled out since $k\nmid vr$. Thus $\theta =1$, and hence $r=2(p^{4}-1)(p^{2}+p+1)$ and $k=p^{6}-p^{5}+p^{3}-p+2$.

Note that $(v,k)=(p,2)$.
If $p=2$, then $k=40$ and $r=210$, but this contradicts Corollary \ref{p2}. 
Thus $p\neq 2$, $(v,k)=1$ and so $k\mid r$. Since $(k,p-1) \mid 2$ it follows that $2r \geq (p-1)k$, and hence $(p,r,k)=(3,2080,512)$ or $(5,38688,12622)$, but these contradict $k\mid r$.

If $(c,a)=(3,3)$, then $v=p^{9}$,
$r=\frac{2(p^{3}-1)(p^{2}+p+1)}{\theta }$ and $%
k=\allowbreak \frac{\theta (p^{6}+p^{3}+1)}{p^{2}+p+1}+1$. 
Since $k$ is an integer and 
$\left(p^{6}+p^{3}+1,p^{2}+p+1\right)=\left(p^{2}+p+1,3\right)$, 
it follows that $p^{2}+p+1\mid 3\theta$,
and so 
\[
\frac{(p^{2}+p+1)^{2}}{(p^{2}+p+1,3)^{2}}%
\leq 
\theta ^{2}<\frac{2(p^{3}-1)(p^{2}+p+1)^{2}}{p^{6}+p^{3}+1}\textit{,}
\]
which does not have admissible solutions.

Assume that $c=2$. 
Then $r\mid 2(p^a-1)(p+1,p^a+1)$.
If $a$ is even, then $r$ divides $4(p^a-1)$, and hence $a=2$ and $(v,r,k)=(16,6,2)$ or $(81,32,6)$ by Lemma \ref{fato} since $a\geq 2$ and $c=2$. Thus, we obtain (1) and (2), respectively.

If $a$ is odd, then $r=\frac{2(p^{a}-1)(p+1)}{\theta}$ and $k=\theta \frac{%
p^{a}+1}{p+1}+1$, and hence $(\theta,p)=1$ and $\theta<\sqrt{2}(p+1)$ since $r>\sqrt{2}p^{a}$ and $r$ is even. It follows from Lemma \ref{cici} and Corollary \ref{p2} that either $k=p^{t}$ for some $t\geq 1$, or $k=2p^{t}$ with $p$ odd, or $k \mid r$. 

Assume that $k=p^{t}$ with $t\geq 1$. Then $\frac{p^{a}+1}{p+1} \mid p^{t}-1$, and hence $k=v$ by \cite[Theorem 5.2.14]{KL} for $(a,p)\neq (3,2)$, but this contradicts $\mathcal{D}$ non-trivial. Thus $(a,p)=(3,2)$, and hence $r=42$ and $k=4$, and we obtain (3). 

Assume that $k=2p^{t}$ with $p$ odd. Then $p \mid \theta +1$ and hence either $\theta =p-1$, or $2p-1 \leq \theta<\sqrt{2}(p+1)$. The latter implies $p=3$ and $\theta=5$, and hence $2\cdot3^{t}=k=\frac{3^{a}5+9}{4}$, which has no admissible solutions. Thus $\theta=p-1$, and hence $p=2$ since $\theta \mid k-1$ and $k-1=2p^{t}-1$, whereas $p$ is odd.

Assume that $k\mid r$. Then there is $e\geq 1$ such that%
\begin{equation}
e \left( \theta\frac{p^{a}+1}{p+1}+1\right) =\frac{2(p^{a}-1)(p+1)}{\theta}\textit{.}  \label{Zelen}
\end{equation}%
Hence, $e<\frac{2(p+1)^{2}}{\theta^{2}}$. Rewriting (\ref{Zelen}), one obtains%
\begin{equation}
e\theta\left( (p^{a}-1)\theta+2\theta+p+1\right)  =2(p+1)^{2}(p^{a}-1)\textit{,}
\end{equation}%
and hence $p^{a}-1\mid e\theta(2\theta+p+1)$ which leads to 
\begin{equation}
p^{a}-1<\frac{2(p+1)^{2}}{\theta}(2\theta+p+1)=4(p+1)^{2}+\frac{%
2(p+1)^{3}}{\theta}  \label{sloboda},
\end{equation}%
and hence $a=3$ or $(a,p)=(5,2)$.

If $(a,p)=(5,2)$, then (\ref{Zelen}) becomes $e\theta(11\theta+1)=20$, and no admissible solutions occur. Thus $a=3$, and from (\ref{Zelen}) we obtain 
\begin{equation}
e\theta\left( (p^{2}-p+1)\theta+1\right) 
=2(p^{3}+1)(p+1)-4(p+1),\label{Zelen3}
\end{equation}%
and so 
$p^{2}-p+1\mid e\theta+4(p+1)$, 
implying
\[
p^{2}-5p-3\leq e\theta <\frac{2(p+1)^{2}}{\theta}.
\]
If $p\geq 23$, then $\theta <\frac{2(p+1)^{2}}{p^{2}-5p-3}<3$.
Assume that $\theta=1$. Then $k=p^2-p+2\mid 10p+2$, as $k\mid r$, forcing $p\leq 11$, contradiction.
Similarly, for the case $\theta=2$, we have that $k=2p^2-2p+3 \mid p^2-8p-2$, impossible.

If $p\leq 19$, then by (\ref{Zelen3}) and $\theta <\sqrt{2}(p+1)$, we have that $(p,\theta)=(3,1)$ or $(11,1)$,
and so $(k,r)=(8,208)$ or $(112,31920)$. The latter is ruled out since it violates $k\mid vr$, the former corresponds to (4).

Assume that $G_{0}$ lies in a maximal member of type $\mathcal{C}_{5}$ of $%
GL_{d}(p)$. Then $G_{0}\leq N_{GL_{d}(p)}(GL_{n}(q_{0}))$ with $q=q_{0}^{h}$
and $h>1$; but this normalizer lies in in a subgroup of $GL_{h}(q_{0})\circ
GL_{n}(q_{0})$ of $GL_{hn}(q_{0})\leq GL_{d}(p)$ and hence $G_{0}$ lies in a
maximal member of type $\mathcal{C}_{4}$ or $\mathcal{C}_{7}$ of $GL_{d}(p)$, which is not the case by the abve argument since $q=q_{0}^{h}$
with $h>1$. This completes the proof.
\end{proof}

\begin{lemma}\label{C1C2C3C4C5C7}
If $G_{0}$ lies in a maximal $\mathcal{C}_{i}$-subgroup of $\Gamma$ with $1\leq i<8$ and $i\neq 6$, then one of the following holds:

\begin{enumerate}
   \item[(i)] $\mathcal{D}$ is the $2$-$(25,4,2)$ design as in Line 1 of Table \ref{smol};
    \item[(ii)] $\mathcal{D}$ is the symmetric $2$-$(16,6,2)$ design as in \cite[Section 1.2.1]{ORR};
    \item[(iii)]$\mathcal{D}$ is one of the two $2$-$(64,4,2)$ designs as in Example \ref{tens};
    \item[(iv)] $\mathcal{D}$ is the $2$-$(81,6,2)$ design as in Example \ref{hyperbolic}.
\end{enumerate}
\end{lemma}

\begin{proof}
Assume that $G_{0}$ lies in a maximal $\mathcal{C}_{i}$-subgroup of $\Gamma L_{n}(q)$ with $1\leq i<8$ and $i\neq 6$. If $i\leq 3$, then (i), (ii) and (iv) follow from Lemma \ref{maxC3}. If $i=4,5,7$, then $G_{0}$, $r$ and $k$ are as in Lemma \ref{tensred}. If $G_{0}$, $r$ and $k$ are as in (1) or (2) of Lemma \ref{tensred}, then $\mathcal{D}$ is as in (ii) or (iii), respectively, by Lemma \ref{fato}. Now, it is not difficult to see that such $2$-designs do admit flag-transitive groups $G=T:G_{0}$ with $G_{0}$ as in (1) or in (2) of Lemma \ref{tensred}.

Suppose that $G_{0}$, $r$ and $k$ are as in (3) or (4) of Lemma \ref{tensred}. In the both cases, $\mathcal{S}=(u\otimes w)^{M}$, where $M=GL_{a}(p)\circ GL_{c}(p)$, has length $\frac{(p^{a}-1)(p^{c}-1)}{p-1}$,
which is equal to $r/2$. Then $%
\mathcal{S}$ is a $G_{0}$-orbit since the length of each $G_{0}$-orbit is
divisible by $r/2$ by Lemma %
\ref{sudbina}(2). Hence, $\left\vert B\cap \mathcal{S}\right\vert =1$ and $%
\left\vert B\cap V^{\ast }\setminus \mathcal{S}\right\vert =k-2$, where $B$ is any block of $\mathcal{D}$ containing $0$.

Assume $G_{0}$, $r$ and $k$ are as in (3) of Lemma \ref{tensred}. Hence, $p=2$. Then $B$ is $2$-dimensional subspace of $V$ by Corollary \ref{p2} since $k\equiv 0\pmod{4}$ and $r\equiv 2\pmod{4}$, and hence (iii) follows by using \texttt{GAP} \cite{GAP}.

Finally, assume that $G_{0}$, $r$ and $k$ are as in (4) of Lemma \ref{tensred}. Hence, $p=3$ and $V=X\otimes Y$ with $\dim X=3$ and $\dim Y=2$. Then $T_{B}=1$ since $k=8$. Further, either $H \leq G_{0}$ with $H=SL_{3}(3)\otimes 1$, or $Z_{13}\times SD_{16} \leq G_{0} \leq F_{39}\times GL_{2}(3)$ by \cite{At} since $G_{0}$ is a subgroup of $GL_{3}(3)\circ GL_{2}(3)$ of order divisible by $13$. In the latter case, $G_{0}$ preserves the strcture of a $2$-dimensional $GF(3^{3})$-structure of $V$, namely $V=GF(3^{3})\otimes Y$, but this contradicts our assumption on the maximality of $q$ (see Section \ref{irr}). Thus, $H \leq G_{0}$ with $H=SL_{3}(3)\otimes 1$. 

There is $3$-element $\alpha$ in $H$ fixing $B$ pointwise since $r=208$, $k=8$ and a Sylow $3$-subgroup of $H$ is of order $27$. Thus, $\left\vert Fix(\alpha)\right\vert \geq 9$. Then $\alpha =\alpha_{1}\otimes 1$ and  $Fix(\alpha)=Fix(\alpha_{1})\otimes Y$ with $\dim Fix(\alpha_{1})=2$ since $\left\vert B\cap \mathcal{S}\right\vert =1$
and $\left\vert B\cap V^{\ast }\setminus \mathcal{S}\right\vert =6$. Moreover, for the same reason, $B\subset Fix(\alpha_{1})\otimes Y$ is the unique element of $(Fix(\alpha_{1})\otimes Y)^{G_{0}}$ containing $B$. Now, $H$ induces a group on $X$ acting transitively on the thirteen $2$-dimensional subspaces of $X$. Thus $H$-orbit of $Fix(\alpha_{1})\otimes Y$ is of length $13$, and the actions of $H$ on $(Fix(\alpha_{1})\otimes Y)^{H}$ and on the set of lines of $PG_{2}(3)$ are equivalent.

Let $B\cap \mathcal{S}=\{x\otimes y\}$ with $x \in X^{\ast}$ and $y\in Y^{\ast}$, then $B,B^{-1} \subset Fix(\alpha_{1})\otimes Y$ with $\{0,x\otimes y\}\subset B$ and $\{0,-x\otimes y\}\subset B^{-1}$. Now, there are exactly $4$ elements of $(Fix(\alpha_{1})\otimes Y)^{H}$ containing $\left\langle x\otimes y \right\rangle$, say $(Fix(\alpha_{1})\otimes Y)^{h_{i}}$ with $h_{i} \in H$ for $i=1,2,3,4$, since the actions of $H$ on $(Fix(\alpha_{1})\otimes Y)^{H}$ and on the set of lines of $PG_{2}(3)$ are equivalent. So, $B^{h_{i}},B^{-h_{i}}$ with $i=1,2,3,4$ are $8$ elements of $B^{H}$, and hence of $B^{G_{0}}$, intersecting $\left\langle x\otimes y \right\rangle$ both in $0$ and in a non-zero vector, but this contradicts $\lambda=2$.
\end{proof}

\subsection{The case where $G_{0}$ lies in a maximal $\mathcal{C}_{6}$-subgroup of $\Gamma $}





\begin{lemma}\label{C6}
If $G_{0}$ lies in a maximal $\mathcal{C}_{6}$-subgroup of $\Gamma $, then one of the following holds:
\begin{enumerate} 
    \item[(i)] $\mathcal{D}$ is one of the $2$-designs as Example \ref{sem1dimjedan}(1) for $q=p=7$ and $n=d=2$;
    \item[(ii)] $\mathcal{D}$ is one of the $2$-designs as in Example \ref{various};
    \item[(iii)] $\mathcal{D}$ is the $2$-$(81,6,2)$ design as in Example \ref{hyperbolic}.
    \item[(iv)] $\mathcal{D}$ is the $2$-$(81,9,2)$ design as in Example \ref{D8Q8}.
\end{enumerate}
\end{lemma}

\begin{proof}
Assume that $G_{0}$ lies in a maximal $\mathcal{C}_{6}$-subgroup of $\Gamma $. Hence, $G_{0}$ lies in the normalizer in $GL_{n}(q)$
of a symplectic-type $s$-group with $s\neq p$. As show in the proof of \cite[%
Section 11]{As}, we may assume that $G_{0}$ contains the $s$-group,
otherwise lies in some other families $\mathcal{C}_{i}$ which we have been settled above. Now, arguing as \cite[Lemmas 3.7 and 3.8]{KL}
with $r/2$ in the role of $r$ and $v^{1/2}/\sqrt{2}$ and $v^{1/2}$, it is
easy to see that only the following cases are admissible:

\begin{enumerate}
\item $s=3$, $G_{0}\leq 3^{1+2}.Sp_{2}(3)\cdot 2$, $q=4$, $n=3$ and $r=18$;

\item $s=2$, $G_{0}\leq Z_{q-1}\circ 2_{\pm }^{1+2m}\cdot O_{2m}^{\pm }(2)\cdot
\log _{p}q$, with $n=2^{m}$, and one of the following holds:

\begin{enumerate}
\item $m=1$ and either $q=p^{2}$ and $r=4(q-1)/\theta $ with $\theta =1,2$,
or $q=p$;

\item $m=2$ and either $q=p^{2}$ and $r=4(q^{2}-1)/\theta $ with $\theta
=1,2 $, or $q=p$ and $r \mid 20(p^{2}-1)$.

\item $m=q=3$ and $r=320$.
\end{enumerate}
\end{enumerate}

Case (1) implies $k=8$, $B$ is a $3$-dimensional $GF(2)$-space and $%
G_{0}\leq \Gamma L_{1}(2^{6})$ in (1) by \cite[Corollary 1.3]{Mo}, a contradiction. In case (2c) one has $v=3^{8}$ and $k=42$, which cannot occur since $k\nmid vr$. Also, cases (2a)--(2b) with $q=p^{2}$ cannot occur by Lemma \ref{fato}, whereas (2a) with $q=p$ is (i).

Suppose that case (2b) holds with $n=4$, $q=p$ and $r=\frac{20(p^{2}-1)}{\theta}$ for some $\theta \geq 1$. Then $k=\frac{(p^{2}+1)\theta}{10}+1$, and hence $\theta \leq 14 $ since $k\leq r$. If $5$ divides $\theta$, it follows that $r\mid 4(p^{2}-1)$, and hence $p=3$ and $\mathcal{D}$ is a $2$-$(81,6,2)$ design as in Example \ref{hyperbolic} by Lemma \ref{fato} since $s\neq p$ (with $D_{8} \circ D_{8} \unlhd G_{0}$). Thus, we obtain (ii). 

Assume that $5$ does not divide $\theta$, then $p^{2} \equiv -1 \pmod{5}$. Let $B$ be any block of $\mathcal{D}$ containing $0$. Then either $k=p^{t}$ and $B$ is $t$-dimensional subspace of $AG_{d}(p)$ with $t=2,3$, or $k=2p^{t}$ with $t=1,2,3$, or $k \mid r$ by Corollary \ref{p2} since $k<v=p^{4}$.

If $k=p^{t}$ with $t=2,3$, then $p^{t}-1=\frac{p^{2}+1}{5(2,p-1)}\frac{%
\theta }{(2,p)}$ since $p^{2}\equiv -1\pmod{5}$. Then $\frac{p^{2}+1}{%
5(2,p-1)}=1$ since $\left( \frac{p^{2}+1}{(2,p-1)},p^{t}-1\right) =1$ and
hence $p=3$ since $p\neq s=2$. Then $\theta =8
$ and $(v,k,r)=(81,9,20)$ since $\theta \leq 14$. Thus, $B$ is a $2$-dimensional subspace of $V$. Further, $R.Z_{5} \unlhd G_{0}\leq R.S_{5}$ with $R \cong D_{8}\circ Q_{8}$ since $r \equiv 0 \pmod{5}$. We deduce from \cite{At} and \cite{Hup0} that there is a unique conjugacy $GL_{4}(3)$-class of subgroups isomorphic to $R.Z_{5}$ and each of these acts transitively on $V^{\ast}$. Thus, we may assume $R.Z_{5}$ is as in \cite[pp. 144--145]{Hup0} (the one denoted by $\mathfrak{G}^{0}_{3^{4},3}$). Now, it is not difficult to see that $R.Z_{5}$ partitions the $2$-dimensional subspaces of $V$ into four orbits of length $10$, $20$, $20$ and $80$, and these are also $R.F_{10}$-orbits. Then $B$ lies in one of the two $R.F_{10}$-orbits of length $20$, and hence $\mathcal{D}$ is isomorphic to the $2$-$(81,9,20)$ design as in Example \ref{D8Q8}. Therefore, we obtain (iii).    

If $k=2p^{t}$ with $p$ odd and $t=1,2,3$, then $2p^{t}-1=\frac{p^{2}+1}{10}%
\theta $. Now, $\left( \frac{p^{2}+1}{10},p^{t}-1\right) $ divides $5$ or $3$
for $t=1,3$ or $2$, respectively. Then either $\frac{p^{2}+1}{10}\mid 5$ or $%
\frac{p^{2}+1}{10}\mid 3$, respectively. Then $p=3$ or $7$. However, none of
them fulfills $2p^{t}-1=\frac{p^{2}+1}{10}\theta $ with $\theta \leq 14$.

If $r\mid k$, from $r=\frac{20(p^{2}-1)}{\theta }$ and $k=\frac{p^{2}+1}{10}%
\theta +1$ we obtain $\theta ^{2}r=20\left( 10k-10-2\theta \right) $ and
hence $k\mid 40(5+\theta )$. If $r=k$, we know $r=k=6$ and hence $\theta =5$%
, whereas $(\theta ,5)=1$. Thus $2k\leq r$, and hence $\theta \leq 9$. Easy computations shows that the admissible parameters triple $(v,k,r)$ is $\left( 3^{4},5,40\right)$, $\left( 7^{4},6,960\right)$, $\left(
7^{4},16,320\right)$, $\left( 13^{4},35,1680\right)$, $\left(
13^{4},120,480\right)$, $\left( 17^{4},30,5760\right)$, or $\left(23^{4},160,3520\right)$. In each case, $R \unlhd G_{0}\leq Z_{p-1}\circ 2_{-}^{1+4}.S_{5}$ with $R \cong D_{8}\circ Q_{8}$ since $r \equiv 0 \pmod{5}$.

The cases $(v,k,r)=\left( 13^{4},35,1680\right)$ or  $\left( 17^{4},30,5760\right)$ are ruled out since $r$ does not divide the order of $Z_{p-1}\circ 2_{-}^{1+4}.S_{5}$ and hence $\left\vert G_{0}\right\vert$.

In the remaining cases, either $R\circ SL_{2}(5)\trianglelefteq G_{0}$ or $%
R:Z_{5}\trianglelefteq G_{0}\leq \left( Z_{p-1}\circ R\right) .F_{20}$. Moreover, $G_{B}$ is isomorphic to a subgroup of $G_{0}$ of index $r/k$ by
Lemma \ref{cici} since $(k,p)=1$. Actually, $G_{B}=G_{B,x}$ for some $x\in V$
by \cite[$(\beta )$ at p.9.]{KanLib} or \cite[Theorem 6.2.1]{Go} according as
the case $q=3$ and $R\circ SL_{2}(5)\trianglelefteq G_{0}$ does or does not
occur, respectively. Further, $x\notin B$ since $G_{B}$ acts transitively on 
$B$, which we may assume to be $0$ since $G$ acts point-transitvely on $\mathcal{D%
}$.

If $(v,k,r)=\left( 3^{4},5,40\right) $ or $\left( 13^{4},120,480\right) $, then $H\leq G_{B}$ with $H\cong Z_{5}$.
Then either $G_{B}\cap R\leq \left\langle -1\right\rangle $, or $R\leq G_{B}$ since
any $H$ acts on $R/\left\langle -1\right\rangle $ irreducibly. The former
implies that $16\mid \frac{r}{k}$, and we reach contradiction, the latter $%
16\mid k$ since each $R$-orbit on $V^{\ast }$ is of length divisible by $16$.

The cases $(v,k,r)=\left( 7^{4},6,960\right)$, $\left(
7^{4},16,320\right)$, or $\left( 23^{4},160,3520\right) $ are ruled out by \texttt{GAP} \cite{GAP}.

Finally, assume that $n=2$ and $G_{0}\leq M$, where $M$ is either $Z_{p-1} \circ D_{8}.Z_{2}$ or $Z_{p-1} \circ Q_{8}.S_{3}$. Then $M$ induces either $D_{8}$ or $S_{4}$ on $PG_{1}(p)$, respectively. the Greatest common divisor $\ell$ of the length of the $M$-orbits on$PG_{1}(p)$ is $2,4$ or $8$ in the former case, $6,8,12$, or $24$ in the latter depending on the residue of $q$ modulo $8$ and $3$ by \cite[Lemmas 8 and 10]{COT}. Thus, $r=\frac{2\ell(p-1)}{\theta }$ with $\ell$ as one of the above possible values and $\theta \geq 1$ by Lemma \ref{sudbina}(2), and hence $k=\frac{(p+1)\theta }{\ell}+1$. Then $\theta <\sqrt{2}\ell$ since $r>\sqrt{2}p$. Further, either $k=2p$ or $k \mid r$ by Corollary \ref{p2}. Easy computations lead to no admissible cases for $k=2p$. Thus $k \mid r$. and hence there is $1\leq j <\frac{2\ell^{2}}{\theta^{2}}$ such that
\begin{equation}\label{gg}
j\left( \frac{(p+1)\theta }{\ell}+1\right) =\frac{2\ell(p-1)}{\theta }\textit{.}    
\end{equation}
By using \texttt{GAP} \cite{GAP}, we determine all the admissible triples $(p^{2},k,r)$ corresponding to solutions of (\ref{gg}), and for each of these we make an exhaustive search to determine, up to isomorphism, the $2$-$(p^{2},k,2)$ designs admitting a flag-transitive automorphism group of the form $G=T:G_{0}$ with $G_{0} \leq M$. The output is that $\mathcal{D}$ is as Example \ref{sem1dimjedan}(1) for $q=p=7$ and $n=d=2$, or it is one of the $2$-designs as in Example \ref{various}.
\end{proof}

\bigskip

\begin{proof}[Proof of Theorem \ref{MaxGeom}]
The assertion immediately follows from Proposition \ref{C8} and Lemmas \ref{C1C2C3C4C5C7} and \ref{C6} according as $G_{0}$ lies in a maximal $\mathcal{C}_{i}$-subgroup of $\Gamma L_{n}(q)$ with $i=8$, or $1 \leq i <8$ and $i \neq 6$, or $i=6$, respectively.  
\end{proof}

\section{The case where $G^{(\infty)}_{0}$ is quasisimple}\label{quasi}
In this section, we investigate the case where $G^{(\infty)}_{0}$ is quasisimple, absolutely irreducible on $V=V_n(q)$, and not realisable over any proper subfield of $GF(q)$, which is the remaining case of Aschbacher's theorem to be analyzed in order to complete the proof of Theorem \ref{main} (see Section \ref{Aschbacher}).

\subsection{Transferring Liebeck's argument}\label{TLA} Set $L=G^{(\infty)}_{0}/Z(G^{(\infty)}_{0})$. Note that, the proof strategy used by Liebeck in \cite{Lieb} to classify the $2$-$(q^{n},k^{\prime},1)$ designs $\mathcal{D}^{\prime}$ (linear spaces) admitting flag transitive group $G=T:G_{0}$, with $V$, $n$, $q$, $G^{(\infty)}_{0}$ and $L$ having the same meaning as ours, is primarily to filter out the possible candidates for the group $G_{0}$ with respect the properties 
\[
r^{\prime}\mid (q^{n}-1,c_{1},...,c_{j},(q-1)\cdot|Aut(L)|)
\text{~and~}
r^{\prime}>q^{n/2}\textit{,}
\]
where $r^{\prime}=\frac{q^{n}-1}{k^{\prime}-1}$ and $c_{1},...,c_{j}$ are the lengths of the $G_{0}$-orbit on $V$, by \cite[Lemma 2.1]{Lieb}. In our context, bearing in mind that $r$ is even by Hypothesis \ref{Hippo}, we have
\begin{equation}\label{pasticciotto}
\frac{r}{2}\mid (q^n-1,c_{1},...,c_{j},(q-1)\cdot|Aut(L)|)
\text{~and~}
\frac{r}{2}>\frac{q^{n/2}}{\sqrt{2}}.
\end{equation}
by Lemma \ref{sudbina}. Hence, we may use Liebeck's argument with $r/2$ and $\frac{q^{n/2}}{\sqrt{2}}$ in the role of $r^{\prime}$ and $q^{n/2}$, respectively, to reduce our investigation to the cases where (\ref{pasticciotto}) is fulfilled. We provide in one exemplary case (namely, when $L$ is the alternating group) more guidance and proof details  to help the reader to control the transfer from the linear space case investigated by Liebeck to our case. For the remaining groups, that is when $L$ is a sporadic group, a group of Lie type in characteristic $p$, or a group of Lie type in characteristic different from $p$, the proof strategy, which relies on Liebeck's argument is similar. The corresponding admissible cases are contained in Tables \ref{tavspor}, \ref{tavLiecharp} and \ref{tavLiecrosscharnoPSL},\ref{tavLiecrosscharisoPSL}, respectively. The objective of this section is to  prove the following  result by using the above mentioned adaptation of Liebeck's argument. 

\bigskip

\begin{theorem}\label{qsabsirrnosub}
Let $\mathcal{D}$ be a non-trivial $2$-$(v,k,2)$ design admitting a flag-transitive point-primitive automorphism group $G=TG_{0}$. If $G^{(\infty)}_{0}$ is quasisimple, absolutely irreducible on $V=V_n(q)$, and not realisable over any proper subfield of $GF(q)$, then one of the following holds:
\begin{enumerate}
     \item $\mathcal{D}$ is the one of the $2$-designs as in Example \ref{sem1dimjedan} for $n=6$ and $q$ even, and $G_{2}(q) \unlhd G_{0}$;
     \item $\mathcal{D}$ is the symmetric $2$-$(2^{4},6,2)$ design as in \cite[Section 1.2.1]{ORR} and $A_{6} \unlhd G_{0} \leq S_{6}$;
     \item $\mathcal{D}$ is the $2$-$(3^{4},3^{2},2)$ design as in Example \ref{hall} and $G_{0} \cong \left(Z_{2^{i}}\circ SL_{2}(5)\right) :Z_{j}$ with $(i,j)=(2,1),(3,0),(3,1)$.
 \end{enumerate}
\end{theorem}

\bigskip

We analyze the cases where $L$ is alternating, sporadic, a group of Lie type in characteristic $p$, or a group of Lie type in characteristic different from $p$ in separate sections.

\bigskip

\subsection{The case  where $L$ in an alternating group}\label{alt}
In this section, we assume that $L \cong A_c$, an alternating group of degree $c\geq 5$. Our aim is to prove the following result.
\begin{theorem}\label{Alt}
 If $L \cong A_c$ with $c \geq 5$, then one of the following holds:
 \begin{enumerate}
     \item $c=5$, $\mathcal{D}$ is the $2$-$(3^{4},3^{2},2)$ design as in Example \ref{hall} and $G_{0} \cong \left(Z_{2^{i}}\circ SL_{2}(5)\right) :Z_{j}$ with $(i,j)=(2,1),(3,0),(3,1)$.
      \item $c=6$, $\mathcal{D}$ is the symmetric $2$-$(2^{4},6,2)$ design as in \cite[Section 1.2.1]{ORR} and $A_{6} \unlhd G_{0} \leq S_{6}$.
 \end{enumerate}   
\end{theorem}

\bigskip

We first consider the case where $V=V_n(q)$ is the fully deleted 
permutation module for $A_c$.
That is, $q=p$, and $A_c$ acts on $GF(p)^c$ by permuting 
coordinates naturally. Let
\[
X=\left\{(a_1,\ldots,a_c)\in GF(p)^c:\sum{a_i}=0\right\}
\text{~and~}
Y=\{(a,\ldots,a) :a\in GF(p)\}.
\]
Then $X/(X\cap Y)$ is the fully deleted permutation module.

\begin{lemma}\label{NFul}
If $L \cong A_c$ and  $V$ is the fully deleted permutation module, then $\mathcal{D}$ is the symmetric $2$-$(2^{4},6,2)$ design as in \cite[Section 1.2.1]{ORR} and $A_{6} \unlhd G_{0} \leq S_{6}$.
\end{lemma}

\begin{proof}
Suppose that $V$ is the fully deleted permutation module.
Then $n=c-1$ if $p\nmid c$ and $n=c-2$ if $p\mid c$.
Recall that $G_0$ has an orbit on $P_1(V)$ of size $c$ if $p\nmid c$, and of size $c(c-1)/2$ if $p\mid c$. 
So that $r$ divides $2c(q-1)$ if $p\nmid c$, 
and $r$ divides $c(c-1)(q-1)$ if $p\mid c$.

Suppose $q=2$. Then $r/2$ is odd, 
and so 
\begin{eqnarray*}
c&>2^{(c-1)/2}/\sqrt{2}, &\text{if c is odd;} \\
(c(c-1)/2)_{2^{'}}&>2^{(c-2)/2}/\sqrt{2}, &\text{if c is even.}
\end{eqnarray*}
It follows that $c=5, 6, 7, 8, 10, 12$ or $14$ since $c \geq 5$.
If $c=7$ then $r=14$ and $k=10$, whereas $b=vr/k\notin \mathbb{Z}$.
Similarly, we have that $c\neq 8, 10, 14$.
If $c=5$ then $n=4$ and $(v,k,r)=(16,4,10)$, if $c=6$ then $n=4$ and $(v,k,r)=(16,4,10)$ or $(16,6,2)$, and 
if $c=12$, then $(v,k,r)=(1024,32,66)$. Now, $(v,k,r)=(16,6,2)$ leads to the symmetric $2$-design as in \cite[Section 1.2.1]{ORR}, the remaining cases are ruled out by \cite[Corollary 1.3 and Lemma 4.2]{Mo}. 

Now, consider $q=3$. Then $(r/2,3)=1$, and so
\begin{eqnarray*}
2c&>3^{(c-1)/2}/\sqrt{2}, &\text{if~} 3\nmid c; \\
(c(c-1))_{3^{'}}&>3^{(c-2)/2}/\sqrt{2}, &\text{if~} 3\mid c.
\end{eqnarray*}
Hence $c=5$ or 6,
following from $r\mid 20$ and $r^2>2v$, 
we have $(v,k,r)=(81,9,20)$. Actually, $c\neq 6$ by Proposition \ref{C8} since $A_{6}\cong \Omega_{4}^{-}(3)$, and a similar argument to that of Proposition \ref{C8} rules out $c=5$ as well. 

When $q\geq 5$, we obtain
\begin{eqnarray*}
c(q-1)&>q^{(c-1)/2}/\sqrt{2}, &\text{if~} p\nmid c; \\
\frac{q-1}{2}\cdot(c(c-1))_{p^{'}}&>q^{(c-2)/2}/\sqrt{2}, &\text{if~} p\mid c.
\end{eqnarray*}
Hence $q=c=5$, and $v=125$.
Since $r$ divides $(c(c-1)(q-1),2(v-1))$,
we have that $r\mid 8$, contradicting the fact $r^2>2v$.

\end{proof}

In the sequel, we always suppose that 
$V$ is not the fully deleted permutation module for $A_c$.

\begin{lemma}\label{AltEvenAtMost16}
If $p=2$ then $c\leq 24$, $n\leq 115$;
and if $p\neq 2$ then $c\leq 16$.
\end{lemma}

\begin{proof}
Suppose $c\geq 15$ and $G^{(\infty)}_{0}=A_c$.
Since $V$ is not the fully deleted permutation module, combining with 
\cite[Theorem 7]{Ja1},
we have that $n\geq c(c-5)/4$.
If $p=2$, then $(c!)_{2^{'}}>2^{\frac{n-1}{2}} \geq 2^{\frac{c(c-5)-4}{8}}$,
implying $c\leq 24$ and $n\leq 115$.
If $p$ is odd, then 
$\sqrt{2}(q-1)(c!)_{p^{'}}>q^{\frac{c(c-5)}{8}}$, implying $c\leq 16$.

If $c\geq 9$ and 
$G^{(\infty)}_{0}=2.A_c$, 
then $p$ is odd, combining with \cite[Proposition 5.3.6]{KL}
we have that $n\geq 2^{[(c-s-1)/2]}$,
where $s$ is the number of terms in the 2-adic expansion of $c$.
Hence, $\sqrt{2}(q-1)(c!)_{p^{'}}>q^{n/2}$ implies $c<16$.

\end{proof}

\begin{lemma}\label{AltpEven}
$p \neq 2$.
\end{lemma}

\begin{proof}
Suppose that $p=2$. Then $c\leq 24$, $n\leq 115$.

\textbf{Case~1:}  $n\geq 20$.

If $10\geq c\geq 5$ and $q\geq 4$, 
then $\sqrt{2}(c!)_{2^{'}}<4^{10}\leq 4^{n/2}$, a contradiction.
If $14\geq c\geq 11$ then $n\geq 32$
by \cite[Appendix]{Ja2};
and if $c\geq 15$ then $n\geq c(c-5)/4$
by \cite[Theorem 7]{Ja1};
so, for $c\geq 11$,
we always have that $3(c!)_{2^{'}}<4^{n/2}$, forcing $q=2$.
Therefore, $q=2$.

If $25\mid r$, then by $r(k-1)=2(2^n-1)$,
we have $25\mid 2^n-1$. From \cite[Lemma 2.8]{Lieb} we obtain $20\mid n$,
and so $n=20$,  $40$, $60$, $80$ or $100$ (with $c\leq 14$, $15$, $18$, $20$ or $22$, respectively).
However, as $(2^{20}-1,10!)<2^{19/2}$,
we have that $n=20$ with $14\geq c\geq 11$,
a contradiction.
Moreover, as $(2^{40}-1,15!)<2^{39/2}$,
$(2^{60}-1,18!)<2^{59/2}$,
$(2^{80}-1,20!)<2^{79/2}$
and $(2^{100}-1,22!)<2^{99/2}$,
contradicting $r/2>2^{\frac{n-1}{2}}$.
Thus, $r$ is not divisible by $25$,
and similarly  $r$ is not divisible by 27 or 49.

Combining with $c\leq 24$ and $\frac{r}{2}\mid (c!)_{2'}$, we have that 
\[
\frac{r}{2}\mid
3^2\cdot 5\cdot7\cdot11^2\cdot13\cdot17\cdot19\cdot23.
\]
If $c\leq 10$, then $\frac{r}{2}$ divides $3^2\cdot 5\cdot7$, so $\frac{r}{2}<2^{19/2}$,
a contradiction.
If $c=11$ or 12, then $\frac{r}{2}$ divides $3^2\cdot 5\cdot7\cdot 11$ and $\frac{r}{2}>2^{19/2}$, implying that $11\mid r$,
and so $10\mid n$ by \cite[Lemma 2.8]{Lieb}.
As $\frac{r}{2}<2^{12}$, we have $n=20$,
whereas $(2^{20}-1,3^2\cdot 5\cdot7\cdot 11)<2^{19/2}$.
If $13\leq c\leq 16$, then $\frac{r}{2}$ divides $3^2\cdot 5\cdot7\cdot 11\cdot 13$, 
and so $11$ or 13 divides $r$.
As $\frac{r}{2}<2^{16}$,  by \cite[Lemma 2.8]{Lieb}, we have $n=20$, 24 or 30.
However, none of these cases satisfy $(2^{n}-1,3^2\cdot 5\cdot7\cdot 11\cdot 13)>2^{\frac{n-1}{2}}$.
If $c=17$ or 18, then $n\geq c(c-5)/4\geq 51$
and $\frac{r}{2}\leq 3^2\cdot 5\cdot7\cdot 11\cdot 13 \cdot 17<2^{\frac{n-1}{2}}$.
Similarly, if $19\leq c \leq 24$, then $\frac{r}{2}<2^{\frac{n-1}{2}}$.

\textbf{Case~2:}  $n<20$.

From \cite[Theorem 7]{Ja1}, we have $c\leq 14$.
Furthermore, \cite[Appendix]{Ja2} shows that if $c\geq 13$ then $n\geq 32$, so $c\leq 12$.
In the following parts, we use the 2-module character tables for $A_c(c\leq 12)$ and its covering groups by \cite{ModAt}.

If $10\leq c\leq 12$ then $n=16$, $q=2$ or 4,
whereas $\sqrt{2}(q^{16}-1, (q-1)12!)<q^{8}$,
a contradiction.

If $c=9$ then $n=8$, $q=2$.
Since $\frac{r}{2}$ divides $(2^8-1,9!)$ and $\frac{r}{2}>2^{7/2}$,
we obtain $r=30$, and so $k=18$, but $b=vr/k\notin \mathbb{Z}$.

If $c=8$ then $q=2$, $n=4$ or 14.
Since $r/2\leq (2^{14}-1,8!)<2^{7}/\sqrt{2}$, we have that $n\neq 14$.
Hence, $n=4$, and so $(v,k,r)=(4^2,4,10)$, impossible by \cite[Corollary 1.3]{Mo}.

Suppose that $c=7$. Then $n=4$, 6, 14 or 15.
If $n=4$, then $q=2$
and $(v,k,r)=(4^2,4,10)$, impossible by \cite[Corollary 1.3]{Mo}. 
If $n=6$ then $q=4$ and $G^{(\infty)}_{0}=3\cdot A_7$, and so $r/2$ divides $(4^6-1, 3\cdot 7!)=3^2\cdot 5\cdot 7$.
As $r^2>2v$, we have that $r=126$, 210 or 630.
However, when $r=126$, $b=vr/k\notin \mathbb{Z}$.
Therefore, we obtain $(k,r)=(40,210)$ or $(14,630)$, but only the latter occurs, as the former is ruled by Corollary \ref{p2}. In this case, $G_{0}<\Gamma U_{6}(2)$, and hence $G_{0}$ preserves a set of $2079$ non-zero singular vectors of $V$, but this contradicts Lemma \ref{sudbina}(2).  
If $n=14$ or 15, then $q=2$ or 4, respectively.
However, $r/2\leq (q^n-1,(q-1)\cdot 7!)<q^{n/2}/\sqrt{2}$.

Suppose that $c=6$. Then $n=2$, $3$, $4$, $8$ or $9$. If $n=2$, then $q=4$ by \cite{ModAt}. However, this contradicts the assumption that $Z_{2}.A_{6} \cong SL_{2}(9)$ is not contained in $G_{0}$.
If $n=3$ or 4, then  $(q,v,k,r)=(4,8^2,8,18)$, $(2,4^2,4,10)$ or $(2,4^2,6,6)$. The first two cases are ruled out by \cite[Corollary 1.3]{Mo}, the latter since we are assuming that $V$ is not the fully deleted permutation module for $A_{c}$. If $n=8$ or 9 then $q=4$, whereas $r/2\leq (4^n-1,3\cdot 6!)<q^{n/2}/\sqrt{2}$.

Finally, suppose that $c=5$. Then then $(n,q)=(2,4)$ or $(4,2)$. 
Actually, $(n,q)\neq(4,2)$ by \cite{ModAt}, and $(n,q)\neq (2,4)$ since $SL_{2}(q)$ is not contained in $G_{0}$. 
\end{proof}

\bigskip
\begin{remark}
 As pointed out in Lemma \ref{fato}, the group $G=T:G_{0}$ with $G_{0} \cong A_{5}$ acts flag-transitively on the symmetric $2$-$(16,6,2)$ design $\mathcal{D}$ having as $V_{4}(2)$ as a point set. However, $A_{5}$ acts irreducibly but not absolutely irreducibly on $V_{4}(2)$ by \cite[Lemma 2.10.1]{KL} since $C_{GL_{4}(2)}(A_{5}) \cong Z_{3}$, and this explains why such $G$ does not appear in the statements of Lemma \ref{AltpEven} or Theorem \ref{Alt}.     
\end{remark}

\bigskip

Throughout the remainder of this section we assume that $p$ is odd.

\bigskip

\begin{lemma}\label{AtMostA7}
We have $c\le 7$.
\end{lemma}

\begin{proof}
First consider the case $12\leq c\leq 16$.
Assume that $G^{(\infty)}_{0}=A_c$.
If $c=12$ then $n\geq 43$ by \cite{ModAt}, and if $c\geq 13$ then $n\geq 43$, arguing as in the proof of \cite[Proposition 2.5]{LP}.
Since $(q-1)\cdot (c!)_{p'}\geq r/2 >q^{n/2}/\sqrt{2}$, we have that 
$q=3$, $c=16$ with $45\geq n\geq 43$.
However, none of these cases satisfy $(3^n-1,2\cdot 16!)>3^{n/2}/\sqrt{2}$.
Now assume then $G^{(\infty)}_{0}=2\cdot A_c$. similarly, as $(q-1)\cdot (c!)_{p'}\geq r/2 >q^{n/2}/\sqrt{2}$, we have that $n\leq 45$. 
Moreover, from \cite[Proposition 5.3.6]{KL}, either $16\mid n$ or $32\mid n$ with $c\geq 14$,
and so $n=16$ or $32$.
Again, $(q^n-1,(q-1)\cdot c!)<q^{n/2}/\sqrt{2}$, a contradiction.

Suppose $c=10$ or 11.
As $q^{13}>\sqrt{2}(q-1)(11!)_{p'}$ for $q$ odd, we obtain $n\leq 25$.
and so $n=8$ or 16, by \cite{ModAt}. Further, 
$(q^n-1,(q-1)\cdot 11!)>q^{n/2}/\sqrt{2}$ implies $(n,q)=(8,3)$ or (8,7).
However, $b\notin \mathbb{Z}$ in both of the two cases.
Similarly, if $c=8$ or 9,
then from $q^{9}>\sqrt{2}(q-1)(9!)_{p'}$ for $q$ odd, we obtain $n\leq 17$, and so $n=8$ or 16, by \cite{ModAt}. $(q^n-1,(q-1)\cdot 9!)>q^{n/2}/\sqrt{2}$ implies $(n,q)=(8,3)$, $b\notin \mathbb{Z}$ again.

\end{proof}

\begin{lemma}\label{A7zer0}
If $c=7$, then $V=V_{3}(25)$, $3.A_7 \unlhd G_{0} \leq (Z_{8} \times 3.A_7):Z_{2}< \Gamma U_{3}(5)$ and $(v,k,r)=(5^{6},5^{3},252)$.
\end{lemma}
\begin{proof}
First suppose $q=3$. 
Since $\sqrt{2}(3^n-1,2\cdot 7!)>3^{n/2}$, we have $n=6$, 8 or 12, combined with \cite{ModAt}.
However, $b\notin \mathbb{Z}$ in the latter cases.
Thus $n=6$ and $A_{7} \unlhd G_{0} \leq Z_{2} \times S_{7}$ by \cite{ModAt}. Further, $r/2$ divides $(3^6-1,2\cdot 7!)=56$ and 
$r^2>2\cdot 3^{6}$, implying $(k,r)=(14,112) \text{~or~} (27,56)$. The latter is ruled out since $G_{0}$ does not have transitive permutation representation of degree $56$ by \cite{At}. Hence, $(k,r)=(14,112)$. Then $G_B$ is isomorphic to a subgroup of $G_0$, say $J$ with $\left\vert G_0:J\right\vert =\frac{r}{k}=8$ by Lemma \ref{cici} and Corollary \ref{p2}. However, $G_{0}$ does not have transitive permutation representation of  degrees $8$ by \cite{At}.

Now suppose $q\geq 5$.
As $q^{6}>\sqrt{2}(q-1)(7!)_{p'}$ for $q$ odd, we obtain $n\leq 11$,
and so $n=3$, 4, 6, 8, 9 or 10, by \cite{At, ModAt}.

If $n=8$, 9 or 10, then the fact that $\sqrt{2}(q^n-1,(q-1)\cdot 7!)>q^{n/2}$ implies $n=8$ and $q=3$, contradicting $q\geq 5$.

If $n=6$ then $r/2$ divides $(q^6-1,(q-1)\cdot 7!)$, and hence divides $21(q^2-1)$. Since $r>\sqrt{2}\cdot q^{3}$, we have that $q\leq 29$. Combining this with $r^2>2v$, $r(k-1)=2(v-1)$ and $bk=vr$, we have $q=5$. Then $G_{0}^{(\infty)} \cong 3.A_7$ since $V$ is not the fully deleted permutation module for $A_{7}$, and hence this case is excluded since the representation of $3.A_7$ on $V_{6}(5)$ is not absolutely irreducible by \cite{ModAt}.

If $n=4$, then $G^{(\infty)}_{0}=2.A_7$ with $q=p$ or $p^2$ by \cite{At, ModAt} since $p$ is odd. If $q=p^{2}$, then $r/2 \mid (p^8-1,7!(p-1))$ and hence $r\mid 40(p^2-1)$. Then $p=3$ or $5$ since $r>\sqrt{2}p^4$ and $p$ is odd. However, checking all the possibilities of $r$, we obtain either $k\notin \mathbb{Z}$ or $b\notin \mathbb{Z}$, a contradiction. Thus, $q=p=7$, $2.A_7 \unlhd G_{0} \leq (Z_{6}\circ A_{7}):Z_{2}$ and $2.A_7 < Sp_{4}(7)$ by \cite{ModAt}. It follows from $r/2$ divides $(7^4-1,6\cdot 7!)$ and 
$r>\sqrt{2}\cdot 7^{2}$ that $(k,r)=(6,960),(16,320)$ or $(21,240)$. The latter is ruled out by Corollary \ref{p2}. In the remaining cases, $V^{\ast}$ is partitioned into two orbits of length $720$ and $1680$ by \cite[Lemma 4.4. and Table 14]{Lieb0}, and each of these is, in turn, a union of $G_{0}$-orbits. So $r/2 \mid 240$ by Lemma \ref{sudbina}(2), a contradiction.

If $n=3$ then $p=5$, $q=25$ and $3.A_7 \unlhd G_{0} \leq (Z_{8} \times 3.A_7):Z_{2}< \Gamma U_{3}(5)$ by \cite{At, ModAt}. Since $r/2$ divides $(5^6-1,24\cdot 7!)=504$ and 
$r^2>2\cdot 5^{6}$ and 
$b\in \mathbb{Z}$, we have that $(v,k,r)=(5^{6},63,504)$ or $(5^{6},5^{3},252)$.

Assume that $(v,k,r)=(5^{6},63,504)$ and let $B$ be any block of $\mathcal{D}$ containing $0$. Then $G_B$ is isomorphic to a subgroup of $G_0$ of index $8$ by Lemma \ref{cici}. Thus $3.A_7 \unlhd G_{B} \leq 3.S_{7}$, and hence and $S_{5} \unlhd G_{0,B} \leq S_{5} \times Z_{2}$ by \cite{At}. Thus, $0^{Z}$ has length $3$, where $Z$ is the center of $G_{B}$, since $G_{B}$ acts transitively on $B$. Then $S_{5} \unlhd G_{0,B} \leq G_{0,\left\langle x\right \rangle } \leq \Gamma U_{3}(5)_{\left\langle x \right\rangle}$ for some $x \in 0^{Z}$, $x \neq 0$. However ,this is impossible by \cite{At}. Thus, the assertion follows. 
\end{proof}

\begin{lemma}\label{A7}
$c \neq 7$.
\end{lemma}

\begin{proof}
Assume that $V=V_{3}(25)$, $K \unlhd G_{0} \leq (Z_{8} \times K):Z_{2}< \Gamma U_{3}(5)$, where $K \cong 3.A_7$, and $(v,k,r)=(5^{6},5^{3},252)$. Recall that $\Gamma U_{3}(5)$ acts on $PG_{2}(25)$ inducing $P\Gamma U_{3}(5)$, and this one acts $2$-transitively on a Hermitian unital $\mathcal{H}$ of order $5$. Clearly, $\mathcal{H}$ is preserved by the group induced by $G_{0}$ on $PG_{2}(25)$.   

Let $B$ be any block of $\mathcal{D}$ containing $0$. Then $B$ is a $3$-dimensional $GF(5)$-subspace of $V$ by Corollary \ref{p2}. Then 
the projection $\pi_{B}$ of $B$ on $PG_{2}(25)$ is either a subset contained in line $\ell$ or a Baer subplane. 

Assume that $\pi_{B} \subset \ell$. Assume that $\ell$ is a tangent to $\mathcal{H}$ and let $\left\langle x\right\rangle$ its tangency point. Then $\left\vert B \cap \left\langle x\right\rangle \right\vert =5^{i}$ since $\mathcal{D}$ is a $2$-design and $B$ is a $3$-dimensional $GF(5)$-subspace of $V$. So $k=5^{i}+4e$, where $e$ is the number of points of $\ell$ distinct from $\left\langle x\right\rangle$ and intersecting $B$ in exactly a $1$-dimensional $GF(5)$-subspace of $V$. Then $i=2$ and $e=25$ since $e \leq 25$, and hence $B$ contains a basis of the subspace of $V$ projecting $\ell$ regarded over $GF(5)$, a contradiction. Thus, $\ell$ is a secant line to $\mathcal{H}$. Actually, $\ell$ is the unique line of $PG_{2}(25)$ containing $B$. Now, there is a Sylow $5$-subgroup of $K$ preserving $B$, and hence $\ell$, since $(r,5)=1$. Therefore, each Sylow $5$-subgroup of $K$ is the center of a Sylow $5$-subgroup of $\Gamma U_{3}(5)$ by \cite[Satz II.10.12]{Hup}, but this is not the case by \cite{At}.

Assume that $\pi_{B}$ is a Baer subplane of $PG_{2}(25)$. The previous argument on the Sylow $5$-subgroups of $K$ shows that $\pi_{B} \cap \mathcal{H}$ cannot be a line or two intersecting lines of $\pi_{B}$, hence $\pi_{B} \cap \mathcal{H}$ is an irreducible conic of $\pi_{B}$ by \cite[Theorem 2.9]{BarEbe}. Further, $\pi_{B} \cap \mathcal{H}$ is preserved by $G_{0,B}$ and hence by $K_{B}$. Note that, $K_{B}$ intersects trivially the center of $3.A_7$ since $B$ is a $3$-dimensional $GF(5)$-subspace of $V$. Thus $K_{B}$ acts faithfully on $PG_{2}(25)$, and hence $K_{B} \leq S_{5} \times Z_{2}$ by \cite{Mi}. On the other hand, $\left\vert K_{B}\right \vert$ is divisible by $30$ since $r=252$ and $K\unlhd G_{0}$. Thus $A_{5} \unlhd K_{B}$, and hence $K_{B}$ permutes $2$-transitively the six $2$-dimensional $GF(5)$-subspaces of $V$ projecting on $PG_{2}(25)$ the points of $\pi_{B} \cap \mathcal{H}$.

Note that, $S_{7}$ acts transitively on $\mathcal{H}$ by Lemma \cite[Lemma 2.8(iv)]{KanLib}, and the stabilizer in $S_{7}$ of a point of $\mathcal{H}$ is $F_{20}$. Hence, $(\pi_{B} \cap \mathcal{H})^{G_{0}}$ is a set of imprimitivity for the group induced by $G_{0}$ on $\mathcal{H}$. Thus, there are exactly $12$ pairwise distinct $G_{0,B}$-invariant blocks of $\mathcal{D}$ containing $0$, including $B$, projecting $\pi_{B}$ since $r=252$.

Let $B_{1},...,B_{12}$ be the blocks of $\mathcal{D}$ containing $0$ and projecting $\pi_{B}$, let $\left\langle x_{j} \right\rangle$, $j=0, 1,2,3,5$, be the $1$-dimensional subspaces of $V$ corresponding to the points of $\pi_{B} \cap \mathcal{H}$ and let $S$ be any Sylow $5$-subgroup of $G_{0,B}$. Then $S \cong Z_{5}$ fixes pointwise one of them, say $\left\langle x_{0} \right\rangle$, and permutes transitively the remaining five ones. Then $\bigcup_{i=1}^{5}\left\langle x_{j} \right\rangle ^{\ast}= \bigcup_{h \in I} \mathcal{O}_{h}$, where $ \mathcal{O}_{h}=\left( \left\langle \omega ^{h} x_{1} \right\rangle _{GF(5)}^{\ast} \right) ^{S}$, $\omega$ is a primitive element of $GF(25)$ and $\{\omega^{h}:h \in I\}$ with $I=\{0,3,8,11,16,19\}$ is system of distinct representatives of the cosets of $GF(5)^{*}$ in $GF(25)^{*}$. Hence, each $\mathcal{O}_{h}$ is contained in at least two distinct $B_{i}$'s since $S<G_{0,B}$ and $G_{0,B}=G_{0,B_{i}}$ for each $i=1,...,12$.

Let $B_{i_{1}}$ and $B_{i_{2}}$ with $i_{1},i_{2} \in \{1,...,12\}$ be such that $\mathcal{O}_{0} \subset B_{i_{1}}\cap B_{i_{2}}$. Then $B_{i_{1}}=B_{i_{2}}$ since $B_{i_{1}}$ and $B_{i_{2}}$ are $3$-dimensional $GF(5)$-subspaces of $V$ and $\mathcal{O}_{0}$ contains a basis of each of them. However, this is impossible since $B_{1},...,B_{12}$ are pairwise distinct. This completes the proof.
\end{proof}

\begin{lemma}\label{A6}
$c \neq 6$.
\end{lemma}

\begin{proof}
Since $(q-1)\cdot (6!)_{p'}\geq r/2 >q^{n/2}/\sqrt{2}$, we have that 
$n\leq 9$. If $n=9$ then $q=3$, and so $\sqrt{2}(3^9-1,2\cdot 6!)<3^{9/2}$, a contradiction. Further, $n\neq 7$ by \cite{At, ModAt}.

If $n=5$ or $8$, then $q=5$, and if $n\in \{3,6\}$ then $q\in \{9,25\}$, again $\sqrt{2}(q^n-1,(q-1))\cdot 6!)<q^{n/2}$, and hence all these cases ruled out. 

If $n=4$ then $q=5$. Since $r/2$ divides $(5^4-1,4\cdot 6!)$ and 
$r>\sqrt{2}\cdot 5^{2}$, we obtain $r=48$ or 96. However, $b\notin \mathbb{Z}$ in both of the two cases.

If $n=2$ then $G^{(\infty)}_{0}=2.A_6=SL_2(q)$ with $q=9$, which is contrary to our assumption.

\end{proof}

\begin{lemma}\label{A5}
If $c=5$ then $V=V_{2}(9)$, $SL_{2}(5) \unlhd G_{0} \leq (Z_{8}) \circ SL_{2}(5)):Z_{2}$ and $(v,k,r)=(81,9,20)$.
\end{lemma}
\begin{proof}
Suppose $c=5$.  Then $n=2$, 3, 4, 5 or 6, by \cite{At,ModAt}.  If $n=3$, 5 or 6, then $\sqrt{2}(q^n-1,(q-1))\cdot 5!)<q^{n/2}$, a contradiction. 
If $n=4$ then $q=5$ or $7$ since the representation of $G_{0}^{(\infty)}$ must be absolutely irreducible. Combining this with $r^2>2v$, $r(k-1)=2(v-1)$ and $bk=vr$, we have $q=7$ with
$(k,r)=(21,240)$, which is ruled out by Corollary \ref{p2}. Thus $n=2$, and hence $q=5,9$ or $49$. Actually, $q=5$ is ruled out since $SL_{2}(q)$ is not contained in $G_{0}$. 

Assume that that $q=9$. Then $SL_{2}(5) \unlhd G_{0} \leq (Z_{8} \circ SL_{2}(5)):Z_{2}$ and $(k,r)=(3,80),(5,40),(6,32)$, or $(9,20)$. Now, $(k,r)=(3,80)$ is ruled out by Corollary \ref{p2}, $(k,r)=(6,32)$ by Lemma \ref{fato}.

If $(k,r)=(5,40)$, for any block $B$ we have that $T_B=1$, and therefore $G_B$ is isomorphic to a subgroup $J$ of $G_0$ with $|G_0:J|=\frac{r}{k}=8$ by Lemma \ref{cici} since $(k,v)=1$. Thus $G_{0} \leq (Z_{8} \circ SL_{2}(5)):Z_{2}$ and $G_{B} \cong SL_{2}(5)$. Then the central involution of $G_{B}$, say $\sigma$, fixes $B$ pointwise since $SL_{2}(5)$ has not transitive permutation representations of degree $5$. Now, $Fix(\sigma)$ is either contained in $1$-dimensional $GF(9)$-subspace of $V$ or in $2$-dimensional $GF(3)$-subspace of $V$ since the involutions in $\Gamma L_{2}(9)$ distinct from $-1$ are either affine homologies or Baer involutions. Moreover, $G_{B}$ preserves $Fix(\sigma)$ inducing $A_{5}$ on it since $B\subset Fix(\sigma)$, which is a contradiction. Thus, $(v,k,r)=(81,9,20)$, and we obtain the assertion in this case.    

Assume that $q=49$. Then $(k,r)=(6,960),(16,320)$, or $(21,240)$. The latter cannot occur by Lemma \ref{p2}. In the remaining cases, $SL_{2}(5) \unlhd G_{0} \leq (Z_{24} \circ SL_{2}(5)):Z_{2} $ by \cite{ModAt} since the normalizer of $SL_{2}(5)$ in $\Gamma L_{2}(49)$ is $(Z_{24} \circ SL_{2}(5)):Z_{2}$. Now, $A_{5}$ has one orbit of length $20$ and one of length $30$ on $PG_{1}(49)$ by \cite[Lemma 11(i)]{COT}. Thus $r/2 \mid 20\cdot 48$ and $r/2 \mid 30\cdot 48$ by Lemma \ref{sudbina}(2), and hence $r/2 \mid 480$. Thus $(k,r)=(16,320)$ is ruled out.

Finally, assume that $(k,r)=(6,960)$. Then $T_B=1$, and therefore $G_B$ is isomorphic to a subgroup of $G_0$ of index $8$ by Lemma \ref{cici} since $(k,v)=1$. Thus $Y \unlhd G_{B}$ with $Y \cong SL_{2}(5)$, where $B$ is any block of $\mathcal{D}$ containing $0$. Now, $Y$ acts on $B$ inducing $A_{5}$ in its $2$-transitive permutation representation of degree $6$. The central involution $\sigma$ of $Y$ fixes $B$ pointwise. Then $Y$ induces $A_{5}$ on $Fix(\sigma)$, which is a $2$-dimensional $GF(7)$-subspace of $V$. So, $A_{5}<GL_{2}(7)$, a contradiction. This completes the proof.
\end{proof}

\bigskip
In order to settle the case $c=5$ and $(v,k,r)=(81,9,20)$, we need the following lemma containing some information on the $SL_{2}(5)$-orbits on the set of $2$-dimensional subspaces of $V_{4}(3)$.  
\bigskip

\begin{lemma}\label{A5bid}
Let $V=V_{4}(3)$, $\mathcal{L}$ the set of $2$%
-dimensional subspaces of $V$, $\Gamma =GL_{4}(3)$, $X=SL_{2}(5)$ and $N=N_{\Gamma}(X)$. Then the following hold:
\begin{enumerate}
\item $\Gamma$ contains a unique conjugacy class of subgroups isomorphic
to $X$. Further, $N\cong (Z_{8}\circ X):Z_{2}$ and $X<Sp_{4}(3)$;
\item $\mathcal{L}$ is partitioned into two orbits of length $30$, two orbits of length $20$, one orbit of length $10$ and four orbits of length $5$;  
 \item The $X$-orbit of length $10$ is a Desarguesian $2$-spread of $V$;
\item The four $X$-orbits of length $5$ are fused into two orbits of each length $10$ by $X:Z_{2}$, and these two orbits are Hall $2$-spreads of $V$.
\end{enumerate}
\end{lemma}

\begin{proof}
Assertion (1) follows from \cite[Lemma 5.1]{Mo}. Now, let $%
Y\cong Sp_{4}(3)$ such that $X<Y$. Then $Y$ partitions the set of $2$-dimensional subspaces
of $V_{4}(3)$ into two orbits of length $40$ and $90$, say $\mathcal{I}$ and $%
\mathcal{N}$. Then $\mathcal{I}$ and $\mathcal{N}$ consist of the totally
isotropic $2$-dimensional subspaces and then non-isotropic ones,
respectively, with regards to the symplectic form preserved by $Y$. The set $%
\mathcal{I}$ is partitioned into two $X$-orbits of length $30$ and $10$, and
the one of length $10$ is a Desarguesian $2$-spread of $V$, say $\mathcal{A}$ (see the proof of \cite[Lemma 5.1]{Mo}).  Clearly, $\Gamma_{\mathcal{A}} \cong \Gamma L_{2}(9)$, which is maximal in $\Gamma$ by \cite[Tables 8.8]{BHRD}. Moreover, $N=N_{\Gamma_{\mathcal{A}}}(X)$.

If $X$ preserves a further Desarguesian $2$-spread of $V$, say $\mathcal{A}'$, then there is $\alpha$ mapping $\mathcal{A}$ onto $\mathcal{A}'$ by \cite[Theorem I.1.11]{Lu}. Then there is $\beta$ mapping $\mathcal{A}$ onto $\mathcal{A}'$ and normalzing $X$ since $\Gamma L_{2}(9)$ has a unique conjugacy class of subgroups isomorphic to $X$ again by \cite{At}. Thus $\beta \in N$, and hence $\beta \in N_{\Gamma_{\mathcal{A}}}(X)$ since $N=N_{\Gamma_{\mathcal{A}}}(X)$. Therefore $\mathcal{A}'=\mathcal{A}$, and hence $\mathcal{A}$ is the unique Desarguesian $2$-spread of $V$ preserved by $X$.

Now, let us focus on $\mathcal{N}$.  We know that $X$ has exactly one orbit $\mathcal{W}$ of length $30$ on  $\mathcal{N}$ \cite[Lemma 5.1]{Mo}. Also, the group $X$ preserves a Hall $2$-spread $\mathcal{S}$
of $V$ and $\Gamma_{\mathcal{S}%
}\cong (D_{8}\circ Q_{8}).S_{5}$ by \cite[Theorem I.8.3]{Lu}. Further, $\mathcal{S}$ is partitioned into two orbits of length $5$ by $N_{\Gamma_{\mathcal{S}}}(X)\cong SL_{2}(5).Z_{2}$ by \cite[Propositions
5.3 and Corollary 5.5]{Fou}. Since $%
\Gamma_{\mathcal{S}}<Sp_{4}(3)$ and $N_{\Gamma}(\Gamma_{\mathcal{S}%
})=\Gamma_{\mathcal{S}}:Z_{2}$ by \cite[Tables 8.8 and 8.12]{BHRD}, it follows that $N_{\Gamma_{\mathcal{S}}}(X)$ preserves a further Hall $2$-spread 
$\mathcal{S}^{\prime }$ of $V$. Thus $\mathcal{S}\cup \mathcal{S}^{\prime }$ consists of four $X$-orbits each of length $5$. It is not difficult to see that not any choice of two $X$-orbits of the four ones yields a $2$-spread of $V$, hence $\mathcal{S}$ and $\mathcal{S}^{\prime }$ are the unique $2$-spreads contained in $\mathcal{S}\cup \mathcal{S}^{\prime }$ arising from the union of two $X$-orbits. Then $N$ permutes such $X$-orbits since $N\cong (Z_{8}\circ X):Z_{2}$ and $N_{\Gamma_{\mathcal{S}}}(X) \cong SL_{2}(5).Z_{2}$ and preserves each of them. Thus $N$ preserves $\mathcal{S}\cup \mathcal{S}^{\prime }$. Moreover, $N$ acts transitively on the four $X$-orbits each of length $5$. Indeed, if it is not so then there the subgroup of $Z_{4}$ which is central in $Z_{8}\circ X$ preserving each each component of $\mathcal{S}$ and  $\mathcal{S}^{\prime}$, but this contradicts \cite[Theorem I.1.12 and subsequent remark]{Lu}.

Assume that $X$ preserves a further Hall $2$-spread of $V$, say $\mathcal{S}^{\prime \prime}$. Then there is $\gamma$ mapping $\mathcal{S}$ onto $\mathcal{S}^{\prime \prime}$ by \cite[Theorem I.1.11]{Lu}. Note that $\Gamma_{\mathcal{S}%
}\cong (D_{8}\circ Q_{8}).S_{5}$ and $S_{5}$ induces $K \cong SO_{4}^{-}(2)$ on the $4$-dimensional $GF(2)$-space $W=(D_{8}\circ Q_{8})/\left\langle -1 \right\rangle$ by \cite[Table 8.12]{BHRD}, and the $\dim H_{1}(W,K)=0$ by \cite{ModAt}, it follows that $(D_{8}\circ Q_{8}).S_{5}$ contains a unique conjugacy class of subgroups isomorphic to $SL_{2}(5)$. Thus we may also assume that $\gamma$ normalizes $X$, and hence $\gamma \in N$. Consequently, either $\mathcal{S}^{\prime \prime}=\mathcal{S}$ or $\mathcal{S}^{\prime \prime}=\mathcal{S}^{\prime}$ since $N$ preserves $\mathcal{S}\cup \mathcal{S}^{\prime }$ and $\mathcal{S},\mathcal{S}^{\prime }$ are the unique $2$-spreads contained in $\mathcal{S}\cup \mathcal{S}^{\prime }$ arising from the union of two $X$-orbits. Thus, $\mathcal{S}$ and $\mathcal{S}^{\prime }$ are the unique $X$-invariant Hall $2$-spreads of $V$.  

Set $\mathcal{R}=\mathcal{N}\setminus(\mathcal{W} \cup \mathcal{S} \cup \mathcal{S}^{\prime})$. Clearly, $\mathcal{R}$ is $X$-invariant set of $2$-dimensional subspaces of $V$ of size $40$. Further, $X$ does not act transitively on $\mathcal{R}$ otherwise the stabilizer of an element of $\mathcal{R}$ in $X$ is of order $3$, whereas $-1 \in X$.

Assume that $X$ has one orbit $\mathcal{O}$ of length $5$ on $\mathcal{R}$. Now, $X$ partitions the nonzero vectors of $V$ into orbits of legnth $40$, so $\mathcal{O}$ is a partition of one of these two point orbit by by \cite[1.2.6]%
{Demb}. If there is $\psi \in N$ such that $\mathcal{O}^{\psi} \neq \mathcal{O}$, then $\mathcal{O}^{\psi} \cup \mathcal{O}$ is a $X$-invariant spread of $V$. Then $\mathcal{O}^{\psi} \cup \mathcal{O}$ is either Desarguesian or Hall $2$-spread of $V$ by \cite[Theorem 17.1.4]{Hir} and \cite[Section II.13]{Lu}, but the above argument shows that this is impossible. Thus $N$ preserves $\mathcal{O}$ and hence $\mathcal{O}$ is contained in a $N$-invariant $2$-spread of $V$ since the central $Z_{8}$ inside $Z_{8}\circ X<N$ preserves each each element of $\mathcal{O}$ by \cite[Lemma 2.1]{Dru} since $Z_{8}$ is a subgroup of Singer cycle of $GL_{4}(3)$. At this point we reach a contradiction as above. 

Assume that $X$ has one orbit of length $10$ on $\mathcal{R}$, say $\mathcal{C}$. The previous argument rules out the possibility for $\mathcal{C}$ to be a $2$-spread of $V$. Let $Z \in \mathcal{C}$, then $X_{Z}=\left\langle -1 \right \rangle .S_{3}$. If the subgroup of order $3$ of $X_{Z}$ acts nontrivially on $Z$, then $X_{Z}$ acts nontrivially on $Z$ preserving a $1$-dimensional subspace of this one. Hence, there is an involution fixing the $1$-dimensional subspace of $Z$ pointwise, but this is impossible since the unique involution of $X_{Z}$ is $-1$. Then the subgroup of order $3$ of $X_{Z}$ fixes $Z$ pointwise, which is a contradiction since subgroups of order $3$ of $X$ fix exactly one element of $2$-dimensional subsapce and this is a component of the $X$-invariant Desarguesian $2$-spread of $V$ (lying in $\mathcal{I}$). Thus $\mathcal{R}$ is partitioned into two $X$-orbits of length $20$. This completes the proof (2)--(4). 
\end{proof}

\begin{lemma}\label{AltpOdd}
If $L \cong A_c$ with $c \geq 5$ and $p$ is odd, then $c=5$, $\mathcal{D}$ is the symmetric $2$-$(81,9,2)$ design as in Example \ref{hall} and $G_{0}$ is $Z_{8}\circ SL_{2}(5)$, $\left(Z_{8}\circ SL_{2}(5)\right) :Z_{2}$ or any of the two subgroups of $\left(Z_{8}\circ SL_{2}(5)\right) :Z_{2}$ isomorphic to $\left(Z_{4}\circ SL_{2}(5)\right) :Z_{2}$.
\end{lemma}

\begin{proof}
It follows from Lemmas \ref{AtMostA7}, \ref{A7}, \ref{A6} and \ref{A5}, that $V=V_{2}(9)$, $SL_{2}(5) \unlhd G_{0} \leq (Z_{8}) \circ X):Z_{2}$, where $X \cong SL_{2}(5)$, and $(v,k,r)=(81,9,20)$.

Let $B$ be any block of $\mathcal{D}$ containing $0$. The case where $B$ is a $1$-dimensional $GF(9)$-subspace of $V$ is ruled out since it contradicts $\lambda=2$. Then $V=V_{4}(3)$ and $B$ is a $2$-dimensional $GF(3)$-subspace of $V$ by Corollary \ref{p2}. Hence, $B$ belongs either to any of the two $X$-orbit of length $20$ or to any of the four $X$-orbits of length $5$ on $V$ by Lemma \ref{A5bid}.

Assume that the former occurs. Then $T:X$ acts flag-transitively on $\mathcal{D}$, and $X_{B} \cong Z_{6}$. Now $V^{\ast}$ is partitioned into two $X$-orbit of length $40$, say $x_{1}^{X}$ and $x_{2}^{X}$. Since $(x_{i}^{X},B^{X})$, $i=1,2$, is a $1$-design by \cite[1.2.6]{Demb} and $B^{\ast}$ is partitioned into two $X_{B}$-orbits of lengths $3$ and $6$, it follows that any nonzero vector of $V$ is contained either in $1$ or in $3$ elements of $B^{X}$, but this contradicts that $\mathcal{D}$ is a $2$-design. Thus $B$ belongs to any of the four $X$-orbits of length $5$, and hence the assertion follows from Lemma \ref{A5bid}(4) (see Example \ref{hall}).
\end{proof}

\bigskip
\begin{proof}[Proof of Theorem \ref{Alt}]
The assertion immediately follows from Lemmas \ref{AltpEven} and \ref{AltpOdd}.    
\end{proof}

\subsection{The case where $L$ is a sporadic group}\label{sporad}
In this section, we prove the following fact:

\begin{theorem}\label{Spor}
$L$ is not a sporadic group.    
\end{theorem}

\begin{proof}    
\begin{table}[tbp]
\caption{Admissible $G_{0}^{(\infty)}$, $V$ and $(r,k)$ when $L$ is sporadic }
\label{tavspor}
\begin{tabular}
{|l|lll|}
\hline
Line & $G_{0}^{(\infty)}$ & $V$ & $(r,k)$ \\
\hline
1 & $M_{11}$ & $V_{5}(3)$  & $(44,12)$ \\
2 &          & $V_{10}(2)$  & $(66,32)$ \\
3 & $M_{12}$ & $V_{10}(2)$  & $(66,32)$ \\
4 & $Z_{3}.M_{22},$ & $V_{6}(4)$  & $(210,40),(630,14)$ \\
5 & $M_{22}$ & $V_{10}(2)$  & $(66,32)$ \\
6 & $Z_{3}.M_{22},$ & $V_{12}(2)$  & $(210,40),(630,14)$ \\
7 & $J_2$ & $V_{6}(4)$  & $(210,40),(630,14)$ \\ 
8 &   & $V_{12}(2)$  & $(210,40),(630,14)$ \\
9 & $Z_{2}.J_2$ & $V_{6}(5)$  & $(252,125),(504,63)$ \\
10 & $Z_{3}.J_3$ & $V_{9}(4)$  & $(1026,512)$ \\
11 &             & $V_{18}(2)$  & $(1026,512)$ \\
12 & $Z_{2}.Suz$ & $V_{12}(3)$ & $(1820,585)$ \\
\hline
\end{tabular}%
\end{table}
Assume that $L$ is a sporadic group. Then $G_{0}^{(\infty)}$, $V$ and $(r,k)$ are as in Table \ref{tavspor} by (\ref{pasticciotto}) and \cite{At,ModAt} (see Section \ref{TLA}). Now, the groups in lines 1,4,6,7 and 8 are excluded by Corollary \ref{p2}, those in lines 2,3,5 and 10 by \cite[Corollary 1.3]{Mo}. Thus, $G_{0}^{(\infty)} \cong Z_{2}.J_2$, $V=V_{6}(5)$ and $(r,k)=(252,125)$ or $(504,63)$, and hence $G_{0} \cong Z_{2^{i}}\circ G_{0}^{(\infty)}$ with $i=0,1$ by \cite{ModAt}. Then $(r,k)\neq (252,125)$ since $G_{0}$ does not have transitive permutation representations of degree $252$ by \cite{At}. Thus, $(r,k)=(504,63)$, and hence $G_{0}$ contains a subgroup of index $r/k=8$ by Lemma \ref{cici}. However, this is impossible by \cite{At}. 
\end{proof}


\subsection{The case where $L$ is a group of Lie type in characteristic $p$}\label{Lie}

In this section, we prove the following result:
\begin{theorem}\label{p-char}
If $L$ is a simple group of Lie type in characteristic $p$, then one of the following holds:
\begin{enumerate}
    \item $\mathcal{D}$ is the symmetric $2$-$(2^{4},6,2)$ design as in \cite[Section 1.2.1]{ORR} and $G_{0} \cong S_{5}$; 
    \item $\mathcal{D}$ is the one of the $2$-designs as in Example \ref{sem1dimjedan} for $n=6$ and $q$ even, and $G_{2}(q) \unlhd G_{0}$.
\end{enumerate}   
\end{theorem}

\begin{proof}
\begin{table}[tbp]
\caption{Admissible groups of Lie type in characteristic $p$, $V$ and $r$}
\label{tavLiecharp}
\begin{tabular}{|l|l|l|l|}
\hline
Line & $L$ & $V$ & $r$ \\
\hline
1 & $PSL_{2}(q^{3})$ & $V_{8}(q)$ & $\frac{\left( (q-1)(q^{3}+1),q^{8}-1\right) }{%
\theta }$, $\theta \geq 1$ \\ 
2 & $PSL_{5}(q)$, $PSU_{5}(q^{1/2})$ & $V_{10}(q)$ & $\frac{4(q^{5}-1)}{\theta }$, $%
\theta =1,2$ \\ 
3 & $\Omega _{7}(q)$, $P\Omega _{8}^{+}(q)$, $P\Omega _{8}^{-}(q^{1/2})$ & $V_{8}(q)
$ & $\frac{4(q^{4}-1)}{\theta }$, $\theta =1,2$ \\ 
4 & $\Omega _{9}(q)$, $P\Omega _{10}^{+}(q)$, $P\Omega _{10}^{-}(q^{1/2})$ & 
$V_{16}(q)$ & $\frac{4(q^{5}-1)}{\theta }$, $\theta =1,2$ \\ 
5 & $PSL_{2}(q)$, $PSL_{2}(q^{1/2})$ & $V_{4}(q)$ & $\frac{2(q^{4}-1)}{\theta }$, $%
\theta \geq 1$ \\ 
6 & $PSL_{4}(q)$, $PSU_{4}(q^{1/2})$ & $V_{16-\delta}(q)$, $\delta =0,1,2$ & $%
\frac{\left( (q^{4}-1)(q^{3}-1),q^{16-\delta }-1\right) }{\theta }$ \\ 
7 & $G_{2}(q)$, $q$ even & $V_{6}(q)$ & $\frac{2(q^{6}-1)}{\theta }$, $\theta \geq 1
$ \\ 
8 & $^{3}D_{4}(q^{1/2})$ & $V_{8}(q)$ & $\frac{4(q^{4}-1)}{\theta }$ \\ 
9 & $Sz(q)$ & $V_{4}(q)$ & $\frac{2(q^{2}+1)(q-1)}{\theta }$, $1\leq \theta <2(q-1)$\\
\hline
\end{tabular}%
\end{table}

\bigskip
Assume that $L$ is a simple group of Lie type in characteristic $p$. Then $G_{0}^{(\infty)}$, $V$ and $(r,k)$ are as in Table \ref{tavLiecharp} by (\ref{pasticciotto}) and \cite{At,ModAt} (see Section \ref{TLA}). Now, the group in Line 1 cannot occur since $\left( (q-1)(q^{3}+1),q^{8}-1\right)
<q^{4}/\sqrt{2}$, and those corresponding to lines 2, 3, 4 and 8 are ruled out by
Lemma \ref{fato}. Case in line 6 is excluded since $\left(
(q^{4}-1)(q^{3}-1),q^{16-\delta }-1\right) <q^{\frac{16-\delta }{2}}/\sqrt{2}
$. 

Assume that the case in line 5 holds. \ If $L\cong PSL_{2}(q^{1/2})$, then $%
r$ divides $\left( q^{1/2}(q-1)(q-1)\log _{p}q^{1/2},2(q^{4}-1)\right) $ and
hence\  $r\mid 8(q-1)\log _{p}q^{1/2}$ and hence $8\sqrt{2}(q-1)\log
_{p}q^{1/2}>q^{4}$, a contradiction. Thus $L\cong PSL_{2}(q)$ and hence $%
r\mid \left( q(q^{2}-1)(q-1)f,2(q^{4}-1)\right) $ with $f=\log _{p}q$. Then $r$ divides $2^{(q-1,2)}(q^{2}-1)\left( f,\frac{q^{2}+1}{(2,q-1)}\right)$ according to whether $q$ is even or odd respectively. If $\left( f,\frac{q^{2}+1}{(2,q-1)}\right)=1$ then assertion (1) follows by Lemma \ref{fato} since $L\cong PSL_{2}(4) \cong A_{5}$. If $\left( f,\frac{q^{2}+1}{(2,q-1)}\right)>1$ then $q = p^{f}$ with $f$ odd and $f>1$. Further, $G_{0}$ preserves a tensor decomposition of $V_{4}(q)$ by \cite[Proposition 5.4.6]{KL}, hence no cases arise by Lemma \ref{tensred}.

Assume that $L\cong Sz(q)$. Then $L\trianglelefteq G_{0}\leq \left( Z\times
L\right) :Z_{t}$, where $Z$ is the center of $GL_{4}(q)$ and $t=loq_{2}q\geq
3$ is odd. Let $B$ be any block of $\mathcal{D}$ containing $0$. Then $%
\left\vert L:L_{B}\right\vert \mid \frac{2(q^{2}+1)(q-1)}{\theta }$ since $r=%
\frac{2(q^{2}+1)(q-1)}{\theta }$ and hence $\theta \frac{q^{2}}{2}\mid
\left\vert L_{B}\right\vert $. Then $L_{B}$ contains exactly a subgroup $Q$
of index $2$ of a Sylow $2$-subgroup $S$ of $G_{0}$ since $L\trianglelefteq
G_{0}\leq \left( Z\times L\right) :Z_{t}$. Thus  $\theta \mid q-1$ and $%
\theta <q-1$ by \cite[Table 8.16]{BHRD}. Note that, $G_{0}$ preserves a L%
\"{u}neburg $2$-spread $\mathcal{L}$ of $V$ by \cite[Theorem IV.27.3]{Lu}
and $S$ fixes a component of $\mathcal{L}$, say $X$, and splits the $q^{2}$
remaining ones into $2$ orbits of length $q^{2}/2$. Let $Y_{1},Y_{2}$ be the
representatives of the $Q$-orbits on $\mathcal{L}\setminus \left\{ X\right\} 
$. Then $k=\left\vert B\cap X\right\vert +\frac{q^{2}}{2}k_{0}$ with $%
k_{0}=\left\vert B^{\ast }\cap Y_{1}\right\vert +\left\vert B^{\ast }\cap
Y_{2}\right\vert $. On the other hand $k=(q+1)\theta +1$ since $k=\frac{%
2(q^{4}-1)}{r}+1$, and hence $\frac{q^{2}}{2}k_{0}\leq (q+1)\theta $ and so $%
\theta >\frac{q^{2}k_{0}}{2(q+1)}>\frac{\left( q-1\right) k_{0}}{2}$. Then $%
k_{0}=0$ since $\theta \mid q-1$ and $\theta <q-1$. Thus $B\subseteq X$. If $%
B=X$, then $\mathcal{D}$ is the L\"{u}neburg translation plane of order $%
q^{2}$ by \cite[Theorem IV.31.1]{Lu} which is not the case since $\lambda
=2$. Then  $\ (X,B^{G_{X}})$ is a $2$-$(q^{2},k,2)$ design with $r_{X}=\frac{%
2(q-1)}{\theta }$ since $\mathcal{L}$ is a $2$-spread of $V$, but this
contradicts Lemma \ref{fato} since $q$ is an odd power of $2$ and $q\geq 8$.

Finally, assume that $L\cong G_{2}(q)$, $q$ even. The $\mathcal{D}$ is the
one of the $2$-designs as in Example \ref{sem1dimjedan} for $n=6$, and the proof is essentially that of Lemma \ref{vrata}.    
\end{proof}

\subsection{The case where $L$ is a group of Lie type in $p^{\prime}$-characteristic}

\begin{theorem}\label{cross-char}
If $L$ is a simple group of Lie type in $p^{\prime}$-characteristic, then one of the following holds:
\begin{enumerate}
    \item $\mathcal{D}$ is the $2$-$(3^{4},3^{2},2)$ design as in Example \ref{hall} and $G_{0} \cong \left(Z_{2^{i}}\circ SL_{2}(5)\right) :Z_{j}$ with $(i,j)=(2,1),(3,0),(3,1)$; 
    \item $\mathcal{D}$ is the symmetric $2$-$(2^{4},6,2)$ design described in \cite[Section 1.2.1]{ORR} and $A_{6} \unlhd G_{0} \leq S_{6}$.
\end{enumerate}   
\end{theorem}

\begin{lemma}\label{SoloUnit}
If $L \ncong PSL_{m}(s)$, then $G_{0}^{(\infty)} \cong PSU_{3}(3)$, $V=V_{6}(5)$ and $(r,k)=(252,125)$. 
\end{lemma}
\begin{proof}

\begin{table}[tbp]
\caption{Admissible groups of Lie type in characteristic $s$, $s \neq p$, with $L \ncong PSL_{m}(s)$, and corresponding $V$ and $(r,k)$}
\label{tavLiecrosscharnoPSL}
\begin{tabular}{|l|l|l|l|}
\hline
Line & $G_{0}^{(\infty )}$ & $V$ & $(r,k)$   \\ 
\hline
 1 & $Sp_{4}(3)$ & $V_{4}(7)$ & $(240,21), (320,16), (960,6)$\\
 2 &             & $V_{4}(13)$ & $(480,120)$  \\
 3 & $PSU_{3}(3)$ & $V_{6}(5)$ & $(252,125)$ \\
 4 & $Z_{3}.PSU_{4}(3)$ & $V_{6}(4)$ & $(210,40),(630,14)$  \\
 5 & $Z_{2}.G_{2}(4)$ & $V_{12}(3)$ & $(1820,585)$\\
\hline
\end{tabular}%
\end{table}
Assume that $L$ is a simple group of Lie type in characteristic $s$ different from $p$ and not isomorphic to $PSL_{m}(s)$. Then $G_{0}^{(\infty)}$, $V$ and $(r,k)$ are as in Table \ref{tavLiecrosscharnoPSL} by (\ref{pasticciotto}) and \cite{At,ModAt} (see Section \ref{TLA}).
In Line 1, $Sp_{4}(3) \unlhd G_{0} \leq Z_{6}\circ Sp_{4}(3)$ by \cite[Tables 8.8 and 8.9]{BHRD}. Then $(r,k)\neq (240,21)$ by Corollary \ref{p2}, and $(r,k)\neq (320,16)$ by Lemma \ref{cici} since $G_{0}$ has not transitive permutation representations of degree $r/k=20$ by \cite{At}. Thus $(r,k) =(960,6)$. A $Z_{6}\circ Sp_{4}(3)$-orbit on $V^{\ast}$ is of length $240$ by \cite[Lemma 3.4]{Lieb0}, and this one is a union of $G_{0}$-orbits. So, $r/2 \mid 240$ by Lemma \ref{sudbina}(2), a contradiction.

In Line 2, $G_{0}$ contains a subgroup $J$ of index $r/k=4$ isomorphic to $G_{B}$, where $B$ is any block of $\mathcal{D}$ containing $0$, by Lemma \ref{cici}. However, this is impossible since  $Sp_{4}(3) \unlhd G_{0} \leq Z_{12}\circ Sp_{4}(3)$ by \cite[Tables 8.8 and 8.9]{BHRD}.

Line 5 and the first case of Line 4 are excluded by by Corollary \ref{p2}. If $(r,k)=(630,14)$ in Line 4, then $Z_{3}.PSU_{4}(3) \unlhd G_{0} \leq Z_{3}.PSU_{4}(3).Z_{2}$ by \cite{ModAt}. If $B$ is any block of $\mathcal{D}$ containing $0$, then $T_{B}=1$ by Corollary \ref{p2}, and hence $G_{0}$ contains a transitive permutation representation of degree $r/k=45$ by Lemma \ref{cici}, which is not the case by \cite{At}. Finally, Line 3 yields the assertion.
\end{proof}

\begin{lemma}\label{SoloLin}
$L \cong PSL_{m}(s)$ with $s \geq 5$. 
\end{lemma}
\begin{proof}
In order to prove the assertion we only need to rule out the case $G_{0}^{(\infty)} \cong PSU_{3}(3)$, $V=V_{6}(5)$ and $(r,k)=(252,125)$ by Lemma \ref{SoloUnit}. In this case, $PSU_{3}(3)\trianglelefteq G_{0}\leq (Z\times
PSU_{3}(3)):Z_{2}<Z\circ K<GSp_{6}(5)$ by \cite{At} and \cite[Table 8.29]%
{BHRD}, where $Z$ is the center of $GL_{6}(5)$ and $K\cong \left\langle
-1\right\rangle .J_{2}.Z_{2}$. Since the blocks of $\mathcal{D}$ containing $%
0$ are subspaces of $V$, then each of them is preserved by $Z$, and hence $%
H=T:H_{0}$ with $PSU_{3}(3)\trianglelefteq H_{0}\leq PSU_{3}(3):Z_{2}$ acts
flag-transitively on $\mathcal{D}$. Thus, we may assume that $%
PSU_{3}(3)\trianglelefteq G_{0}\leq PSU_{3}(3):Z_{2}$ without loss. 

Let $B$ be any block of $\mathcal{D}$ containing $0$. Then $B$ is a $3$%
-dimensional subspace of $V$ by Lemma \ref{cici}, and $U\leq G_{0,B}$ with $%
U\cong Z_{3}$ a Sylow $3$-subgroup of $G_{0}$ since $r=252$. Also, $\dim
Fix(U)\cap B\geq 1$ since $\dim B=3$.

Let $S$ be a Sylow $3$-subgroup of $G_{0}$ containing $U$. Note that, $G_{0}$
has two conjugacy classes of subgroups of order $3$ by \cite{At}.
Representatives of these classes are $Z(S)$ and any subgroups of order $3$
of $S$ distinct from $Z(S)$. The group $S$ is a Sylow $3$-subgroup of $K$
and the two conjugacy $G_{0}$-classes are not fused in $K$ again by \cite{At}%
. From \cite[Lemma 5.1]{Lieb0} and its proof we deduce that, $K$ partitions
the points of $PG_{5}(5)$ into two orbits of length $1890$ and $2016$ and
the stabilizer in $K$. Moreover, the stabilizer of any point in the $K$%
-orbit of length $1890$ has order coprime to $3$, whereas that in the $K$%
-orbit of length  $2016$ is isomorphic to $\left\langle -1\right\rangle
.\left( (Z_{5})^{2}:\left( Z_{4}\times S_{3}\right) \right) $. The $3$%
-subgroups of $\left\langle -1\right\rangle .\left( (Z_{5})^{2}:\left(
Z_{4}\times S_{3}\right) \right) $ do not lie in the conjugacy $K$-class
containing $Z(S)$ by \cite{At}. Thus $U$ is contained in a suitable copy of $%
\left\langle -1\right\rangle .\left( (Z_{5})^{2}:\left( Z_{4}\times
S_{3}\right) \right) $ since $\dim Fix(U)\cap B\geq 1$, and hence $\dim
Fix(U)=2$ since $\left\vert N_{K}(U)\right\vert =288$ again by \cite{At}.
This forces $Fix(U)\cap B=\left\langle x\right\rangle $ for some non zero
vector $x$ of $V$ since $\dim B=3$ and since $U$ acts semiregularly on $%
V\setminus Fix(U)$. Also, the six $1$-dimensional subspace of $Fix(U)$ are
contained in the $K$-orbit of length $2016$. Then the number of blocks
preserved by $U$ is $12$ since $\lambda =2$, and the above argument shows
that $U$ is a Sylow $3$-subgroup of the stabilizer in $G_{0}$ of any of
these blocks. On the other hand, the number of blocks preserved in $B^{G_{0}}$
is given by the ration $\frac{\left\vert N_{G_{0}}(U)\right\vert }{%
\left\vert N_{G_{0,B}}(U)\right\vert }$. Thus $\left\vert
N_{G_{0}}(U)\right\vert =12\left\vert N_{G_{0,B}}(U)\right\vert $, which
forces $G_{0}\cong PSU_{3}(3):Z_{2}$, $N_{G_{0}}(U)\cong S_{3}\times S_{3}$
and $N_{G_{0,B}}(U)=U$ by \cite{At}.

Let $M$ be a maximal subgroup of $G_{0}$ containing $G_{0,B}$. Since $U$
does not belong to the conjugacy $K$-class containing $Z(S)$, then $U$ does
not belong to the conjugacy $G_{0}$-class containing $Z(S)$, hence either $%
M\cong PGL_{2}(7)$, or $M\cong (Z_{4})^{2}:D_{12}$ by \cite{At}. The former
case is ruled out since $G_{0,B}$ contains $S_{3}$ or $Z_{6}$ since $%
r=\left\vert G_{0}:M\right\vert \left\vert M:G_{0,B}\right\vert $ and $r=252$%
, whereas $N_{G_{0,B}}(U)=U$. The latter implies $G_{0,B}\cong (Z_{4})^{2}:U$%
. Further $N_{M}(U)\cong D_{12}$ and $N_{M}(U)=N_{M}(G_{0,B})$.

Let $\pi $ be the projection of $B$ on $PG_{5}(5)$, then $\pi \cong PG_{2}(5)
$ and $G_{0,B}$ acts on $\pi $. Let $H$ be the kernel of the action of $%
G_{0,B}$ on $\pi $. If $H\neq 1$, then $(Z_{2})^{2}\leq H$. Hence, $%
\left\langle -1\right\rangle \times A_{4}\leq \left\langle -1\right\rangle
\times G_{0,\left\langle x\right\rangle }\leq \left( K:Z_{2}\right)
_{\left\langle x\right\rangle }\leq Z_{2}.\left( (Z_{5})^{2}:D_{12}\right)
:Z_{2}$, a contradiction. Thus $H=1$, and hence $G_{0,B}$ acts faithfully $%
\pi $. Note that, $(Z_{4})^{2}$ consists of homologies in a triangular
configuration $\Delta $, whose vertices are permuted transitively by $U$ by 
\cite{Mi}. Thus, $G_{0,B}$ acts irreducibly on $B$. Further, $N_{M}(U)$
preserves $\left\langle x\right\rangle $ since $Fix(U)\cap B=\left\langle
x\right\rangle $. Therefore, $N_{M}(G_{0,B})=N_{M}(U)=N_{G_{0,\left\langle
x\right\rangle }}(U)$ since $N_{M}(U)\cong D_{12}$, $N_{G_{0}}(U)\cong
S_{3}\times S_{3}$ and $U$ is a Sylow $3$-subgroup of $G_{0,\left\langle
x\right\rangle }$.

Let $B^{\prime }$ be the other block of $\mathcal{D}$ containing $%
\left\langle x\right\rangle $ since $\lambda =2$. Clearly $U$ preserves $%
B^{\prime }$. Then there is $\varphi \in N_{G_{0}}(U)$ such that $B^{\prime
}=B^{\varphi }$ since $G_{0}$ acts flag-transitively on $\mathcal{D}$ and $U$
is a sylow $3$-subgroup of $G_{0,B}$ as well as of $G_{0,B^{\prime }}$.
Arguing as above with $B^{\prime }$ in the role of $B$ we see that $%
Fix(U)\cap B^{\prime }=\left\langle x\right\rangle $. Thus  $\left\langle
x\right\rangle ^{\varphi }=\left( Fix(U)\cap B\right) ^{\varphi }=\left(
Fix(U)\cap B\right) ^{\varphi }=\left\langle x\right\rangle $ and hence $%
\varphi \in N_{G_{0,\left\langle x\right\rangle }}(U)$. Therefore, $\varphi
\in N_{M}(G_{0,B})$ since $N_{M}(G_{0,B})=N_{M}(U)=N_{G_{0,\left\langle
x\right\rangle }}(U)$. Then $G_{0,B}=G_{0,B^{\prime }}$, and hence $G_{0,B}$
preserves the proper subspace $B\cap B^{\prime }$ of $B$ since $\left\langle
x\right\rangle \subset B\cap B^{\prime }$ and $B\neq B^{\prime }$. However,
this is impossible since $G_{0,B}$ acts irreducibly on $B$.     
\end{proof}

\begin{lemma}\label{linearcross}
One the following holds:
\begin{enumerate}
\item $m=2$ and one of the following holds:
\begin{enumerate}
     \item $s=5$, $\mathcal{D}$ is the $2$-$(3^{4},3^{2},2)$ design as in Example \ref{hall} and $G_{0} \cong \left(Z_{2^{i}}\circ SL_{2}(5)\right) :Z_{j}$ with $(i,j)=(2,1),(3,0),(3,1)$.
      \item $s=7$, $V=V_{3}(9)$, $G_{0}^{(\infty)} \cong PSL_{2}(7)<SU_{3}(3)$ and $(r,k)=(54,27),(112,14)$; 
      \item $s=7$, $V=V_{3}(11)$, $G_{0}^{(\infty)} \cong PSL_{2}(7)<SL_{3}(11)$ and $(r,k)=(140,20)$; 
      \item $s=7$, $V=V_{3}(25)$, $G_{0}^{(\infty)} \cong PSL_{2}(7)<SU_{3}(5)$ and $(r,k)=(252,125),(504,63)$; 
      \item $s=7$, $V=V_{6}(3)$, $G_{0}^{(\infty)} \cong PSL_{2}(7)<\Omega^{-}_{6}(3)$ and $(r,k)=(54,27)$;
      \item $s=9$, $\mathcal{D}$ is the symmetric $2$-$(2^{4},6,2)$ design described in \cite[Section 1.2.1]{ORR} and $A_{6} \unlhd G_{0} \leq S_{6}$;
     \item $s=13$, $V=V_{6}(3)$, $G_{0}^{(\infty)} \cong SL_{2}(13)$ and $(r,k)=(182,9)$;
     \item $s=13$, $V=V_{6}(4)$, $G_{0}^{(\infty)} \cong PSL_{2}(13)$ and $(r,k)=(546,16)$;
\end{enumerate}
\item $m=3$, $s=4$, $V_{6}(3)$, $G_{0}^{(\infty)} \cong Z_{2}.PSL_{3}(4)<\Omega^{-}_{6}(3)$ and $(r,k)=(54,27)$.          
\end{enumerate}
\end{lemma}

\begin{proof}.
We may assume that $L \cong PSL_{m}(s)$ with $m \geq 2$ and $s \neq p$. If $(m,s)=(2,5)$ or $(2,9)$, then $L$ is isomorphic to $A_{5}$ or $A_{6}$, respectively. Hence (1.a) and (1.f) follows from Theorem \ref{Alt}, respectively. If $(m,s)\neq (2,5),(2,9)$, then $G_{0}^{(\infty)}$, $V$ and $(r,k)$ are as in Table \ref{tavLiecrosscharisoPSL} by (\ref{pasticciotto}) and \cite{At,ModAt}.  

\begin{table}[tbp]
\caption{Admissible groups of Lie type in characteristic $s$, $s \neq p$, with $L \cong PSL_{m}(s)$, and the corresponding $V$ and $(r,k)$}
\label{tavLiecrosscharisoPSL}
\begin{tabular}{|l|l|l|l|}
\hline
Line & $G_{0}^{(\infty )}$ & $V$ & $(r,k)$   \\ 
\hline
1 & $PSL_{2}(7)$   & $V_{3}(q)$, $q=p,p^{2}$  & $\left(\frac{42(p-1)}{\theta},\frac{(p^2+p+1)\theta}{21}+1 \right)$ with $(p,14)=1$   \\
2 &                & $V_{6}(3)$  & $(56,27),(112,14)$ \\
3 & $SL_{2}(7)$    & $V_{4}(q)$  & $\left(\frac{4(q^{2}-1)}{\theta},\frac{(q^{2}+1)\theta}{2}+1 \right)$ with $(q,14)=1$   \\
4 & $PSL_{2}(11)$  & $V_{5}(3)$  & $(44,12)$   \\ 
5 &                & $V_{5}(4)$  & $(66,32)$  \\ 
6 &                & $V_{10}(2)$ & $(66,32)$  \\
7 & $SL_{2}(13)$   & $V_{6}(3)$  & $(56,27),(112,14),(182,9),(208,8)$ \\ 
8 & $PSL_{2}(13)$  & $V_{6}(4)$  & $(91,91),(234,36),(546,16),(1638,6)$   \\  
9 & $PSL_{2}(17)$ & $V_{8}(2)$  & $(34,16),(102,6)$   \\
10 & $PSL_{2}(19)$ & $V_{9}(4)$  & $(1026,512)$   \\
11 & $PSL_{2}(25)$ & $V_{12}(2)$ & $(130,64)$   \\ 
12 & $PSL_{3}(3)$  & $V_{12}(2)$ & $(234,36)$   \\ 
13 & $Z_{2}.PSL_{3}(4)$  & $V_{6}(3)$  & $(56,27),(112,14)$  \\
\hline
\end{tabular}%
\end{table}
 
Assume that $G_{0}^{(\infty )}$, $V$  and $(r,k)$ are as in Line 1. Suppose that that $T_{B}=1$. Then $k\mid r$ by Lemma \ref{cici}. Hence%
\begin{equation}
\frac{42(q-1)}{\theta }=a\left[ \frac{\left( q^{2}+q+1\right) \theta }{21}+1%
\right]   \label{umoran}
\end{equation}%
for some $a\geq 1$. Then $882>a(q+1)\theta ^{2}$ and hence $\theta \leq 14$
since $q$ is odd. Now, it is not difficult to see that the admissible
solutions for (\ref{umoran}) are $(q,\theta )=(9,3),(11,3)$ or $(25,2)$.
Hence, $(r,k)=(112,14)$, $(140,20)$ or $(504,63)$, respectively. Further, $G_{0}^{(\infty )}$ is a subgroup of $SU_{3}(q)$ in the first and in the third case, a subgroup of $SL_{3}(11)$ in the second case by \cite[Tables 8.3--8.7]{BHRD}, respectively. Thus, we  obtain the second case of (1.b), (1.d), and case (1.c), respectively. 

Assume that $T_{B}\neq 1$. Then $k=p^{t}c$ with either $c=1$ and $t>1$, or $c=2$ by Corollary \ref{p2}%
. Then $\frac{r}{2}\left( p^{t}c-1\right) =q^{3}-1$ with $p^{t}c-1<\sqrt{2}%
q^{3/2}$ since $r$ is even and $r>\sqrt{2}q^{3/2}$ by our assumptions. If $%
q=p$, then $t=1$ and hence $c=2$ and $2p-1\mid p^{3}-1$,
a contradiction. Thus $q=p^{2}$ and hence $p^{t}c-1<\sqrt{2}p^{3}$. If $c=2$%
, then $t=1,2$. Actually, $t=1$ since $2p^{t}-1\mid p^{6}-1$. Thus, $%
2p-1\mid 63$ and hence either $q=25$ and $r=10$, or $q=121$ and $k=22$.
However, $k\neq \frac{\left( q^{2}+q+1\right) \theta }{21}+1$ in both cases,
and hence they are ruled out. 

Finally, if $c=1$ and $t>1$, then $t\mid 6$ and hence $t=2$ or $3$ since $p^{t}-1<\sqrt{2}%
p^{3}$.
Then $p^{t}=\frac{\left( p^{4}+p^{2}+1\right) \theta }{21}+1$. If $t=2$,
then $p=3$ since $p$ is odd but no integer $\theta $ arise. Hence $t=3$ and
hence $p<21$, and actually, either $p=3$ and $\theta =6$, or $p=5$ and $%
\theta =4$. Then $(q,r,k)=(3^{2},3^{3},56)$ or $(5^{2},5^{3},252)$. Thus $G_{0}^{(\infty )}$ is subgroup of $SU_{3}(3)$ or $SL_{3}(5)$, respectively, by \cite[Tables 8.3--8.7]{BHRD}. Thus, we  obtain the first case of (1.b) and (1.d), respectively. 

In line 2, $G_{0}^{(\infty )} \cong PSL_{2}(7)<\Omega_{6}^{-}(3)$ by \cite{ModAt} and \cite[Tables 8.31--8.34]{BHRD}, and hence (1.e) follows.

Lines 3, 4 and 5,6 are excluded by Lemma \ref{fato}, Corollary \ref{p2} and \cite[Corollary 1.3]{Mo} respectively. 

In line 7, one has $G_{0} \cong SL_{2}(13)$ by \cite[Section 8]{BHRD}. The first case is excluded since $r=56$, but $G_{0}$ does not have a transitive permutation representation of degree $56$ by \cite{At}. In the second and the fourth case the group $G_{0}$ must have a transitive permutation representation of degree $r/k=8$ or $26$, respectively, by Lemma \ref{cici}, but this is impossible by \cite{At}. Finally, the third case in line 7 of Table XX corresponds to (1.g).

In line 8, one has $PSL_{2}(13) \unlhd G_{0} \leq M<G_{2}(4):Z_{2}<\Gamma Sp_{6}(4)$, where $M \cong (Z_{3} \times PSL_{2}(13)):Z_{2}$, by \cite{At}. Now, the first case is excluded by \cite[Theorem 1.3]{ORR}, the second by Corollary \ref{p2}. Note that $M$ has one orbit of length $273$ or $546$ on $V_{6}(4)^{\ast}$ by \cite[p. 52]{Lieb0}, and this is a union of $G_{0}$-orbits. Therefore $r/2$ divides $273$ or $546$, respectively, by Lemma \ref{sudbina}(1), and hence $(r,k)\neq (1638,6)$. Thus $(r,k)=(546,16)$, and we obtain (1.h).   

Lines 10, 11 and the first value of $(r,k)$ of line 9 are ruled out by \cite[Corollary 1.3]{Mo}, Line 12 and the second case of line 9 cannot occur by Corollary \ref{p2}. Finally, line 13 leads to (1.h) by \cite{ModAt} and \cite[Tables 8.31--8.34]{BHRD}.
\end{proof}

\bigskip
We analyze the cases of Lemma \ref{linearcross} separately in the same order they appear in lemma' statement except for case (1.b) which will be last to be investigated, as it requires a slightly deeper analysis.  
\bigskip

\begin{lemma}\label{UnoC}
Case (1.c) cannot occur.
\end{lemma}

\begin{proof}
Assume that $q=11$. Then $V=V_{3}(11)$. Clearly, $k\mid vr$ rules out $(70,39)$. Thus $(r,k)=(140,20)$, and hence $Z_{5} \times PSL_{2}(7) \unlhd G_{0} \leq Z_{10} \times PSL_{2}(7)$ by \cite[Tables 8.3--8.7]{BHRD}. Denote by $H$ the group $G_{0}^{(\infty)}$. Then $H <SL_{3}(11)$ and both $H$ and $SL_{3}(11)$ have unique conjugacy class of involution. Further, the conjugacy $SL_{3}(11)$-classes of Klein subgroups and of groups isomorphic to $PSL_{2}(7)$ are unique by \cite[Tables 8.3--8.4]{BHRD}. Thus, we may assume that $H$ contains the Klein subgroup $K=\left\langle \alpha, \beta \right \rangle$, where $\alpha$ and $\beta$ are represented by $Diag(-1,1,1)$ and $Diag(1,1,-1)$, respectively. Thus, for any non-zero vector of $V$, if $H_{\left\langle w \right \rangle}$ is of even order, there is an involution in $H_{\left\langle w \right \rangle}$ acting as $-1$ on $\left\langle w \right \rangle$. 

The group $H$ acts faithfully on $PG_{2}(11)$. Let $S$ be a Sylow $2$-subgroup of $H$ containig $K$. Then $S \cong D_{8}$ fixes a unique point $\left\langle x_{1} \right \rangle$ of $PG_{2}(11)$, where $x_{1}=(0,1,0)$.  Moreover, any $K$ consists of homologies of $PG_{2}(11)$ in a triangular configuration and the vertices are permuted transitively by $S_{3}$ (see \cite{Mi}). Thus $H_{\left\langle x_{1} \right \rangle}=S$, and hence $\left\langle x_{1} \right \rangle^{H}$ has length $21$. Now, let $U$ be any subgroup of order $3$ of $H$. Then $U$ fixes a unique point $\left\langle x_{2} \right \rangle$ of $PG_{2}(11)$, which is then preserved by $N_{H}(S) \cong S_{3}$. If $H_{\left\langle x_{2} \right \rangle}\neq S_{3}$, then $H_{\left\langle x_{2} \right \rangle}\cong S_{4}$ and hence $\left\langle y \right \rangle \in \left\langle x \right \rangle^{H}$, a contradiction. Therefore, the length of $\left\langle x_{2} \right \rangle^{H}$ is $28$. Finally, if $\left\langle x_{3} \right \rangle$ is any point of the axis of an involutory homology and not equal to the any of the vertices of a Klein invariant triangle described above, then the length $\left\langle x_{3} \right \rangle^{H}$ is $84$. Thus, the orbits $\left\langle x_{i} \right \rangle^{H}$, $i=1,2,3$ provide a partition of the point set of $PG_{2}(11)$. Since the stabilizer of $H_{\left\langle x_{i} \right \rangle}$ is of even order for each $i=1,2,3$, it follows that there is an involution in $H_{\left\langle x_{i} \right \rangle}$ acting as $-1$ on $\left\langle x_{i} \right \rangle$. Thus $x_{i}^{G_{0}}$, $i=1,2,3$ partition $V^{\ast}$ and have length $210,280$ and $840$ respectively.

There is a block $C$ of $\mathcal{D}$ such that $Z_{5} \unlhd Z(GL_{3}(11)) \cap G_{C}$ since $k=20$ and $Z_{5} \times PSL_{2}(7) \unlhd G_{0} \leq Z_{10} \times PSL_{2}(7)$. Thus $G_{C}=G_{0,C}$ with $0 \notin C$. Hence, if $B$ is any block containing $0$, then there is a unique point $y$ of $\mathcal{D}$ such that $y \notin B$ such that $G_{B}=G_{y,B}$ since $G$ acts block-transitively on $\mathcal{D}$. Then $\left\vert y^{G_{0}} \right \vert \leq \left\vert B^{G_{0}} \right \vert=r=140$, a contradiction.
\end{proof}

\begin{lemma}\label{UnoD}
Case (1.d) cannot occur.
\end{lemma}

\begin{proof}
Assume that $(r,k)=(252,125)$ or $(504,63)$. Note that $M$ partitions on the
set of $1$-dimensional subspaces of $V$ into two orbits of length $126$ and $%
525$, namely the set of isotropic $1$-dimensional subsapces of $V$ and the
non isotropic ones. Let $S$ be any Sylow $3$-subgroup of $M$. Then there is
a non isotropic vector $x$ of $V$ such that $S_{x}\cong Z_{3}$ since $%
\left\vert S\right\vert =27$ and $S$ contains the center of $M$. From \cite%
{At} we know that there are exactly two conjugacy $M$-classes of subgroups
of order $3$: one consiting of the center of $M$, the other containing all
the other subgroups of $M$ of order $3$. Then any subgroup of order $3$ in
the second class fixes a non non-isotropic vector of $V$. Let $U$ be a sylow 
$3$-subgroup of $G_{0}^{(\infty )}$, then $U$ fixes non-isotropic vector $y$
of $V$ since $\ G_{0}^{(\infty )}\cong PSL_{2}(7)<M$. Then $\left\vert
y^{G_{0}}\right\vert $ is not divisible by $9$ since $G_{0}^{(\infty
)}\trianglelefteq G_{0}<\left( Z_{8}\times G_{0}^{(\infty )}\right) :Z_{2}$.
\ Then $r/2$ does not divide $\left\vert y^{G_{0}}\right\vert $ since $r/2$
is divisible by $9$, but this contradicts Lemma \ref{sudbina}(2). Thus, this
case is excluded.
\end{proof}

\begin{lemma}\label{UnoE}
Case (1.e) cannot occur.
\end{lemma}
\begin{proof}
By \cite{At,ModAt}, there are conjugacy two $\Omega $-classes of subgroups
isomorphic to $PSL_{2}(7)$ acting irreducibly on $V$, say $\mathcal{K}_{1}$
and $\mathcal{K}_{2}$. Any group in $\mathcal{K}_{1}$ lies in a maximal $%
\mathcal{C}_{3}$-subgroup of $\Omega _{6}^{-}(3)$, isomorphic to $\frac{1}{2}%
GU_{3}(3)$, and hence its acts on $V$ regarded as $3$-dimensional $GF(9)$%
-space, whereas any group in $\mathcal{K}_{2}$ acts absolutely irreducibly
on $V$. Then $G_{0}^{\left( \infty \right) }\in \mathcal{K}_{1}$ since $%
G_{0}^{\left( \infty \right) }$ acts absolutely irreducibly on $V$ by ous
assumptions, and from \cite{At} again, we have $PSL_{2}(7)\trianglelefteq
G_{0}\leq PSL_{2}(7):Z$ with $Z$ a cyclic group of order $4$ and whose
involution is $-1$.

Let $W$ be a Sylow $3$-subgroup of $G_{0}$. Then $W\leq F\leq G_{0}$ with $%
F\cong F_{21}$. Clearly, $F$ is contained in a suitable member of $\mathcal{K%
}_{1}$, say $Q$. Let $M\cong \frac{1}{2}GU_{3}(3)$ containing $Q$. If $S$ is
any Sylow $3$-subgroup of $M$ containing $W$, it follows that $W\cap Z(S)=1$
by \cite[p. 14]{At}, hence $W$ fixes exactly on $1$-dimensional $GF(9)$%
-subspace $V$ pointwise, which is isotropic with respect the the hermitian
form preserved by $M^{\prime }$, by \cite{Hup}. Then $Fix(W)$ is a $2$%
-dimensional subspace of $V$ (regarded over $GF(3)$) and  is totally
singular with respect the quadratic form preserved by $\Omega $ (see \cite[%
Theorem 2(ii)]{Dye}).

If any two distinct elements of $Fix(W)^{G_{0}}$ pairwise intersects only in 
$0$, then the set of the projections on $PG_{5}(3)$ of the elements $%
Fix(W)^{G_{0}}$ produces a partition in lines of an elliptic quadric of $%
PG_{5}(3)$. Then $G_{0}^{(\infty )}\leq \frac{1}{2}GU_{3}(3)$ by \cite[%
Theorem 2(ii)]{Dye}, and hence $G_{0}^{(\infty )}\in \mathcal{K}_{1}$, a
contradiction. Therefore, there is $g\in G_{0}$ such that $W^{g}\neq W$ and  
$x\in Fix(W^{g})\cap Fix(W)$ fro some $x\in V$, $x\neq 0$. Then $%
\left\langle W,R\right\rangle \leq G_{0,x}$, and hence $A_{4}\leq G_{0,x}$
by \cite[p. 14]{At} since the Sylow $7$-subgroup acts irreducibly on $V$ by 
\cite{He}. Thus $\left\vert x^{G_{0}}\right\vert $ divides $28$.

It follows from Lemma \ref{cici} that $r/2$ divides $\left\vert
x^{G_{0}}\right\vert $ and hence $28$. Thus $(r,k)\neq (112,14)$ and hence $%
(r,k)=(56,27)$. Then the blocks of $\mathcal{D}$ containing $0$ are $3$%
-dimensional subspaces of $V$ by Lemma \ref{cici}. Now, there is a block $B$
of $\mathcal{D}$ containing $0$ preserved by $W$ since $r=56$, and hence $%
\dim \left( Fix(W)\cap B\right) \geq 1$. If $B^{\prime }$ is any other block
of of $\mathcal{D}$ preserved by $W$, then there is $\varphi \in N_{G_{0}}(W)
$ such that $B^{\varphi }=B^{\prime }$ since $G$ acts flag-transitively on $%
\mathcal{D}$ and $W$ is Sylow $3$-subgroup of $G_{0,B}$ and of $%
G_{0,B^{\prime }}$. Hence, the number of blocks of $\mathcal{D}$ containg $0$
and fixed by $W$ is $\left\vert N_{G_{0}}(W):N_{G_{0,B}}(W)\right\vert $.

There are exacty $2$ blocks of $\mathcal{D}$ containing each of the four $1$%
-dimensional subspaces of $Fix(W)$ including $B$ since $\lambda =2$. Hence,
the number of blocks of $\mathcal{D}$ containg $0$ and fixed by $W$ is $8$
or $2$ according as $\dim \left( Fix(W)\cap B\right) =1$ or $Fix(W)\subset B$%
. Then $\left\vert N_{G_{0}}(W):N_{G_{0,B}}(W)\right\vert $ is $8$ or $2$,
respectvely. If $G_{0}\cong PSL_{2}(7)$, then $\left\vert
N_{G_{0}}(W):N_{G_{0,B}}(W)\right\vert \leq 4$ and this forces $%
Fix(W)\subset B$. If $G_{0}\not\cong PSL_{2}(7)$ then $-1\in G_{0}$ and
hence $-1\in G_{0,B}$ since $B$ is a subsapce of $V$, and hence $\left\vert
N_{G_{0}}(W):N_{G_{0,B}}(W)\right\vert \leq 4$ and $Fix(W)\subset B$. Thus, $%
Fix(W)\subset B$ in each case.

Let $g\in G_{0}$ such that $W^{g}\neq W$ and $Fix(W^{g})\cap Fix(W)\neq
\left\{ 0\right\} $ defined above. Then $B\cap B^{g}\neq \left\{ 0\right\} $%
, and hence $W^{g}$ preserves $B$ since $\lambda =2$. Then $A_{4}\leq
\left\langle W,W^{g}\right\rangle \leq \left( G_{0}^{(\infty )}\right) _{B}$
Then either $A_{4}\leq G_{0,B}$ or $\left\langle -1\right\rangle \times
A_{4}\leq G_{0,B}$ according as $-1$ lies or not in $G_{0}$, respectively.
So $r=\left\vert B^{G_{0}}\right\vert \leq 28$ in both cases, a
contradiction. This completes the proof.
\end{proof}

\begin{lemma}\label{UnoG}
Case (1.g) cannot occur.
\end{lemma}

\begin{proof}
Let $B$ any block of $\mathcal{D}$ containing $0$. Then $B$ is $2$%
-dimensional subsapce of $V$ by Corollary \ref{p2} since $r=182$ and $k=9$.
Further, $G_{0,B}$ is isomorphic either to $Z_{12}$ or to $D_{12}$ since $%
G_{0}\cong SL_{2}(13)$ by \cite{At}. In each case $G_{0,B}$ contains a
normal cyclic subgroup of order $3$, say $U$. Now, $Fix(U)$ is a $2$%
-dimensional subsapce of $V$ and is one of the flag-transitive linear spaces
on $3^{6}$ points, with lines of size $3^{2}$, discovered by Hering \cite%
{He1}. Then $B\neq Fix(U)\,$  since $\lambda =2$. Thus $U$ preserves exactly
one $1$-dimensional subspace of $B$. On the other hand, $U$ preserves each
of the $2$ blocks of $\mathcal{D}$ incident with any of the four $1$%
-dimensional subspaces of $Fix(U)$. Thus $U$ preserves exactly $8$ elements
of $B^{G_{0}}$. However, this is impossible since the number of elements of $%
B^{G_{0}}$ preserved by $U$ is the equal to the ratio $\frac{\left\vert
N_{G_{0}}(U)\right\vert }{\left\vert G_{0,B}\right\vert }$, which is $2$
since $N_{G_{0}}(U)\cong Z_{2}.D_{12}$ and $U\trianglelefteq G_{0,B}$. 
\end{proof}

\begin{lemma}\label{UnoH}
Case (1.h) cannot occur.
\end{lemma}

\begin{proof}
Let $C=\left\langle B\right\rangle _{GF(4)}$ and $c=\dim _{GF(4)}C$. Then $%
\dim _{GF(4)}C\leq \dim _{GF(2)}B=4$ and $\left\vert C\right\vert \geq
\left\vert B\right\vert =2^{4}$ imply $c=2,3$ or $4$. Further, $%
r=m\left\vert C^{G_{0}}\right\vert $ where $m$ is the constant number of
elements in $B^{G_{0}}$ contained in any element of $C^{G_{0}}$. For any
non-zero vector $y$ of $V$, there is at least an element of $B^{\prime }\in
B^{G_{0}}$ such that $y\in B^{\prime }$ and hence $\left\langle
y\right\rangle _{GF(4)}\subset \left\langle B^{\prime }\right\rangle _{GF(4)}
$ with $\left\langle B^{\prime }\right\rangle _{GF(4)}\in C^{G_{0}}$. Now,
counting the pairs $\left( X,\left\langle x\right\rangle _{GF(4)}\right) $
with $X\in C^{G_{0}}$ and $\left\langle x\right\rangle _{GF(4)}\subset X$ we
see that $\frac{4^{c}-1}{4-1}\frac{r}{m}=\frac{4^{6}-1}{4-1}t$, where $t$
denotes the number of elements $C^{G_{0}}$ containing a fixes $1$-dimensional
subspace of $V$. Then $c\neq 3$ since $r=2\cdot 3\cdot 7\cdot 13$. If $c=4$,
then $17\frac{r}{m}=t\frac{r}{2}$, and hence $m=1,2$. Thus either $m=1$ and $%
t=34$, or $m=2$ and $t=17$, and in both cases then $(V,C^{G})$ is
flag-transitive $2$-$(4^{6},4^{3},\lambda ^{\prime })$ design with $\lambda
^{\prime }=\frac{34}{m}$ and $r^{\prime }=\allowbreak 65\frac{34}{m}$. Now,
for any non-zero vector $x$ of $V$, the incidence structure $%
(x^{G_{0}},C^{G})$ is a $1$-design by \cite[1.2.6]{Demb}, hence $\frac{%
r^{\prime }}{(r^{\prime },\lambda ^{\prime })}\mid \left\vert
x^{G_{0}}\right\vert $. One $(Z_{3}\times PSL_{2}(13)):Z_{2}$-orbit on $V^{\ast }$ is of length $273$ or $546$ on $V^{\ast }$ by \cite[p.502]{Lieb0}, and
this, in turn, is a union of $G_{0}$-orbits. So, $\frac{r^{\prime }}{%
(r^{\prime },\lambda ^{\prime })}$ divides $273$ or $546$, a contradiction.
Thus $c=2$, and hence $B$ is $2$-dimensional $GF(4)$-subspace of $V$.
Therefore, $Z\cong Z_{3}$, where $Z=Z(GL_{6}(4))$, is an automorphism group
of $\mathcal{D}$. Hence we may assume that  $Z\leq G_{0}$. Then $G_{0}$
contains an elementary abelian group $E$ of order $9$, and hence $E_{z}\cong
Z_{3}$ for some nonzero vector $z$ of $V$. Clearly $E_{z}\cap Z=1$ and $E_{z}
$ preserves the $2$ blocks of $\mathcal{D}$ containing $0$ and $z$. These
are also preserved by $Z$ being these $2$-dimensional $GF(4)$-subspaces of $V
$. Then $E\leq G_{0,B_{0}}$, where $B_{0}$ is any of the two blocks of $%
\mathcal{D}$ containing $0$ and $z$, and hence $r=\left\vert
B^{G_{0}}\right\vert =\left\vert B_{0}^{G_{0}}\right\vert $ is coprime to $3$%
, a contradiction.
\end{proof}

\begin{lemma}\label{Due}
Case (2) cannot occur.
\end{lemma}
\begin{proof}
Finally, in line 15 one has $Z_{2}.PSL_{3}(4) \unlhd G_{0} \leq Z_{2}.PSL_{3}(4).Z_{2}<GO_{6}^{-}(3)$ by \cite{ModAt} and \cite[Tables 8.31--8.34]{BHRD}.
If $(r,k)= (112,14)$, then $G_{0}$ contains a subgroup $J$ isomorphic to $G_{B}$ and such that $\left\vert G_{0}:J\right\vert=r/k=7$ by Lemma \ref{cici}. Then $\left\vert G_{0}^{(\infty )}:G_{0}^{(\infty )}\cap J\right\vert=7$ since $ G_{0}^{(\infty )} \cong Z_{2}.PSL_{3}(4)$ and $\left\vert G_{0}:G_{0}^{(\infty )}\right\vert \leq 2$. However, this is impossible by \cite{At}.

If $(r,k) = (56,27)$, arguing as in the $PSL_{2}(7)$-case, we see that each block containing $0$ contains exactly four elements of the $PGO_{6}^{-}(3)$-invariant elliptic quadric $\mathcal{E}$ of $PG_{5}(3)$. Clearly, $G_{0}^{(\infty )}\cong Z_{2}.PSL_{3}(4)$ and $\left\vert G_{0}^{(\infty )}:(G_{0}^{(\infty )})_{B} \right\vert$ divides $56$. Then either $G_{0}^{(\infty )}=(G_{0}^{(\infty )})_{B}$, or $\left\vert G_{0}^{(\infty )}:(G_{0}^{(\infty )})_{B} \right\vert= 56$ by \cite{At}. The former is ruled out since $r=56$ and $G_{0} \leq Z_{2}.PSL_{3}(4).Z_{2}$, whereas the latter implies that $H=T:G_{0}^{(\infty )}$ acts flag-transitively on $\mathcal{D}$ and $H_{0,B}\cong SL_{2}(9)$. In this case $H_{0,B}$ permutes the four elements of $\mathcal{E}$, and since the minimal primitive permutation representation degree of $SL_{2}(9)$ is six, these are fixed and pointwise by $H_{0,B}$. Let $\left\langle x \right\rangle$ be any of these four elements preserved by $H_{0,B}$. Then $H_{0,B}$ preserves $B$ and a further block of $\mathcal{D}$ containing $\left\langle x \right\rangle$, but this is impossible since $H_{0,B}$ is maximal in $H_{0}$ (see \cite{At}).
\end{proof}.

\begin{lemma}
\label{pomoc}Let $\Gamma \cong P\Gamma U_{3}(3)$, $K$ be a subgroup of $%
\Gamma $ isomorphic to $PGL_{2}(7)$ and $\mathcal{H}$ be a $\Gamma $-invariant
hermitian unital  of order $3$ of $PG_{2}(9)$. Then the following hold:

\begin{enumerate}
\item $K$ and $K^{\prime }$ are the representative of the unique conjugacy $%
\Gamma $-classes of subgroups isomorphic to $PSL_{2}(7)$ or $PGL_{2}(7)$,
resepctively.

\item If $S$ is any Sylow $3$-subgroup of $\Gamma $, and $R$ is any subgroup
of order $3$ of $S$ distinct from $Z(S)$, then $\mathcal{L}_{1}=Z(S)^{\Gamma
}$ and $\mathcal{L}_{2}=R^{\Gamma }$ two conjugacy $\Gamma $-classes of
subgroups of order $3$.  Further, if $U$ is any Sylow $3$-subgroup of $K$,
and $D$ is any Sylow $3$-subgroup of $\Gamma $ containing $U$, then $U\in 
\mathcal{L}_{2}$, $N_{K}(U)\cong D_{12}$ and $N_{\Gamma }(U)=Z(D)N_{K}(U)$.

\item Both $K$ and $K^{\prime }$ act transitively on $\mathcal{H}$, and $%
K_{P}\cong D_{12}$ and $(K^{\prime })_{P}\cong D_{6}$ for any $P$ point of $%
\mathcal{H}$;

\item Let $X$ denote $(K^{\prime })_{P}$ or its Sylow $3$-subgroup. Then
there are exactly three $X$-invariant Baer subplanes of $PG_{2}(9)$
intersecting $\mathcal{H}$ is an Baer subconic of $PG_{2}(9)$.

\item If $\pi _{i}$, $i=1,2,3$ are the $(K^{\prime })_{P}$-invariant Baer
subplanes as in (2), then $\pi _{i}$, $i=1,2,3$ are representatives of
pairwise distinct $K^{\prime }$-orbits, and $K_{\pi _{i}}^{\prime }$ is
isomorphic to $S_{4}$, $S_{4}$ or $D_{6}$ for $i=1,2$ or $3$, respectively.

\item $\pi _{1}^{K}=\pi _{2}^{K}=\pi _{1}^{K^{\prime }}\cup \pi
_{2}^{K^{\prime }}$ and $\pi _{3}^{K}=\pi _{3}^{K^{\prime }}$. Further, $%
\left( \pi _{j}\cap \mathcal{H}\right) ^{K^{\prime }}$ for $j=1,2$ is $%
K^{\prime }$-transitive partition of $\mathcal{H}$ into Baer subconics of $%
PG_{2}(9)$.
\end{enumerate}
\end{lemma}

\begin{proof}
Assertions (1) and (2) follows from \cite{At}, indeed $N_{\Gamma }(U)\cong
D_{6}\times D_{6}$ and $N_{K}(U)\cong D_{12}$. Assertion (3) follows from 
\cite[Lemma 2.8(a.iii)]{KanLib}. Now, let $P$ be any point of $\mathcal{H}$,
then $J_{P}\cong D_{6}$, where $J$ denotes $K^{\prime }$. The Sylow $3$%
-subgroup of $J_{P}$, say $S$, does not lie in the center of a Sylow $3$%
-subgroup of $PSU_{3}(3)$ by \cite{At}. Hence, $S$ fixes $P$ and acts
semiregularly both on $\mathcal{H}\setminus \{P\}$ and on the $9$ lines of $%
PG_{2}(9)$ incident with $P$ and secant to $\mathcal{H}$ (see \cite{Hup}).
Moreover, any involution in $J_{P}$ fixes pointwise one of these $9$ lines
and permutes semiregularly the remaining $8$ ones. \ Thus, the $J_{P}$%
-orbits on $\mathcal{H}\setminus \{P\}$ are three of length $6$ and three of
length $3$, and each of the latter ones together with $\{P\}$ form a
quadrangle. Denote by $\mathcal{C}_{i}$, $i=1,2,3$, such quadrangles. For
each $i=1,2,3$ there is a unique Baer subplane $\pi _{i}\cong PG_{2}(3)$ of $%
PG_{2}(9)$ containing $\mathcal{C}_{i}$ (for instance, see \cite{Cof}), and
each $\mathcal{C}_{i}=\pi _{i}\cap \mathcal{H}$ is an irreducible conic of $%
\pi _{i}$, that is a Baer subconic of $PG_{2}(9)$, by \cite[Theorem 2.9]%
{BarEbe}.

The group $J$ has two conjugacy classes of subgroups isomorphic to $S_{4}$
by \cite{At}, and a representative of each class contains $J_{P}$. Hence,
there are two systems of imprimitivity for $\mathcal{H}$ each of the
consisting of blocks of imprimitivity of size $4$, say $\mathcal{I}_{j}$
with $j=1,2$. Let $\Delta \in \mathcal{I}_{j}$ such $P\in \Delta $. Then $%
J_{P}<J_{\Delta }\cong S_{4}$ and $J_{\Delta }$ acts naturally on $\Delta $.
Thus $\Delta \setminus \{P\}$ is $J_{P}$-orbit of length $3$, and hence $%
\Delta $ is one of the $\mathcal{C}_{i}$. Therefore, we may assume that $%
\mathcal{I}_{j}=\mathcal{C}_{j}^{J}$ for $j=1,2$. Thus, $\pi _{j}^{J}$ is a $%
J$-orbit of Baer subplanes of $PG_{2}(9)$ intersecting $\mathcal{H}$ in a
Baer subconics of $PG_{2}(9)$, and these Baer subconics form a $J$-invariant
partition of $\mathcal{H}$. Moreover, $\pi _{1}^{K}=\pi _{2}^{K}=\pi
_{1}^{J}\cup \pi _{2}^{J}$ since the two conjugacy $J$-classes of subgroups
isomorphic to $S_{4}$ are fused in $K$ by \cite{At}. Finally, $\pi
_{3}^{K}=\pi _{3}^{J}$ and $J_{P}<K_{P}\cong D_{12}$ again by \cite{At}.
\end{proof}

\begin{lemma}\label{no1.b(56,27)}
$(r,k)\neq (56,27)$ in case (1.b)
\end{lemma}

\begin{proof}
From \cite{At} we know that $PSL_{2}(7)\trianglelefteq G_{0}\leq (Z\times
PSL_{2}(7)):Z_{2}<A$, where $Z$ is the center of $GL_{3}(9)$, $A=(Z\times
M):Z_{2}$ and $M\cong SU_{3}(3)$. Let $\varphi $ be the $M$-invariant
hermitian form $\varphi $ on $V$. Since the hermitian forms on $V_{3}(9)$
are all isomoetric, we may assume that $\varphi $ is the one such that of
the set of isotropic $1$-dimensional subspaces of $V$ is the hermitian
unital 
\[
\mathcal{H}:X_{1}X_{3}^{3}+X_{2}^{4}+X_{1}^{3}X_{3}=0.
\]%
Hence, $A=\left( \left\langle \delta \right\rangle \times M\right)
:\left\langle \sigma \right\rangle $, where $\delta $ is represented by $%
Diag(\omega ,\omega ,\omega )$ with $\omega $ a primitive element of $%
GF(9)^{\ast }$, and $\sigma :(X,Y,Z)\longrightarrow (X^{3},Y^{3},Z^{3})$ is
the semilinear map induced by the Frobenius automorphism of $GF(9)$.
Clearly, $A$ induces $\Gamma $ on $PG_{2}(9)$. Further, $G_{0}^{(\infty )}$
acts faithfully on $PG_{2}(9)$ inducing $K^{\prime }$. Let $U=\left\langle
\alpha \right\rangle $ be a Sylow $3$-subgroup of $G_{0}^{(\infty )}$, then $%
N_{G_{0}^{(\infty )}}(U)=U\left\langle \beta \right\rangle \cong D_{6}$. By
(1) and (2) of Lemma \ref{pomoc}, we may assume that $\alpha $ and $\beta $,
respectively represented by the matrices%
\[
\left( 
\begin{array}{ccc}
1 & 1 & 1 \\ 
0 & 1 & -1 \\ 
0 & 0 & 1%
\end{array}%
\right) \text{ and }\left( 
\begin{array}{ccc}
-1 & 0 & 0 \\ 
0 & 1 & 0 \\ 
0 & 0 & -1%
\end{array}%
\right) 
\]%
Hence, $Fix(U)=\left\langle (0,0,1)\right\rangle _{GF(9)}$.

Now, let us focus on $\mathcal{D}$. The blocks of $\mathcal{D}$ containing $0
$ are $3$-dimensional subspaces of $V$ by Lemma \ref{cici}, and there is a
block $B$ of $\mathcal{D}$ containing $0$ preserved by $U$ since $r=56$.
Thus $\dim \left( Fix(U)\cap B\right) \geq 1$, where $Fix(U)$ is regarded as
a $2$-dimensional $GF(3)$-space. Let $\pi _{B}$ the projection of $B$ on $%
PG_{2}(9)$. Then either $\pi _{B}$ is contained in a line $\ell $ of $%
PG_{2}(9)$, or $\pi _{B}\cong PG_{2}(3)$ (a Baer subplane).

\begin{claim}
$\pi _{B}$ is a Baer subplane of $PG_{2}(9)$. In particular, $\pi _{B}\in
\left\{ \pi _{1},\pi _{2},\pi _{3}\right\} $ 
\end{claim}

If $\pi _{B}\subset \ell $, then $\ell $ is unique since $B$ is $3$%
-dimensional $GF(3)$-subspace of $V$, and hence $\ell $ is preserved by $U$.
Then $\ell $ is tangento $\mathcal{H}$. Indeed, $U$ acts semiregularly on
the $9$ lines of $PG_{2}(9)$ incident with $P$ and secant to $\mathcal{H}$
(see \cite[Satz II.10.12]{Hup}) since $U\in \mathcal{L}_{2}$ by Lemma \ref%
{pomoc}(2). Assume that $B$ does not contain $1$-dimensional subsapces of $V$
and let $e$ be the number of $1$-dimensional subsapces of $V$ contained in $%
\ell $ and intersecting $B$ in a $1$-dimensional $GF(3)$-subspace of $V$.
Then $e-1\equiv 0\pmod{3}$ since $S$ acts semiregularly on $\ell
\setminus \{P\}$, hence $k-1=2+2(e-1)$ since $\left\vert B^{\ast }\cap
P\right\vert =2$, a contradiction since $k=27$. Thus $B$ contains a $1$%
-dimensional $GF(9)$-subsapce of $V$. Then $B\subset Fix(U)$ since $B\neq
\ell $ and $U$ acts semiregularly on $\ell \setminus \{P\}$.Then $k=9+2e$
and hence $e=9$. Then $B$ intersects each subsapces of $\ell $ regarded over 
$GF(3)$ in a non-zero vector, thus containg a basis of $\ell $, a
contradiction. Therefore, $\pi _{B}$ is a Baer subplane of $PG_{2}(9)$
containing $P$ and preserved by $U$. Further, $\pi _{B}\cap \mathcal{H}$ in
a subconic of $PG_{2}(9)$ by \cite[Theorem 2.9]{BarEbe} since $\pi _{B}$
contains $P$ and $U$ acts semiregularly on the $9$ lines of $PG_{2}(9)$
incident with $P$ and secant to $\mathcal{H}$. Then $\pi _{B}\in \left\{ \pi
_{1},\pi _{2},\pi _{3}\right\} $ by (4) and (5) of Lemma \ref{pomoc}.

\begin{claim}
The final contradiction.
\end{claim}

It is not difficult to see that the $3$-dimensional $GF(3)$-subsapces of $V$
projecting $\pi _{1},\pi _{2}$ and $\pi _{3}$ are  
\begin{eqnarray*}
W_{1i} &=&\left\langle (-\omega ^{i},0,\omega ^{i}),(-\omega ^{i},-\omega
^{i},0),(-\omega ^{i},\omega ^{i},0)\right\rangle _{GF(3)} \\
W_{2i} &=&\left\langle (\omega ^{i},0,\omega ^{i}),(\omega ^{i},\omega
^{i},-\omega ^{i}),(\omega ^{i},-\omega ^{i},-\omega ^{i})\right\rangle
_{GF(3)} \\
W_{3i} &=&\left\langle (\omega ^{i},0,0),(\omega ^{i},\omega ^{i},\omega
^{i}),(\omega ^{i},-\omega ^{i},\omega ^{i})\right\rangle _{GF(3)},
\end{eqnarray*}%
respectively, where $i\in \left\{ 0,1,2,3\right\} $. Note that Then $%
W_{ji}\cap W_{j^{\prime }i^{\prime }}=\left\{ 0\right\} $ for any $%
j,j^{\prime }\in \left\{ 1,2,3\right\} $ and for any $i,i^{\prime }\in
\left\{ 0,1,2,3\right\} $ with $i\neq i^{\prime }$, and $W_{ji}\cap
W_{j^{\prime }i}=\left\langle (0,0,\omega ^{i})\right\rangle $ for for any $%
j,j^{\prime }\in \left\{ 1,2,3\right\} $ with $j\neq j^{\prime }$. Hence, $%
B=W_{j_{0}i_{0}}$ for some suitable $j_{0}\in \left\{ 1,2,3\right\} $ and $%
i_{0}\in \left\{ 0,1,2,3\right\} $.

Let $B^{\prime }$ be any other block of of $\mathcal{D}$ preserved by $U$.
Then there is $\varphi \in N_{G_{0}}(U)$ such that $B^{\varphi }=B^{\prime }$
since $G$ acts flag-transitively on $\mathcal{D}$ and $U$ is a Sylow $3$%
-subgroup of $G_{0,B}$ and of $G_{0,B^{\prime }}$. Hence, the number of
blocks of $\mathcal{D}$ containg $0$ and fixed by $U$ is $\left\vert
B^{N_{G_{0}}(U)}\right\vert $. On the other hand, there are exactly two
blocks for any of the four $1$-dimensional $GF(3)$-subspace of $Fix(U)$
since $\lambda =2$, and these are preserved by $U$. Thus $U$ fixes exactly $8
$ ements in $B^{G_{0}}$, and hence $B^{N_{G_{0}}(U)}=\left\{
B_{s}:s=0,...,0\right\} $, where $B_{0}=B$.

As we saw, $\dim \left( Fix(U)\cap B_{s}\right) \geq 1$, hence $Fix(U)\cap
B_{s}=\left\langle (0,0,\omega ^{i_{s}})\right\rangle $ for some $i_{s}\in
\left\{ 0,1,2,3\right\} $. Then $B_{s}=W_{j_{s}i_{s}}$ for some $j_{s}\in
\left\{ 1,2,3\right\} $. Then 
\[
\left\{ B_{s}:s=1,...,8\right\} =B^{N_{G_{0}}(U)}=\left\{
W_{1i},W_{2i}:i=0,1,2,3\right\} 
\]%
Lemma \ref{pomoc}(6) since $\lambda =2$. Now, $N_{G_{0}}(U)\leq Q$, where $%
Q=U:\left\langle \delta ,\beta ,\sigma \right\rangle $. Since $%
Q_{W_{11}}=U:\left\langle \delta ^{4},\beta ,\sigma \right\rangle $, it
follows that $\left\vert W_{11}^{N_{G_{0}}(U)}\right\vert \leq \left\vert
W_{11}^{Q}\right\vert =4$, which is a contradiction. This completes the
proof.
\end{proof}

\bigskip

\begin{proof}[Proof of Theorem \ref{cross-char}]
It follows from Lemmas \ref{SoloLin}--\ref{no1.b(56,27)} that ether the assertion follows, or $V=V_{3}(9)$, $G_{0}^{(\infty)} \cong PSL_{2}(7)<SU_{3}(3)$ and $(r,k)=(112,14)$. Assume that the latter occurs. Let $B$ be any block of $\mathcal{D}$. Then 
$T_{B}=1$ since $k$ is coprime to $3$, and hence\ $G_{B}$ is isomorphic to a
subgroup of $G_{0}$ of index $r/k=8$ by Lemma \ref{cici}. Then either $%
F_{21}\vartriangleleft G_{B}\leq \left( Z_{8}\times F_{21}\right)
:\left\langle \sigma \right\rangle $, or $PSL_{2}(7)\trianglelefteq G_{B}$
by \cite{At}. 

Assume that $PSL_{2}(7)\trianglelefteq G_{B}$. Then $\left( G_{B}\right)
^{\prime }\cong PSL_{2}(7)$, and hence $\left( G_{B}\right) ^{\prime
}=\left( G_{B}\right) _{z}^{\prime }$ for a unique $z\in V$ since $\dim
H_{1}(PSL_{2}(7),V)=0$ by \cite[Example 1 at p.9]{KanLib}. Actually, $z\in
V\setminus B$ since $G_{B}$ acts transitively on $B$. We may assume that $z=0
$ since $G$ acts point-transitively on $B$. Then $PSL_{2}(7)$ has one orbit
of length $14$ on $PG_{2}(9)$, which is necessarily contained in $\mathcal{H}
$ since any Sylow $3$-subgroup of $PSL_{2}(7)$ fixes exactly on point in $%
PG_{2}(9)$, and this lie in $\mathcal{H}$. However, this is impossible by
Lemma \ref{pomoc}(3).

Finally, assume that $F_{21}\vartriangleleft G_{B}\leq \left( Z_{8}\times
F_{21}\right) :\left\langle \sigma \right\rangle $. Then $G_{B}$ contains a
normal Sylow $7$-subgroup of $G$, then $G_{B}=G_{B,x}$ for some $x\in V$
since any Sylow $7$-subgroup of $G$ lying in $G_{0}$ acts irreducibly on $V$%
. Actually, $x\in V\setminus B$ since $G_{B}$ acts transitively on $B$. We
may assume that $x=0$ since $G$ acts point-transitively on $B$.

Now, since $G_{B}$ contains a Sylow $3$-subroup of $G_{0}$, we may also
assume that such Sylow $3$-subroup is $U$. Note that, $B$ is a union of two $%
F_{21}$-orbits of length $7$ since $F_{21}\trianglelefteq G_{B}$ and $k=14$.
Thus the $U$ fixes exactly two points on $B$, one in each $F_{21}$-orbits.
Thus, $U<G_{0,B}\leq U\times \left\langle -1,\sigma \right\rangle $ since $%
F_{21}\vartriangleleft G_{B}\leq \left( Z_{8}\times F_{21}\right)
:\left\langle \sigma \right\rangle $. Therefore, $G_{0,B}$ is isomorphic to
one of the groups $\left\langle -1\right\rangle \times U$, $\left\langle
\sigma \right\rangle \times U$, $\left\langle -\sigma \right\rangle \times U$
or $U\times \left\langle -1,\sigma \right\rangle $ and is a subgroup of
index $112$ of $\left\langle -1\right\rangle \times PSL_{2}(7)$, $PGL_{2}(7)$%
, $PGL_{2}(7)$ or $\left\langle -1\right\rangle \times PGL_{2}(7)$,
resepctively, and hence $H=T:H_{0}$, where $H_{0}$ is one of the previous
four groups containing $PSL_{2}(7)$, acts flag-transitively on $\mathcal{D}$%
. Thus, w.l.g. we may assume that $G_{0}=H_{0}$. 

The group $U$ preserves the two blocks of $\mathcal{D}$ containig $0$ and $%
y\in Fix(U)$, $y\neq 0$, say $C$ and $C^{\prime }$. Then $0$ and $y$ are the
unique points of $C,C^{\prime }$ fixed by $U$. Thus $U$ fixes exactly $16$
blocks of $\mathcal{D}$. If $C^{\prime \prime }$ is any other block of of $%
\mathcal{D}$ preserved by $U$, then there is $\varphi \in N_{G_{0}}(U)$ such
that $C^{\varphi }=C^{\prime \prime }$ since $G$ acts flag-transitively on $%
\mathcal{D}$ and $U$ is a Sylow $3$-subgroup of $G_{0,B}$ and of $G_{0,C}$.
Hence, the number of blocks of $\mathcal{D}$ containg $0$ and fixed by $U$
is $\left\vert N_{G_{0}}(U):N_{G_{0,C}}(U)\right\vert $. Therefore, $%
\left\vert N_{G_{0}}(U):N_{G_{0,C}}(U)\right\vert =16$. Howerver, this is
impossible since $N_{G_{0}}(U)\leq \left\langle -1\right\rangle \times D_{12}
$ since $G_{0}\leq \left\langle -1\right\rangle \times PGL_{2}(7)$ by our
assumption. This completes the proof.
\end{proof}

\begin{proof}[Proof of Theorem \ref{qsabsirrnosub}]
The assertion follows from Theorems \ref{Alt}, \ref{Spor}, \ref{p-char} and \ref{cross-char} according as $L$ is alternating, sporadic or a group of lie type in characteristic $p$ or different from $p$, respectively.   
\end{proof}

\bigskip

Now, we are in position to prove our main result.

\bigskip

\begin{proof}[Proof of Theorem \ref{main}]
Let $\mathcal{D}$ be a non-trivial $2$-$(v,k,2)$ design admitting a flag-transitive automorphism group $G$. Then either assertions (2)-(4) holds, or $Soc(G)$ is an elementary abelian $p$-group for some prime $p$ acting point-regularly on $\mathcal{D}$ by Theorem \ref{redtheo}. Assume that the latter occurs. Clearly, if $G \leq A\Gamma L_{1}(p^{d})$ assertion (1) follows. Hence, assume that $G \nleq A\Gamma L_{1}(p^{d})$. If $r$ is odd, then $(\mathcal{D},G)$ is as in \cite[Example 7]{MBF} by Lemma \ref{rodd}, and (5) hols in this case. If $r$ is even and $SL_{n}(q) \unlhd G_{0}$, then $(\mathcal{D},G)$ is as in \cite[Example 2.1]{Mo} or in Example \ref{sem1dimjedan} by Lemma \ref{vrata}, which is (5) in this case. If $SL_{n}(q) \nleq G_{0}$, according to \cite{As}, either $G_{0}$ lies in a maximal member of one the geometric classes $\mathcal{C}_{i}$ of $\Gamma L_{n}(q)$, $i=1,...,8$, and hence (5) holds by Theorem \ref{MaxGeom}, or $G_{0}^{(\infty)}$ is a quasisimple group, and its action on $V_{n}(q)$ is absolutely irreducible and not realisable over any proper subfield of $GF(q)$, and hence (5) holds by Theorem \ref{qsabsirrnosub}. This completes the proof.     
\end{proof}

\bigskip

\textbf{Acknowledgements.} The authors are grateful to Prof. Seyed Hassan Alavi (Bu-Ali Sina University) for having suggested them to investigate the problem of classifying the flag-transitive $2$-designs of affine type with $\lambda=2$. 

The first author's work on this paper was done while she was visiting the University of Salento in Italy.
She is grateful to the Department of Mathematics and Physics 'E. De Giorgi' at the  University of Salento for hospitality in 2023, and to the National Natural Science Foundation of China (No:12101120), the Project of Department of Education of Guangdong Province (No:2019KQNCX164) and Foshan University for financial support.

\end{document}